\documentclass[11pt]{article}

\usepackage[margin=1.2in]{geometry}

\usepackage{amssymb,amsmath}
\usepackage{amsthm}
\usepackage{hyperref}
\usepackage{mathrsfs}
\usepackage{textcomp}
\hypersetup{
    colorlinks,%
    citecolor=black,%
    filecolor=black,%
    linkcolor=black,%
    urlcolor=black
}

\usepackage{verbatim}

\usepackage{sectsty}

\makeatletter

\newdimen\bibspace
\renewenvironment{thebibliography}[1]{%
 \section*{\refname %or \bibname if you use ``book'' as the documentclass
       \@mkboth{\MakeUppercase\refname}{\MakeUppercase\refname}}%
     \list{\@biblabel{\@arabic\c@enumiv}}%
          {\settowidth\labelwidth{\@biblabel{#1}}%
           \leftmargin\labelwidth
           \advance\leftmargin\labelsep
           \itemsep\bibspace
           \parsep\z@skip     %
           \@openbib@code
           \usecounter{enumiv}%
           \let\p@enumiv\@empty
           \renewcommand\theenumiv{\@arabic\c@enumiv}}%
     \sloppy\clubpenalty4000\widowpenalty4000%
     \sfcode`\.\@m}
    {\def\@noitemerr
      {\@latex@warning{Empty `thebibliography' environment}}%
     \endlist}

\makeatother

\makeatletter

\newtheorem{thm}{Theorem}[section]
\newtheorem{lem}[thm]{Lemma}
\newtheorem{prop}[thm]{Proposition}
\newtheorem{defn}[thm]{Definition}

\newtheorem{rem}[thm]{Remark}

\def\XXint#1#2#3{{\setbox0=\hbox{$#1{#2#3}{\int}$}
  \vcenter{\hbox{$#2#3$}}\kern-.5\wd0}}

                \newcommand{\lda}{\lambda}
                \newcommand{\pa}{\partial}
\newcommand{\va}{\varepsilon}           \newcommand{\ud}{\mathrm{d}}
\newcommand{\be}{\begin{equation}}      \newcommand{\ee}{\end{equation}}
                 
\newcommand{\Lda}{\Lambda}              
\newcommand{\R}{\mathbb{R}}

\DeclareMathOperator{\dist}{dist}

\begin{document}

\title{\textbf{H\"older regularity for the linearized porous medium equation in bounded domains}
\bigskip}

\author{\medskip  Tianling Jin\footnote{T. Jin was partially supported by NSFC grant 12122120, Hong Kong RGC grants GRF 16302519, GRF 16306320 and GRF 16303822}\quad and \quad
Jingang Xiong\footnote{J. Xiong was is partially supported by  the National Key R\&D Program of China No. 2020YFA0712900, and NSFC grants 12325104 and 12271028.}}

\date{\today}

\maketitle

\begin{abstract}
In this paper, we systematically study weak solutions of a linear singular or degenerate parabolic equation in a mixed divergence form and nondivergence form, which arises from the linearized fast diffusion equation and the linearized porous medium equation with the homogeneous Dirichlet boundary condition. We prove the H\"older regularity of their weak solutions.
\end{abstract}

\section{Introduction} 

Let $\Omega\subset\R^n$, $n\ge 1$, be a smooth bounded open set, and $\omega$ be a smooth function in $\overline\Omega$ comparable to the distance function $d(x):=\dist(x,\partial\Omega)$, that is, $0<\inf_{\Omega}\frac{\omega}{d}\le \sup_{\Omega}\frac{\omega}{d}<\infty$. For example, $\omega$ can be taken as the positive normalized first eigenfunction of $-\Delta$ in $\Omega$ with Dirichlet zero boundary condition. Let 
\begin{equation}\label{eq:rangep}
p>-1
\end{equation}
be a fixed constant throughout the paper unless otherwise stated.
 
In this paper, we would like to study regularity of weak solutions to
\begin{equation} \label{eq:general}
\begin{split}
a\omega^{p} \pa_t u-D_j(a_{ij} D_i u+d_j u)+b_iD_i u+\omega^pcu+c_0u&=\omega^pf+f_0 -D_if_i\quad \mbox{in }\Omega \times(-1,0],\\
u&=0\quad \mbox{on }\partial\Omega \times(-1,0],
\end{split}
\end{equation}
where all $a,a_{ij}, d_j,b_i,c, c_0,f, f_0,f_i$ are functions of $(x,t)$, $D_i=\partial_{x_i}$, and the summation convention is used. Throughout this paper, we always assume the ellipticity condition, that is,  $(a_{ij})$ is a matrix satisfying
\be \label{eq:ellip}
\forall\ (x,t)\in \Omega\times[-1,0],\ \lda\le a(x,t)\le \Lda,  \quad \lda |\xi|^2 \le \sum_{i,j=1}^na_{ij}(x,t)\xi_i\xi_j\le \Lda |\xi|^2 \quad\forall\ \xi\in\R^n,
\ee
where $0<\lda\le \Lda<\infty$. 

The study of the equation \eqref{eq:general} is motivated by the linearized equation of the fast diffusion equations (corresponding to $p>0$ in \eqref{eq:fde}) or slow diffusion equations (corresponding to $-1<p<0$ in \eqref{eq:fde}, which are also called porous medium equations) 
\begin{equation} \label{eq:fde}
\begin{split}
\pa_t v^{p+1}&=\Delta v\quad \mbox{in }\Omega \times(0,\infty),\\
v&=0\quad \mbox{on }\partial\Omega \times(0,\infty).
\end{split}
\end{equation}
From DiBenedetto-Kwong-Vespri \cite{DKV}, we know that the solution $v$ of \eqref{eq:fde} with $p>0$ satisfies the global Harnack inequality
\begin{equation}\label{eq:boundarybehavior}
0<\inf_{\Omega}\frac{v(t,x)}{d(x)}\le \sup_{\Omega}\frac{v(t,x)}{d(x)}<\infty
\end{equation} 
before its extinction time. See Bonforte-Figalli \cite{BFi} for a survey. From Aronson-Peletier \cite{AP}, we also know that the solution $v$ of \eqref{eq:fde} with $-1<p<0$ satisfies \eqref{eq:boundarybehavior} as well after certain waiting time. Therefore, the linearized equation of \eqref{eq:fde}, which plays an important role in proving optimal regularity of solutions to \eqref{eq:fde} in \cite{JX19, JX22, JRX}, falls into a form of the equation \eqref{eq:general}. In our earlier work \cite{JX19}, we have obtained many properties for equations like \eqref{eq:general} with $p>0$, such as well-posedness, local boundedness and Schauder estimates. In this paper, we study  the equation \eqref{eq:general} in a more general and systematic way. The main goal of this paper is the H\"older regularity of its weak solutions to \eqref{eq:general} up to the boundary $\{x_n=0\}$.

After the De Giorgi-Nash-Moser theory on the H\"older regularity for uniformly elliptic and uniformly parabolic equations, there have been many investigations on regularity for degenerate or singular elliptic and parabolic equations. By the work of Fabes-Kenig-Serapioni \cite{FKS}, we still have H\"older regularity for elliptic equations whose coefficients are of $A_2$ weight. See also earlier work of Kruzkov \cite{Kruzkov}, Murthy-Stampacchia \cite{MS}, Trudinger \cite{Trudinger1, Trudinger2}, as well as recent work Sire-Terracini-Vita \cite{STV1,STV2} and Wang-Wang-Yin-Zhou \cite{WWYZ}, on degenerate elliptic equations. However, Chiarenza-Serapioni \cite{CS} provided several counterexamples showing that the aforementioned elliptic results do not carry over directly to the parabolic case. Nevertheless, H\"older regularity and Harnack inequality for degenerate or singular parabolic equations with various conditions and structures have  been obtained in, e.g., Chiarenza-Serapioni \cite{CS2,CS3} and Guti\'errez-Wheeden \cite{GW0,GW}, with either the same weight or different weights of singular/degenerate coefficients of $u_t$ and $D^2u$. Recently, in a series of papers \cite{DP21-1,DP21-2,DP21-3,DP21-4}, Dong-Phan obtained results on the wellposedness and regularity estimates in weighted Sobolev spaces for parabolic equations with singular-degenerate coefficients, where the weights of singular/degenerate coefficients of $u_t$ and $D^2u$ appeared in a balanced way. Such Sobolev regularity was obtained later in Dong-Phan-Tran \cite{DPT} for equations similar to our equation \eqref{eq:general}  for $-2<p<0$. Note that although our results on the boundedness of the weak solutions hold for $p>-2$ as well, our H\"older regularity results require the assumption \eqref{eq:rangep} that $p>-1$, and thus, $x_n^p$ is locally integrable. The assumption \eqref{eq:rangep} is used in Proposition \ref{prop:weightedpoincare2}, and also in the beginning of Section \ref{subsec:holderregularity} when defining the measure $\mu_p$, that is the natural choice to measure the improvement of the oscillation of the solution. We need the measure $\mu_p$ to be locally finite in this step. H\"older estimates and Schauder estimates for $p=-1$ with a special structure that the coefficients in the drift terms are positive have been studied in Daskalopoulos-Hamilton \cite{DH}, Koch \cite{Koch} and Feehan-Pop \cite{FP}. The literature on regularity theory for degenerate elliptic and parabolic equations is vast, and one can refer to the above papers for more references.

Under the condition \eqref{eq:ellip}, the equation \eqref{eq:general} is uniformly  parabolic (in a mixed divergence and nondivergence form) when $x$ stays away from the boundary $\partial\Omega$. Therefore, to obtain global estimates for \eqref{eq:general}, we need to establish estimates near $\partial\Omega$, that is in $(B_r(x_0)\cap\Omega)\times(-1,0]$, where $x_0\in\partial\Omega,r>0$ and $B_r(x_0)$ is the open ball in $\R^n$ centered at $x_0$ with radius $r$. By the standard flattening the boundary techniques for studying boundary estimates, we only need to consider the equation in the half ball case.

Now we suppose $\Omega$ is a half ball.  For $\bar x=(\bar x', 0)$,  denote $B_R^+(\bar x) =B_R(\bar x)\cap \{(x',x_n):x_n>0\}$,
\[
Q_R^+(\bar x, \bar t)= B_R^+(\bar x)  \times [\bar t-R^2, \bar t], \quad \mathcal{Q}_R^+(\bar x,\bar t)= B_R^+(\bar x) \times [\bar t-R^{p+2}, \bar t].
\] 
For brevity, we drop $(\bar x)$ and $(\bar x,\bar t)$ in the above notations if $\bar x=0$ or $(\bar x,\bar t)=(0,0)$.  

Consider the equation
\be \label{eq:linear-eq}
ax_n^{p} \pa_t u-D_j(a_{ij} D_i u+d_j u)+b_iD_i u+cx_n^pu+c_0 u=x_n^pf+f_0 -D_if_i \quad \mbox{in }Q_1^+
\ee
with partial Dirichlet condition
\be \label{eq:linear-eq-D}
u=0 \quad \mbox{on }\pa' B_1^+\times[-1,0],
\ee
where $$\pa' B_R^+= B_R \cap \{x_n=0\}.$$ 
We also denote $$\pa'' B_R^+=\pa B_R^+\setminus \pa' B_R^+,$$ and 
\[
\partial_{pa} Q_R^+ \mbox{ as the standard parabolic boundary of } Q_R^+. 
\]

 We  establish H\"older regularity estimates for solutions of \eqref{eq:linear-eq} and \eqref{eq:linear-eq-D} up to the boundary $\{x_n=0\}$, that is, in $\overline B_{1/2}^+\times[-1/2,0]$. If it additionally satisfies $u(\cdot,-1)=0$, then we also  establish H\"older regularity up to the initial time, that is, in $\overline B_{1/2}^+\times[-1,0]$. 
 
 \bigskip

Our results are scattered in the following four sections. 
\begin{itemize}
\item In Section \ref{sec:sobolev}, we introduce a corresponding weighted Sobolev space. We prove a weighted parabolic Sobolev inequality in Theorem \ref{thm:weightedsobolev} and Theorem \ref{thm:weightedsobolev2},  and a De Giorgi type isoperimetric inequality in Theorem \ref{thm:degiorgiisoperimetricelliptic}.

\item In Section \ref{sec:weaksolution}, we introduce the definition of weak solutions in Definition \ref{defn:weaksolutionp}, and establish the wellposedness in Theorem \ref{thm:existenceofweaksolution}. 

\item in Section \ref{sec:bound}, we prove the local-in-time boundedness up to  $\{x_n=0\}$ of weak solutions in Theorems \ref{thm:localboundedness},  and space-time global boundedness in Theorem \ref{thm:localboundednessglobal},

\item In Section \ref{sec:holderregularity}, we prove local-in-time H\"older estimates up to  $\{x_n=0\}$ of weak solutions in Theorems \ref{thm:holdernearboundary}, and space-time global H\"older estimates in Theorem \ref{thm:uniformholderglobal}. In the end of the paper, we show the well-posedness of the Cauchy-Dirichlet problem \eqref{eq:general}. 

\end{itemize}

Our proof of the boundedness and H\"older estimates of weak solutions uses the De Giorgi iteration.  The local-in-time boundedness and H\"older estimates for \eqref{eq:linear-eq} with $-1<p<1$ and $a\equiv 1$ but without lower order terms follow from Guti\'errez-Wheeden \cite{GW0,GW}.

\section{Sobolev spaces and inequalities}\label{sec:sobolev}

\subsection{Some weighted Sobolev spaces}

In this section, we will introduce several  Sobolev spaces that will be needed to define and study  weak solutions of \eqref{eq:linear-eq}. Denote
\[
Q^+_{R,T}=B_R^+\times (-T,0].
\]

Let 
\begin{align}
W^{1,1}_2(Q^+_{R,T})&:=\{g\in L^2(Q^+_{R,T}): \pa_tg\in L^2(Q^+_{R,T}), D_i g \in L^2(Q^+_{R,T}),\ i=1,\cdots,n\}\label{eq:standardspace},\\
\|g\|_{W^{1,1}_2(Q^+_{R,T})}&:=\|g\|_{L^2(Q^+_{R,T})}+\|\pa_t g\|_{L^2(Q^+_{R,T})}+ \sum_{i=1}^n\|D_i g\|_{L^2(Q^+_{R,T})}\nonumber
\end{align} 
be the standard Sobolev space with the standard Sobolev norm.  

Let $p>-1$. Let 
\begin{align}
&V^{1,1}_2(Q^+_{R,T}):=\{g\in L^2(Q^+_{R,T}): \partial_tg\in L^2(Q^+_{R,T},x_n^{p}\ud x\ud t), D_i g \in L^2(Q^+_{R,T}), i=1,\cdots,n\}\label{eq:standardspaceweighted},\\
&\|g\|_{V^{1,1}_2(Q^+_{R,T})}:=\|g\|_{L^2(Q^+_{R,T})}+\|\pa_t g\|_{L^2(Q^+_{R,T},x_n^{p}\ud x\ud t)}+ \sum_{i=1}^n\|D_i g\|_{L^2(Q^+_{R,T})}\nonumber
\end{align} 
be a weighted Sobolev space, with the weight $x_n^p$ only applied on $\pa_t g$. Let
\begin{align}
V_2(Q^+_{R,T}) &:=L^\infty ((-T,0]; L^2(B_R^+,x_n^{p}\ud x)) \cap  L^2((-T,0];H^1(B_R^+)), \label{eq:weightedspaceinfinity}\\
  \|u\|_{V_2(Q^+_{R,T})}&:= \left(\sup_{-T<t<0} \int_{B_R^+ }u^2 x_n^{p}\,\ud x +  \|\nabla u\|_{L^2(B_R^+ \times(-T,0])} ^2\right)^{1/2}\label{eq:V2norm},
\end{align}
and 
\begin{equation}\label{eq:weightedspaceC}
V_2^{1,0}(Q^+_{R,T}) =C ([-T,0]; L^2(B_R^+,x_n^{p}\ud x)) \cap  L^2((-T,0];H^1(B_R^+)) 
\end{equation}
be a subspace of $V_2(Q^+_{R,T})$ endowed with the norm \eqref{eq:V2norm}. 

Then all of $W^{1,1}_2(Q^+_{R,T})$, $V^{1,1}_2(Q^+_{R,T})$,   $V_2^{1,0}(Q^+_{R,T})$ and $V_2(Q^+_{R,T}) $ are Banach spaces. If $p\ge 0$, then
\[
W^{1,1}_2(Q^+_{R,T})\subset V^{1,1}_2(Q^+_{R,T})\subset V_2^{1,0}(Q^+_{R,T})\subset V_2(Q^+_{R,T}).
\]
If $-1<p<0$, then
\[
V^{1,1}_2(Q^+_{R,T})\subset W^{1,1}_2(Q^+_{R,T}),\quad  V^{1,1}_2(Q^+_{R,T})\subset V_2^{1,0}(Q^+_{R,T})\subset V_2(Q^+_{R,T}).
\]
In fact, $V_2^{1,0}(Q^+_{R,T})$ is the closure of $V^{1,1}_2(Q^+_{R,T})$ under the norm $  \|\cdot \|_{V_2(Q^+_{R,T})}$.

We also denote $$\mathring W^{1,1}_2(Q^+_{R,T}),\ \mathring V^{1,1}_2(Q^+_{R,T}),\ \mathring V_2^{1,0}(Q^+_{R,T}),\ \mathring V_2(Q^+_{R,T})$$ as the set of functions in 
\[
W^{1,1}_2(Q^+_{R,T}), \  V^{1,1}_2(Q^+_{R,T}),\  V_2^{1,0}(Q^+_{R,T}),\  V_2(Q^+_{R,T}) \mbox{ vanishing a.e. on } \partial B_R^+\times[-T,0]
\] 
in the trace sense, respectively.

\begin{lem}\label{lem:sobolevdense}
For $p>0$, $\mathring W^{1,1}_2(Q^+_{R,T})$ is dense in $\mathring V^{1,1}_2(Q^+_{R,T})$.
\end{lem}
\begin{proof}
For $\varphi\in \mathring V^{1,1}_2(Q^+_{R,T})$ and $\va>0$, let
\[
\varphi_\va(x,t):=e^{-\va /x_n}\varphi(x,t).
\]
Then
\[
\partial_t \varphi_\va=e^{-\va /x_n}\partial_t\varphi, \quad D_i \varphi_\va=e^{-\va /x_n}D_i\varphi,  \ i=1,\cdots,n-1;
\]
and
\[
D_{n} \varphi_\va=e^{-\va /x_n}D_n\varphi+\frac{\va e^{-\va /x_n}}{x_n}\frac{\varphi}{x_n}.
\]
Hence, $\varphi_\va \in \mathring W^{1,1}_2(Q_1^+)$. By Hardy's inequality, we have
\[
\int_{Q_{R,T}^+} \frac{\varphi^2}{x_n^2}\,\ud x\ud t\le C \int_{Q_{R,T}^+} |\nabla\varphi|^2\,\ud x\ud t. 
\] 
Therefore, it follows from Lebesgue's dominated convergence theorem that  $\|\varphi_\va-\varphi\|_{V^{1,1}_2(Q_1^+)}\to 0$ as $\va\to 0^+$.
\end{proof}
This density fact will be used for the  existence of weak solutions to \eqref{eq:linear-eq} (see Theorem \ref{thm:existenceofweaksolution}).

\begin{lem}\label{lem:cutoffconstantspace}
Let $u\in V_2^{1,0}(Q^+_{R,T})$. Then for every $k\in\R$,
\[
(u-k)^+:=\max(u-k,0)\in V_2^{1,0}(Q^+_{R,T}).
\]
\end{lem}
\begin{proof}
It is clear that $(u-k)^+\in V_2(Q^+_{R,T})$.  For two real numbers $r_1$ and $r_2$, we have the pointwise estimate
\begin{equation}\label{eq:cutoffconstant}
|(r_1-k)^+-(r_2-k)^+|\le |r_1-r_2|.
\end{equation}
Hence
\[
\|(u-k)^+(\cdot,t+h)-(u-k)^+(\cdot,t)\|_{L^2(B_R^+,x_n^p\ud x)}\le \|u(\cdot,t+h)-u(\cdot,t)\|_{L^2(B_R^+,x_n^p\ud x)}.
\]
Since $u\in C ([-T,0]; L^2(B_R^+,x_n^{p}\ud x))$, then $(u-k)^+\in C ([-T,0]; L^2(B_R^+,x_n^{p}\ud x))$ as well.
\end{proof}

\begin{lem}\label{lem:cutoffconstantconvergence}
Suppose $\{u_j\}\subset V_2^{1,0}(Q^+_{R,T})$ converges to $u$ in $V_2^{1,0}(Q^+_{R,T})$. Then for every $k\in\R$, 
\[
(u_j-k)^+\to (u-k)^+\quad\mbox{in }V_2^{1,0}(Q^+_{R,T})\quad\mbox{as }j\to\infty.
\] 
\end{lem}
\begin{proof}
It follows from \eqref{eq:cutoffconstant}.
\end{proof}

Denote
\[
u_h(x,t)=\frac{1}{h}\int_{t-h}^{t}u(x,s)\,\ud s
\]
as the Steklov average of $u$.
\begin{lem}\label{lem:steklovaverageconvergence}
Let $u\in V_2^{1,0}(Q^+_{R,T})$, and $\delta\in(0,T)$. Then for every $h\in (0,\delta)$, $u_h\in V^{1,1}_2(Q^+_{R,T-\delta})$, and 
\[
u_h\to u\quad\mbox{in }V_2(Q^+_{R,T-\delta})\mbox{ as }h\to 0.
\]
\end{lem}
\begin{proof}
It is straightforward to verify that $u_h\in V^{1,1}_2(Q^+_{R,T-\delta})$. Also, by the Minkowski inequality, we have
\begin{align*}
\|(u_h-u)(\cdot,t)\|_{L^2(B_R^+,x_n^p\ud x)}&\le \frac 1h \int_{t-h}^t\|u(\cdot,s)-u(\cdot,t)\|_{L^2(B_R^+,x_n^p\ud x)}\,\ud s\\
&\le \sup_{t-h\le s\le t}\|u(\cdot,s)-u(\cdot,t)\|_{L^2(B_R^+,x_n^p\ud x)}\\
&\to 0\mbox{ as }h\to 0,
\end{align*}
where $u\in V_2^{1,0}(Q^+_{R,T})$ is used in the last inequality. Similarly,
\begin{align*}
\|D_x u_h-D_x u\|_{L^2(Q^+_{R,T-\delta})}&\le \frac 1h \int_{-h}^0\|D_x u(x,t+s)-D_x u(x,t)\|_{L^2(Q^+_{R,T-\delta})}\,\ud s\\
&\le \sup_{-h\le s\le 0}\|D_x u(x,t+s)-D_xu(x,t)\|_{L^2(Q^+_{R,T-\delta})}\\
&\to 0\mbox{ as }h\to 0,
\end{align*}
where we used the continuity of Lebesgue integrals with respect to translations in the last inequality. 
\end{proof}

\subsection{Sobolev inequalities}
Next, we will prove a Sobolev inequality for functions in $\mathring V_2(Q^+_{R,T})$ (in fact, in a slightly larger space). To accommodate the partial boundary condition \eqref{eq:linear-eq-D}, we define the following space:
\[
H^1_{0,L}(B_R^+)=\{u\in H^1(B_R^+): u\equiv 0\ \mbox{on}\ \partial'B_R^+\}.
\]
Then we have the well-known Hardy inequality.
\begin{lem}[Hardy's inequality]\label{lem:hardyinequality}
For every $u\in H^1_{0,L}(B_R^+)$, there holds
\[
\int_{B_R^+}\frac{u(x)^2}{x_n^2}\,\ud x\le 4 \int_{B_R^+}|\nabla u(x)|^2\,\ud x.
\]
\end{lem}
Consequently, we have
\begin{lem}\label{lem:interpolation}
Let $p>0$. For every $u\in H^1_{0,L}(B_R^+)$ and every $\va>0$, there holds
\[
\int_{B_R^+}u^2\,\ud x\le 4 \va \int_{B_R^+}|\nabla u|^2\,\ud x+\va ^{-\frac{p}{2}} \int_{B_R^+}x_n^pu^2\,\ud x.
\]
\end{lem}
\begin{proof}
We have
\begin{align*}
\int_{B_R^+}u^2\,\ud x&= \int_{B_R^+}x_n^{\frac{2p}{p+2}}u^{\frac{4}{p+2}} x_n^{-\frac{2p}{p+2}}u^{\frac{2p}{p+2}}\,\ud x\\
&\le \left(\int_{B_R^+}x_n^{p}u^{2}\,\ud x  \right)^\frac{2}{p+2}
\left(\int_{B_R^+}x_n^{-2}u^{2} \,\ud x  \right)^\frac{p}{p+2}\\
&\le \va \int_{B_R^+}x_n^{-2}u^{2}\,\ud x +\va ^{-\frac p2}\int_{B_R^+}x_n^{p}u^{2} \,\ud x \\
&\le 4\va \int_{B_R^+}|\nabla u(x)|^2\,\ud x +\va ^{-\frac p2}\int_{B_R^+}x_n^{p}u^{2}\,\ud x,
\end{align*}
where we used H\"older's inequality, Young's inequality and Lemma \ref{lem:hardyinequality}.
\end{proof}

By the usual Sobolev inequality,  Hardy's inequality, H\"older's inequality, and a scaling argument, we have the following Sobolev inequality for functions in $H^1_{0,L}(B_R^+)$.
\begin{lem}[Sobolev's inequality]\label{lem:Sobolevinequalitynot}
There exists $C>0$ depending only on $n$ such that for every $u\in H^1_{0,L}(B_R^+)$, there holds
\begin{align*}
\|u\|_{L^\frac{2n}{n-2}(B_R^+)}&\le C \|\nabla u\|_{L^2(B_R^+)}\quad\mbox{if }n\ge 3,\\
\|u\|_{L^q(B_R^+)}&\le CR^{\frac{n}{q}+\frac{2-n}{2}} \|\nabla u\|_{L^2(B_R^+)}\ \forall\,q>0\quad\mbox{if }n=1, 2.
\end{align*}
\end{lem}

Combining Hardy's inequality and Sobolev's inequality, we have the following Hardy-Sobolev inequality for functions in $H^1_{0,L}(B_R^+)$.

\begin{lem}[Hardy-Sobolev inequality]\label{lem:Hardy-Sobolev} Let $s\in (0,2)$. Then
\begin{equation}\label{eq:hardysobolev}
\Big(\int_{B_R^+}\frac{|u(x)|^{\frac{2(n-s)}{n-2}}}{x_n^s}\Big)^{\frac{n-2}{n-s}}\le C(n)^{\frac{n-2}{n-s}}\int_{B_R^+}|\nabla u|^2\quad\forall\ u\in H_{0,L}^1(B_R^+),
\end{equation}
when $n\ge 3$, and for $s\le r<\infty$,
\begin{equation}\label{eq:hardysobolev-2d}
\Big(\int_{B_R^+}\frac{|u(x)|^{r}}{x_n^s}\Big)^{\frac{2}{r}}\le C(r,s)R^{\frac{2(n-s)}{r}+2-n}\int_{B_R^+}|\nabla u|^2\quad\forall\ u\in H_{0,L}^1(B_R^+),
\end{equation}
when $n=1,2$.
\end{lem} 

\begin{proof}  By scaling, we only need to prove for $R=1$. If $n\ge 3$, using the H\"older inequality,  Hardy inequality and Sobolev inequality,  we have 
\begin{align*}
\int_{B_1^+}\frac{|u(x)|^{\frac{2(n-s)}{n-2}}}{x_n^s}& = \int_{B_1^+}\frac{|u(x)|^{s}}{x_n^s} |u(x)|^{\frac{2n-sn}{n-2}} 
\\&  \le \Big( \int_{B_1^+}\frac{|u(x)|^{2}}{x_n^2} \Big)^{\frac{s}{2}}  \Big( \int_{B_1^+}|u|^{\frac{2n}{n-2}} \Big)^{\frac{2-s}{2}} \\&
\le C(n)\Big( \int_{B_1^+} |\nabla u|^2 \Big)^{\frac{s}{2}} \Big( \int_{B_1^+} |\nabla u|^2 \Big)^{\frac{n}{n-2}\frac{2-s}{2}} \\&= C(n)\Big( \int_{B_1^+} |\nabla u|^2 \Big)^{\frac{n-s}{n-2}}  .
\end{align*}
If $n=1,2$, we have 
\begin{align*}
\int_{B_1^+}\frac{|u(x)|^{r}}{x_n^s} &= \int_{B_1^+}\frac{|u(x)|^{s}}{x_n^s} |u|^{r-s}  \\&
\le \Big( \int_{B_1^+}\frac{|u(x)|^{2}}{x_n^2} \Big)^{\frac{s}{2}}  \Big( \int_{B_1^+}|u|^{\frac{2(r-s)}{2-s}} \Big)^{\frac{2-s}{2}}  \\&
\le C(r,s)  \Big( \int_{B_1^+} |\nabla u|^2 \Big)^{\frac{r}{2}} . 
\end{align*}
Therefore, we complete the proof. 
\end{proof}

The next theorem is a mild generalization of Lemma 2.2 in \cite{JX19}.

\begin{thm}\label{thm:weightedsobolev} 
Let $p\ge 0$. For every $u\in L^\infty ((-T,0]; L^2(B_R^+,x_n^{p}\ud x)) \cap  L^2((-T,0];H^1_{0,L}(B_R^+)) $ (in particular, $u\in  \mathring V_2(Q^+_{R,T})$), we have
 \[
 \left(\int_{Q^+_{R,T}} |u|^{2\chi}\ud x \ud t \right)^{\frac{1}{\chi}} \le C \|u\|_{V_2(Q^+_{R,T})}^2,
 \]
 where  $\chi =\frac{n+p+2}{n+p}$ and $C$ depends only on $n$ and $p$  if $n\ge 3$;  while $\chi=\frac{p+2}{p+1}$ and $C=C(p)R^{\frac{p+2-n}{p+2}}$ with the constant $C(p)$ depending only on $p$ if $n=1,2$.
\end{thm}
\begin{proof} We prove the case $n\ge 3$ first.
Let $s\in(0,2)$ be such that $\frac{s(n-2)}{2-s}=p$. By \eqref{eq:hardysobolev} and the H\"older inequality, we have
\begin{align*}
\int_{B_R^+  }|u|^{\frac{2(n+2-2s)}{n-s} } \,\ud x &=\int_{B_R^+  }|u|^{2} x_n^{-\frac{s(n-2)}{n-s}}|u|^{\frac{2(2-s)}{n-s} } x_n^{\frac{s(n-2)}{n-s}}\,\ud x  \\
&\le \Big(\int_{B_R^+}\frac{|u|^{\frac{2(n-s)}{n-2}}}{x_n^s}\,\ud x\Big)^{\frac{n-2}{n-s}} \Big( \int_{B_R^+ } u^2 x_n^{\frac{s(n-2)}{2-s}}\,\ud x\Big)^{\frac{2-s}{n-s}} \\&
\le C(n,p) \Big(\int_{B_R^+  } |\nabla u|^2\,\ud x\Big)   \Big( \int_{B_R^+ } u^2 x_n^{p}\,\ud x\Big)^{\frac{2-s}{n-s}}.
\end{align*}
Integrating the above inequality in $t$, we have
\begin{align*}
&\Big(\int_{-T}^0\int_{B_R^+  }|u(x,t)|^{\frac{2(n+2-2s)}{n-s}} \,\ud x \ud t\Big)^{\frac{n-s}{n+2-2s}}\\
& \le C(n,p)  \sup_{-T<t<0}\Big( \int_{B_R^+  }u^2 x_n^{p}\,\ud x\Big)^{\frac{2-s}{n+2-2s}} \Big(\int_{B_R^+ \times[-T,0]} |\nabla u|^2 \,\ud x \ud t  \Big)^{\frac{n-s}{n+2-2s}}\\&
\le C(n,p) \Big( \|\nabla u\|_{L^2(B_R^+ \times(-T,0])} ^2 + \sup_{-T<t<0} \int_{B_R^+  }u^2 x_n^{p}\,\ud x \Big),
\end{align*}
where we have used the Young inequality in the last inequality.

If $n=1,2$, using \eqref{eq:hardysobolev-2d} and the H\"older inequality, we have 
\begin{align*}
\int_{B_R^+  }|u|^{2+\frac{2}{p+1}} \,\ud x&= \int_{B_R^+  }|u|^{2} x_n^{-\frac{p}{p+1}} |u|^{\frac{2}{p+1}} x_n^{\frac{p}{p+1}}  \,\ud x \\&
\le \Big( \int_{B_R^+  }  \frac{|u|^{\frac{2(p+1)}{p}}}{x_n}\,\ud x \Big)^{\frac{p}{p+1}} \Big( \int_{B_R^+ } u^2 x_n^{p}\,\ud x\Big)^{\frac{1}{p+1}}  \\&
\le C R^{\frac{p+2-n}{p+1}}\Big( \int_{B_R^+  }  |\nabla u|^2\,\ud x \Big) \Big( \int_{B_R^+ } u^2 x_n^{p}\,\ud x\Big)^{\frac{1}{p+1}} .
\end{align*}
Integrating the above inequality in $t$, we have
\begin{align*}
&\Big(\int_{-T}^0\int_{B_R^+  }|u(x,t)|^{\frac{2(p+2)}{p+1}} \,\ud x \ud t\Big)^{\frac{p+1}{p+2}}\\
& \le C(p)R^{\frac{p+2-n}{p+2}}  \sup_{-T<t<0}\Big( \int_{B_R^+  }u^2 x_n^{p}\,\ud x\Big)^{\frac{1}{p+2}} \Big(\int_{B_R^+ \times[-T,0]} |\nabla u|^2 \,\ud x \ud t  \Big)^{\frac{p+1}{p+2}}\\&
\le C(p) R^{\frac{p+2-n}{p+2}}\Big( \|\nabla u\|_{L^2(B_R^+ \times(-T,0])} ^2 + \sup_{-T<t<0} \int_{B_R^+  }u^2 x_n^{p}\,\ud x \Big),
\end{align*}
where we have used the Young inequality in the last inequality.
\end{proof}
For $-2<p<0$, then we have another parabolic Sobolev inequality.
\begin{thm}\label{thm:weightedsobolev2} 
For every $u\in L^\infty ((-T,0]; L^2(B_R^+,x_n^{p}\ud x)) \cap  L^2((-T,0];H^1_{0,L}(B_R^+)) $ (in particular, $u\in  \mathring V_2(Q^+_{R,T})$), where $-2<p<0$, we have
 \[
 \left(\int_{Q^+_{R,T}} |u|^{2\chi}x_n^p\ud x \ud t \right)^{\frac{1}{\chi}} \le C \|u\|_{V_2(Q^+_{R,T})}^2,
 \]
 where  $\chi =\frac{n+2p+2}{n+p}$ and $C$ depends only on $n$ and $p$  if $n\ge 3$;  while $\chi=\frac{3}{2}$ and $C=C(p)R^{\frac{p+4-n}{3}}$ with the constant $C(p)$ depending only on $p$ if $n=1,2$.
\end{thm}
\begin{proof} We prove the case $n\ge 3$ first. By \eqref{eq:hardysobolev} and the H\"older inequality, we have
\begin{align*}
\int_{B_R^+  }|u|^{\frac{2(n+2+2p)}{n+p} } x_n^p \,\ud x &=\int_{B_R^+  }|u|^{2} x_n^{\frac{p(n-2)}{n+p}}|u|^{\frac{2(2+p)}{n+p} } x_n^{\frac{p(p+2)}{n+p}}\,\ud x  \\
&\le \Big(\int_{B_R^+}\frac{|u|^{\frac{2(n+p)}{n-2}}}{x_n^{-p}}\,\ud x\Big)^{\frac{n-2}{n+p}} \Big( \int_{B_R^+ } u^2 x_n^{p}\,\ud x\Big)^{\frac{2+p}{n+p}} \\&
\le C(n,p) \Big(\int_{B_R^+  } |\nabla u|^2\,\ud x\Big)   \Big( \int_{B_R^+ } u^2 x_n^{p}\,\ud x\Big)^{\frac{2+p}{n+p}}.
\end{align*}
Integrating the above inequality in $t$, we have
\begin{align*}
&\Big(\int_{-T}^0\int_{B_R^+  }|u(x,t)|^{\frac{2(n+2+2p)}{n+p}} x_n^p\,\ud x \ud t\Big)^{\frac{n+p}{n+2+2p}}\\
& \le C(n,p)  \sup_{-T<t<0}\Big( \int_{B_R^+  }u^2 x_n^{p}\,\ud x\Big)^{\frac{2+p}{n+2+2p}} \Big(\int_{B_R^+ \times[-T,0]} |\nabla u|^2 \,\ud x \ud t  \Big)^{\frac{n+p}{n+2+2p}}\\&
\le C(n,p) \Big( \|\nabla u\|_{L^2(B_R^+ \times(-T,0])} ^2 + \sup_{-T<t<0} \int_{B_R^+  }u^2 x_n^{p}\,\ud x \Big),
\end{align*}
where we have used the Young inequality in the last inequality.

If $n=1,2$, using \eqref{eq:hardysobolev-2d} and the H\"older inequality, we have 
\begin{align*}
\int_{B_R^+  }|u|^{3} x_n^p\,\ud x&= \int_{B_R^+  }|u|^{2} x_n^{\frac{p}{2}} |u| x_n^{\frac{p}{2}}  \,\ud x \\&
\le \Big( \int_{B_R^+  }  \frac{|u|^{4}}{x_n^{-p}}\,\ud x \Big)^{\frac{1}{2}} \Big( \int_{B_R^+ } u^2 x_n^{p}\,\ud x\Big)^{\frac{1}{2}}  \\&
\le C R^{\frac{p+4-n}{2}}\Big( \int_{B_R^+  }  |\nabla u|^2\,\ud x \Big) \Big( \int_{B_R^+ } u^2 x_n^{p}\,\ud x\Big)^{\frac{1}{2}} .
\end{align*}
Integrating the above inequality in $t$, we have
\begin{align*}
&\Big(\int_{-T}^0\int_{B_R^+  }|u(x,t)|^{3} x_n^p\,\ud x \ud t\Big)^{\frac{2}{3}}\\
& \le C(p)R^{\frac{p+4-n}{3}} \sup_{-T<t<0}\Big( \int_{B_R^+  }u^2 x_n^{p}\,\ud x\Big)^{\frac{1}{3}} \Big(\int_{B_R^+ \times[-T,0]} |\nabla u|^2 \,\ud x \ud t  \Big)^{\frac{2}{3}}\\&
\le C(p) R^{\frac{p+4-n}{3}} \Big( \|\nabla u\|_{L^2(B_R^+ \times(-T,0])} ^2 + \sup_{-T<t<0} \int_{B_R^+  }u^2 x_n^{p}\,\ud x \Big),
\end{align*}
where we have used the Young inequality in the last inequality. Note that
\[
\frac{p+4-n}{3}>0
\]
if $-2<p<0$ and $n=1,2$.
\end{proof}

Uing the idea of Fabes-Kenig-Serapioni \cite{FKS}, we have the following  Poincar\'e inequality.
\begin{prop}\label{prop:weightedpoincare2}
Let $n\ge 1$, $p>-1$ and $r>0$. Then there exists a constant $C>0$ depending only on $n$ and $p$ such that 
\[
\int_{B_r}|u(x)-(u)_{p,r}| |x_n|^{p}\,\ud x\le C r^{1+p} \int_{B_r}|\nabla u(x)|\,\ud x
\]
for all $u\in H^1(B_r)$, where 
\[
(u)_{p,r}=\frac{\int_{B_r}u(x) |x_n|^{p}\,\ud x}{\int_{B_r} |x_n|^{p}\,\ud x}.
\]
\end{prop}
\begin{proof}
By scaling, we only need to prove it for $r=1$. By a density argument, we only need to show it for Lipschitz continuous (in $\overline B_1$) functions.

Using the triangle inequality and Lemma 1.4 of  Fabes-Kenig-Serapioni \cite{FKS}, we have for all $x\in B_1$ that
\[
\left|u(x)-(u)_{0}\right|\le \frac{1}{|B_1|}\int_{B_1}|u(x)-u(y)|\,\ud y\le C \int_{B_1}\frac{|\nabla u(z)|}{|x-z|^{n-1}}\,\ud z.
\]
where 
\[
(u)_{0}= \frac{1}{|B_1|}\int_{B_1}u(y)\,\ud y.
\]
Then
\[
\int_{B_1} \left|u(x)-(u)_{0}\right| |x_n|^p\,\ud x\le C \int_{B_1}\left(\int_{B_1}\frac{|x_n|^p}{|x-z|^{n-1}} \,\ud x\right) |\nabla u(z)|\,\ud z.
\]
Since 
\begin{align*}
\int_{B_1}\frac{|x_n|^p}{|x-z|^{n-1}} \,\ud x &\le \int_{\{|x_n\le 1, |x'|\le 1|\}}\frac{|x_n|^p}{|x-z|^{n-1}} \,\ud x\\
&\le C\int_{\{|x_n\le 1, |x'|\le 1|\}}\frac{|x_n|^p}{|x|^{n-1}} \,\ud x\\
&=C\int_{-1}^1 \left(\int_{|x'|\le \frac{1}{|x_n|}}\frac{1}{(1+|x'|^2)^{\frac{n-1}{2}}}\,\ud x'\right) |x_n|^p\,\ud x_n\\
&\le C\int_{-1}^1 |\log |x_n||\cdot |x_n|^p\,\ud x_n\\
&\le C,
\end{align*}
where we used $p>-1$ in the last inequality,  we have
\[
\int_{B_1} \left|u(x)-(u)_{0}\right| |x_n|^p\,\ud x\le C \int_{B_1}|\nabla u(z)|\,\ud z.
\]
Then the conclusion follows from the fact that
\begin{align*}
|(u)_{p,1}-(u)_{0}|=\left|\frac{\int_{B_1} (u-(u)_{0}) |x_n|^{p}\,\ud x}{\int_{B_1} |x_n|^{p}\,\ud x} \right|\le C \int_{B_1} |u-(u)_{0}| |x_n|^{p}\,\ud x\le C \int_{B_1}|\nabla u(z)|\,\ud z.
\end{align*}
\end{proof}

The last inequality is a De Giorgi type isoperimetric inequality.

\begin{thm}\label{thm:degiorgiisoperimetricelliptic}
Let $p>-1$, $k<\ell, r>0$ and $u\in H^1(B_r)$. For every
$
0<\va<\min\left(\frac{1}{2},\frac{1}{2(p+1)}\right),
$
there exists a positive constant $C$ depending only on $n,p$ and $\va$ such that
\begin{align*}
&(\ell-k)\int_{\{u\ge l\}\cap B_r}  |x_n|^{p}\,\ud x\int_{\{u\le k\}\cap B_r} |x_n|^{p}\,\ud x\\
& \le C r^{n+2p+1+\frac{n(1-2\va)}{2}-\va p} \left(\int_{\{k<u<\ell\}\cap B_r} |\nabla  u|^2\,\ud x\right)^{1/2} \left(\int_{\{k<u<\ell\}\cap B_r}|x_n|^{p}\,\ud x\right)^{\va}.
\end{align*}
\end{thm}
\begin{proof}
Let
\[
v=\sup(k, \inf(u,\ell))-k,\quad (v)_{p,r}=\frac{\int_{B_r} v(x) |x_n|^{p}\,\ud x}{\int_{B_r} |x_n|^{p}\,\ud x}.
\]
Then by Proposition \ref{prop:weightedpoincare2},
\begin{align*}
\int_{\{v=0\}\cap B_r} (v)_{p,r}  |x_n|^{p}\,\ud x&\le \int_{B_r}|v(x)-(v)_{p,r}| |x_n|^{p}\,\ud x\\
&\le Cr^{1+p} \int_{B_r}|\nabla v(x)| \,\ud x\\
&=Cr^{1+p} \int_{\{k<u<\ell\}\cap B_r}|\nabla u(x)| \,\ud x.
\end{align*}
Using H\"older's inequality, we have
\begin{align*}
&\int_{\{k<u<\ell\}\cap B_r}|\nabla u(x)| \,\ud x \\
& \le C \left(\int_{\{k<u<\ell\}\cap B_r} |\nabla u|^2 \,\ud x\right)^{1/2}  \left(\int_{\{k<u<\ell\}\cap B_r}|x_n|^{p}\,\ud x\right)^{\va} \left(\int_{\{k<u<\ell\}\cap B_r}|x_n|^{-\frac{2p\va}{1-2\va}}\,\ud x\right)^{\frac{1-2\va}{2}}\\
& \le C r^{\frac{n(1-2\va)}{2}-\va p} \left(\int_{\{k<u<\ell\}\cap B_r} |\nabla  u|^2\,\ud x\right)^{1/2} \left(\int_{\{k<u<\ell\}\cap B_r}|x_n|^{p}\,\ud x\right)^{\va},
\end{align*}
where we used
$
0<\va<\min\left(\frac{1}{2},\frac{1}{2(p+1)}\right),
$
so that we can use H\"older's inequality and $|x_n|^{-\frac{2p\va}{p-2\va}}$ is integrable. 

On the other hand, we have
\begin{align*}
\int_{\{v=0\}\cap B_r} (v)_{p,r}  |x_n|^{p}\,\ud x&=\frac{\int_{B_r} v(x) |x_n|^{p}\,\ud x}{\int_{B_r} |x_n|^{p} \,\ud x}\cdot \int_{\{u\le k\}\cap B_r} |x_n|^{p} \,\ud x\\
& \ge \frac{(\ell-k)\int_{\{u\ge l\}\cap B_r}  |x_n|^{p} \,\ud x}{\int_{B_r} |x_n|^{p} \,\ud x}\cdot \int_{\{u\le k\}\cap B_r} |x_n|^{p} \,\ud x\\
&\ge C r^{-n-p} (\ell-k)\int_{\{u\ge l\}\cap B_r}  x_n^{p}\,\ud x\int_{\{u\le k\}\cap B_r} x_n^{p}\,\ud x.
\end{align*}
Hence, the conclusion follows. 
\end{proof}

\section{Weak solutions}\label{sec:weaksolution}

\subsection{Definitions}

%For simplicity, we denote
%\begin{equation}\label{eq:operatorL}
%Lu:=-D_j(a_{ij} D_i u+d_j u)+b_iD_i u+cu.
%\end{equation}
%Therefore, the equation \eqref{eq:linear-eq} can be written as
%\be \label{eq:linear-eq-L}
%ax_n^{p} \pa_t u+Lu=f_0 -D_if_i \quad \mbox{in }Q_1^+
%\ee
Regarding the coefficients of the equation \eqref{eq:linear-eq}, besides \eqref{eq:rangep}, we assume that 
\begin{itemize}
\item there exist $0<\lambda\le\Lambda<\infty$ such that 
\be \label{eq:ellip2}
\lda\le a(x,t)\le \Lda,  \quad \lda |\xi|^2 \le \sum_{i,j=1}^na_{ij}(x,t)\xi_i\xi_j\le \Lda |\xi|^2, \quad\forall\ (x,t)\in Q_1^+,\ \forall\ \xi\in\R^n;
\ee
\item 
\begin{equation}\label{eq:assumptioncoefficient}
\Big\||\pa_t a| + |c|\Big\|_{L^q(Q_1^+,x_n^p\ud x\ud t)} +\left\|\sum_{j=1}^n(b_j^2+d_j^2)+|c_0|\right\|_{L^q(Q_1^+)}\le\Lambda 
\end{equation}
for some $q>\frac{\chi}{\chi-1}$;
\item
\begin{equation}\label{eq:assumptionf}
F_0:=\|f\|_{L^\frac{2\chi}{2\chi-1}(Q_1^+,x_n^p\ud x\ud t)}+\|f_0\|_{L^\frac{2\chi}{2\chi-1}(Q_1^+)}+ \sum_{j=1}^n\|f_j\|_{L^2(Q_1^+)}<\infty,
\end{equation}
\end{itemize}
where $\chi>1$ is the constant in Theorem \ref{thm:weightedsobolev} or Theorem \ref{thm:weightedsobolev2} depending on the value of $p$.

\begin{defn}\label{defn:weaksolutionp}
We say $u$ is a weak solution of \eqref{eq:linear-eq} with the partial boundary condition \eqref{eq:linear-eq-D} if $u\in C ((-1,0]; L^2(B_1^+,x_n^{p}\ud x)) \cap  L^2((-1,0];H_{0,L}^1(B_1^+)) $ and satisfies
\begin{equation}\label{eq:definitionweaksolution}
\begin{split}
&\int_{B^+_1}a(x,s) x_n^{p}  u(x,s) \varphi(x,s)\,\ud x-\int_{-1}^s\int_{B_1^+} x_n^{p}(\varphi\partial_t a+a\partial_t \varphi)u\,\ud x\ud t\\
&\quad+ \int_{-1}^s\int_{B_1^+} \big(a_{ij}D_iuD_j\varphi+d_juD_j\varphi+b_jD_ju\varphi+c x_n^p u \varphi+c_0u\varphi\big)\,\ud x\ud t\\
&=\int_{-1}^s\int_{B_1^+} (x_n^p f\varphi+f_0\varphi+f_jD_j\varphi)\,\ud x\ud t\quad\mbox{a.e. }s\in (-1,0]
\end{split}
\end{equation}
for every $\varphi\in \mathring V^{1,1}_2(Q^+_1)$ satisfying $\varphi (\cdot,-1)\equiv 0$ in $B_1^+$ (in the trace sense).
\end{defn}
Using Theorem \ref{thm:weightedsobolev} and Theorem \ref{thm:weightedsobolev2}, one can verify that  under the assumptions \eqref{eq:rangep}, \eqref{eq:ellip2}, \eqref{eq:assumptioncoefficient} and  \eqref{eq:assumptionf}, each integral in \eqref{eq:definitionweaksolution} is finite. 

\begin{defn}\label{defn:weaksolutionwithinitialtime}
We say that $u$ is a weak solution of \eqref{eq:linear-eq} with the partial boundary condition \eqref{eq:linear-eq-D} and the initial condition $u(\cdot,-1)\equiv 0$, if  $u\in C ([-1,0]; L^2(B_1^+,x_n^{p}\ud x)) \cap  L^2((-1,0];H_{0,L}^1(B_1^+)) $, $u(\cdot,-1)\equiv 0$, and satisfies \eqref{eq:definitionweaksolution} for all $\varphi\in \mathring V^{1,1}_2(Q^+_1)$. 
\end{defn}

\begin{defn}\label{defn:weaksolutionglobal}
We say that $u$ is a weak solution of \eqref{eq:linear-eq} with the full boundary condition $u\equiv 0$ on $\pa_{pa} Q_1^+$, if  $u\in \mathring V_2^{1,0}(Q^+_{1})$, $u(\cdot,-1)\equiv 0$, and satisfies \eqref{eq:definitionweaksolution} for all $\varphi\in \mathring V^{1,1}_2(Q^+_1)$.
\end{defn}

%Note that $C ([-1,0]; L^2(B_1^+))\subset W^{1,1}_2(Q_1^+)$.

\begin{defn}\label{defn:weaksolutionglobalin}
Let $g\in V^{1,1}_2(Q_1^+)$. We say that $u$ is a weak solution of \eqref{eq:linear-eq} with the inhomogeneous boundary condition $u\equiv g$ on $\pa_{pa} Q_1^+$, if  $u\in V_2^{1,0}(Q^+_{1})$, $u= g$ on $\pa_{pa} Q_1^+$, and $v:=u-g$ is a weak solution of
\begin{align*}
&a x_n^{p} \pa_t v -D_j(a_{ij} D_i v+d_j v)+b_iD_i v+cx_n^pv+c_0 v\\
&=x_n^{p} (f-a \pa_t g -cg) + (f_0-b_iD_i g -c_0 g)- D_i(f_i-a_{ij} D_i g-d_j g).
\end{align*}
with homogeneous boundary condition $v\equiv 0$ on $\pa_{pa} Q_1^+$.
\end{defn}

\subsection{Energy estimates, uniqueness and existence}
We start with energy estimates.
\begin{lem}\label{lem:Steklovapproximation}
Suppose $u\in C ([-1,0]; L^2(B_1^+,x_n^{p}\ud x)) \cap  L^2((-1,0];H_{0,L}^1(B_1^+))$ is a weak solution of \eqref{eq:linear-eq} with the partial boundary condition \eqref{eq:linear-eq-D}, where the coefficients of the equation satisfy \eqref{eq:rangep}, \eqref{eq:ellip2},  \eqref{eq:assumptioncoefficient} and \eqref{eq:assumptionf}. 
Let $k\ge \sup_{\pa_{pa}Q_1^+}|u|$ and $\varphi=(u-k)^+$. Then there exists $C>0$ depending only on $n,p,\lambda,\Lambda$ such that
 \begin{equation}\label{eq:uastestfunction}
\begin{split} 
&\int_{B^+_1}x_n^{p}  \varphi(x,s)^2\,\ud x+ \int_{-1}^s\int_{B_1^+} |\nabla\varphi|^2\,\ud x\ud t \\
&\le C\int_{-1}^s\int_{B_1^+} \varphi^2 \Big[(|\partial_t a|+|c|)x_n^p+\sum_{j=1} ^n(d_j^2+b_j^2)+|c_0|\Big]\,\ud x\ud t\\
&\quad + C\int_{-1}^s\int_{B_1^+\cap\{u>k\}}  k^2 \Big(|c|x_n^p+|c_0|+\sum_j d_j^2\Big)\,\ud x\ud t\\
&\quad+C\int_{-1}^s\int_{B_1^+\cap\{u>k\}} \Big(x_n^p f\varphi + f_0\varphi+\sum_{j=1}^n f_j^2\Big)\,\ud x\ud t\quad\mbox{a.e. }s\in (-1,0].
\end{split}
\end{equation}
\end{lem}

\begin{proof}
If $u\in V^{1,1}_2(Q^+_1)$ (cf. \eqref{eq:standardspaceweighted}), then $\varphi\in \mathring V^{1,1}_2(Q^+_1)$ and $\varphi (\cdot,-1)\equiv 0$ in $B_1^+$. Then \eqref{eq:uastestfunction} follows from \eqref{eq:definitionweaksolution}, by using \eqref{eq:ellip2} and H\"older's inequality.

In the following, we will show that we do not need to assume $u\in V^{1,1}_2(Q^+_1)$, and that $u\in C ([-1,0]; L^2(B_1^+,x_n^{p}\ud x)) \cap  L^2((-1,0];H_{0,L}^1(B_1^+))$ would be sufficient.

Denote
\[
u_h(x,t)=\frac{1}{h}\int_{t-h}^{t}u(x,s)\,\ud s
\]
as the Steklov average of $u$. Then for every $v\in V_2^{1,1}(B^+_1\times(-1+h,0))$ such that $v=0$ on $\pa (B^+_1\times(-1+h,0))$, by taking $v_{-h}$ as the test function in \eqref{eq:definitionweaksolution}, we have 
\[%\label{eq:definitionweaksolution4}
\begin{split}
&-\iint_{Q_1^+} x_n^{p}(v_{-h}\partial_t a+a\partial_t v_{-h})u\,\ud x\ud t\\
&+ \iint_{Q_1^+} \big(a_{ij}D_iuD_jv_{-h}+d_j u D_j v_{-h}+b_jD_juv_{-h}+cx_n^puv_{-h}+c_0uv_{-h}\big)\,\ud x\ud t\\
&=\iint_{Q_1^+} (x_n^pfv_{-h}+f_0v_{-h}+f_jD_jv_{-h})\,\ud x\ud t.
\end{split}
\]
By changing the order of the integration, we have
\begin{align}
&\iint_{Q_1^+} \big(a_{ij}D_iuD_jv_{-h}+d_j u D_j v_{-h}+b_jD_juv_{-h}+cx_n^puv_{-h}+c_0uv_{-h}\big)\,\ud x\ud t\nonumber\\
&\quad= \iint_{B_1^+\times(-1+h,0)} \big((a_{ij}D_iu+d_j u)_h D_jv+(b_jD_ju+cx_n^p u+c_0u)_hv\big)\,\ud x\ud t,\label{eq:localmaxapp1}\\
&\iint_{Q_1^+} (x_n^pfv_{-h}+f_0v_{-h}+f_jD_jv_{-h})\,\ud x\ud t\nonumber\\
&\quad=\iint_{B_1^+\times(-1+h,0)} ((x_n^pf+f_0)_hv+(f_j)_hD_jv)\,\ud x\ud t,\label{eq:localmaxapp2}
\end{align}
and 
\begin{align}
&-\iint_{Q_1^+} x_n^{p}(v_{-h}\partial_t a+a\partial_t v_{-h})u\,\ud x\ud t\nonumber\\
&=\iint_{B_1^+\times(-1+h,0)} x_n^{p}v\{\partial_t [(au)_{h}]-(u\partial_t a)_h\}\,\ud x\ud t\nonumber\\
&= \iint_{B_1^+\times(-1+h,0)} x_n^{p}v\{a\partial_t u_h+u(\cdot,t-h)\partial_t a_h-(u\partial_t a)_h\}\,\ud x\ud t.\label{eq:localmaxapp3}
\end{align}
Furthermore,
\begin{align}
 &\iint_{B_1^+\times(-1+h,0)} x_n^{p}v\{u(\cdot,t-h)\partial_t a_h-(u\partial_t a)_h\}\,\ud x\ud t\nonumber\\
 &=\frac{1}{h}\iint_{B_1^+\times(-1+h,0)} x_n^{p}v \int_{t-h}^{t}[u(x,t-h)-u(x,s)]\partial_s a(x,s)\,\ud s\ud x\ud t\nonumber\\
 &\to 0\quad\mbox{as }h\to 0.\label{eq:localmaxapp4}
\end{align}
The proof of \eqref{eq:localmaxapp4} is as follows. By Theorem \ref{thm:weightedsobolev} and Theorem \ref{thm:weightedsobolev2},  $u\in L^{2\chi}(Q_1^+,x_n^p\ud x\ud t)$. For every $\va>0$, there exists $\phi\in C^1(\overline Q_1^+)$ such that
\[
\|u-\phi\|_{L^{2\chi}(Q_1^+,x_n^p\ud x\ud t)}\le\va.
\]
Using $\phi\in C^1(\overline Q_1^+)$ and the dominated convergence theorem,
\[
\lim_{h\to 0}\frac{1}{h}\iint_{B_1^+\times(-1+h,0)} x_n^{p}v \int_{t-h}^{t}[\phi(x,t-h)-\phi(x,s)]\partial_s a(x,s)\,\ud s\ud x\ud t=0.
\]
Then
\begin{align*}
&\lim_{h\to 0^+}\left|\frac{1}{h}\iint_{B_1^+\times(-1+h,0)} x_n^{p}v \int_{t-h}^{t}[u(x,t-h)-u(x,s)]\partial_s a(x,s)\,\ud s\ud x\ud t\right|\\
&\le \|u-\phi\|_{L^{2\chi}(Q_1^+,x_n^p\ud x\ud t)}\|v\|_{L^{2\chi}(Q_1^+,x_n^p\ud x\ud t)}\|\pa_t a\|_{L^{\frac{\chi}{\chi-1}}(Q_1^+,x_n^p\ud x\ud t)}\\
&\le C(n,p,\lambda,\Lambda)\va.
\end{align*}
Since $\va$ is arbitrary, the conclusion \eqref{eq:localmaxapp4} follows.

For $0<\delta<1/4$, $3\delta-1<\tau<0$, define
\[
\xi_\delta(t)=
\begin{cases}
0, \mbox{ when } t<\delta-1,\\
\frac{t+\tau+\delta}{\delta}, \mbox{ when } \delta-1 \le t<2\delta-1,\\
1, \mbox{ when } 2\delta-1\le t<\tau-\delta, \\
\frac{-t}{\delta}, \mbox{ when } \tau-\delta\le t<\tau,\\
0, \mbox{ when } t>\tau.
\end{cases}
\]
Take $v=\xi_\delta(t)(u_h-k)^+$. Combining \eqref{eq:localmaxapp3} and \eqref{eq:localmaxapp4}, and using $u\in C ([-1,0]; L^2(B_1^+,x_n^{p}\ud x))$, Lemma \ref{lem:steklovaverageconvergence}, Theorem \ref{thm:weightedsobolev} and Theorem \ref{thm:weightedsobolev2},   we have
\begingroup
\allowdisplaybreaks
\begin{align}
&-\lim_{h\to 0}\iint_{Q_1^+} x_n^{p}(v_{-h}\partial_t a+a\partial_t v_{-h})u\,\ud x\ud t\nonumber\\
&= \lim_{h\to 0}\iint_{B_1^+\times(-1+h,0)} x_n^{p}v a\partial_t u_h\,\ud x\ud t\nonumber\\
&=\lim_{h\to 0}\frac12 \iint_{B_1^+\times(-1+h,0)} x_n^{p}a(x,t)\xi_\delta(t)\partial_t [(u_h-k)^+]^2\,\ud x\ud t\nonumber\\
&=-\lim_{h\to 0}\frac12 \iint_{B_1^+\times(-1+h,0)} x_n^{p}\big\{[(u_h-k)^+]^2 a\pa_t\xi_\delta(t) +[(u_h-k)^+]^2\xi_\delta(t) \pa_ta\big\} \,\ud x\ud t\nonumber\\
&\ge -\frac{\Lambda}{2\delta} \iint_{B_1^+\times(-1+\delta,-1+2\delta)} x_n^{p}[(u-k)^+]^2 +\frac{\lambda}{2\delta} \iint_{B_1^+\times(\tau-\delta,\tau)} x_n^{p}[(u-k)^+]^2  \nonumber\\
&\quad-\frac12 \iint_{B_1^+\times(-1,0)}  [(u-k)^+]^2 |\pa_ta|\,\ud x\ud t\nonumber\\
&\to \left.\frac{\lambda}{2} \int_{B_1^+} x_n^{p}[(u-k)^+]^2 \,\ud x\right\vert_{\tau}-\frac12 \iint_{B_1^+\times(-1,0)}  [(u-k)^+]^2 |\pa_ta|\,\ud x\ud t \quad\mbox{as }\delta\to 0.\label{eq:localmaxapp22}
\end{align}
\endgroup
Also, by the proof of Lemma \ref{lem:steklovaverageconvergence}, we have 
\begin{align*}
&\lim_{\delta\to 0} \lim_{h\to 0}\iint_{B_1^+\times(-1+h,0)} \big((a_{ij}D_iu)_h D_jv+(d_j u)_h D_j v+(b_jD_ju+cx_n^p u+c_0u)_hv\big)\,\ud x\ud t\\
&\quad =\iint_{B_1^+\times(-1,\tau)} \big[(a_{ij}D_iu+ d_j u) D_j(u-k)^+   +(b_jD_ju+cx_n^p u+c_0u)(u-k)^+\big]\,\ud x\ud t,\\
& \lim_{\delta\to 0} \lim_{h\to 0} \iint_{B_1^+\times(-1+h,0)} ((x_n^pf+f_0)_hv+(f_j)_hD_jv)\,\ud x\ud t\\
&\quad = \iint_{B_1^+\times(-1,\tau)} ((x_n^pf+f_0)(u-k)^++f_jD_j(u-k)^+)\,\ud x\ud t.
\end{align*}
Therefore,  \eqref{eq:uastestfunction} follows from \eqref{eq:localmaxapp1}, \eqref{eq:localmaxapp2},  \eqref{eq:localmaxapp22}, and the Cauchy-Schwarz inequality.
\end{proof}

We have the following uniqueness of weak solutions.

\begin{thm}\label{thm:uniquenessofweaksolution}
Suppose $u$ is a weak solution of \eqref{eq:linear-eq} with the full boundary condition $u\equiv 0$ on $\pa_{pa} Q_1^+$, where the coefficients of the equation satisfy \eqref{eq:rangep}, \eqref{eq:ellip2},  \eqref{eq:assumptioncoefficient} and \eqref{eq:assumptionf}. Then there exists $C>0$ depending only on $n$, $\lda, \Lda$ and $p$ such that
\begin{equation}\label{eq:energyestimateu}
\|u\|_{V_2(Q_1^+)}\le C \|f\|_{L^\frac{2\chi}{2\chi-1}(Q_1^+,x_n^p\ud x\ud t)}+C\|f_0\|_{L^\frac{2\chi}{2\chi-1}(Q_1^+)}+ C\sum_{j=1}^n\|f_j\|_{L^2(Q_1^+)}.
\end{equation}
Consequently,
there exists at most one weak solution of \eqref{eq:linear-eq} with the full boundary condition $u\equiv 0$ on $\pa_{pa} Q_1^+$.
\end{thm}

\begin{proof}
By letting $k=0$ in Lemma \ref{lem:Steklovapproximation}, we have
\begin{align}
&\int_{B^+_1}x_n^{p}  u(x,s)^2\,\ud x+\int_{-1}^s\int_{B_1^+} |\nabla u|^2\,\ud x\ud t\nonumber\\
&\le C \int_{-1}^s\int_{B_1^+} \Big[\Big(\sum_{j=1}^n(d_j^2+b_j^2)+|c_0|+(|\pa_t a|+|c|)x_n^p\Big)u^2 +\sum_{j=1}^n f_j^2 +|f_0u|+|x_n^pfu|\Big]\,\ud x\ud t\nonumber\\
& \le C\|u\|^2_{L^{\frac{2q}{q-1}}(B_1^+\times[-1,s])}+C\|u\|^2_{L^{\frac{2q}{q-1}}(B_1^+\times[-1,s],x_n^p\ud x\ud t)}\nonumber\\
&\quad+ C \int_{-1}^s\int_{B_1^+} (\sum_{j=1}^n f_j^2+|f_0u|+|x_n^pfu|)\,\ud x\ud t.\label{eq:auxenergyestimate}
%&\le \left\{
  %\begin{aligned}
   %C\|u\|^2_{L^{\frac{2q}{q-1}}(B_1^+\times[-1,s])}+ C \int_{-1}^s\int_{B_1^+} (\sum_{j=1}^n f_j^2+|f_0u|)\,\ud x\ud t,  \quad  &\mbox{for }   p\ge 0, \\
  % C\|u\|^2_{L^{\frac{2q}{q-1}}(B_1^+\times[-1,s],x_n^p\ud x)}+ C \int_{-1}^s\int_{B_1^+} (\sum_{j=1}^n f_j^2+|f_0u|)\,\ud x\ud t, \quad   &\mbox{for }    -1<p<0.
%\end{aligned}
%\right.
\end{align}
Since $q>\frac{\chi}{\chi-1}$, it follows from Theorem \ref{thm:weightedsobolev}, Theorem \ref{thm:weightedsobolev2} and   Young's inequality that
\begin{align*}
\|u\|^2_{L^{\frac{2q}{q-1}}(B_1^+\times[-1,s])}&\le \delta \|u\|^2_{V_2(B_1^+\times[-1,s])} + C(\delta)  \|u\|^2_{L^{2}(B_1^+\times[-1,s])},\\
\|u\|^2_{L^{\frac{2q}{q-1}}(B_1^+\times[-1,s],x_n^p\ud x)}&\le \delta \|u\|^2_{V_2(B_1^+\times[-1,s])} + C(\delta)  \|u\|^2_{L^{2}(B_1^+\times[-1,s],x_n^p\ud x)},\\
 \int_{-1}^s\int_{B_1^+}|f_0u|\,\ud x\ud t&\le \delta \|u\|^2_{V_2(B_1^+\times[-1,s])} + C(\delta)\|f_0\|^2_{L^{\frac{2\chi}{2\chi-1}}(B_1^+\times[-1,s])},\\
  \int_{-1}^s\int_{B_1^+}|x_n^p f u|\,\ud x\ud t&\le \delta \|u\|^2_{V_2(B_1^+\times[-1,s])} + C(\delta)\|f\|^2_{L^{\frac{2\chi}{2\chi-1}}(B_1^+\times[-1,s],x_n^p\ud x\ud t)}.
\end{align*}
Plugging these to \eqref{eq:auxenergyestimate} and using Lemma \ref{lem:interpolation}, we obtain
\begin{align}\label{eq:beforegronwall}
\|u\|^2_{V_2(B_1^+\times[-1,s])} \le C \int_{-1}^s\int_{B_1^+} x_n^pu^2\,\ud x\ud t +CF_0^2,
\end{align}
where $F_0$ is defined in \eqref{eq:assumptionf}. In particular, 
\[
\|u(\cdot,s)\|_{L^2(B_1^+,x_n^p\ud x)}^2\le C\|u\|_{L^2(B_1^+\times(-1,s],x_n^p\ud x\ud t)}^2+CF_0^2.
\]
By Gronwall's inequality, we have
\[
\|u\|_{L^2(B_1^+\times(-1,s],x_n^p\ud x\ud t)}^2\le CF_0^2.
\]
Plugging this back to \eqref{eq:beforegronwall}, the estimate \eqref{eq:energyestimateu} follows. Therefore, the uniqueness holds. 
\end{proof}

%\begin{rem}\label{rem:uniformweakconvergence}
%If one examines the proof of \eqref{eq:auxenergyestimate},  instead of assuming $u\in \mathring V_2^{1,0}(Q^+_{1})$, it would suffice to assume $u\in \mathring V_2(Q^+_{1})$ and that $u$ satisfies
%\[
%\lim_{h\to 0} \int_{B_1^+}(u(x,t+h)-u(x,t))\phi(x)x_n^p\,\ud x=0\ \forall\quad\mbox{uniformly in }t,
%\] 
%that is, for every $\va>0$, there exists $\delta>0$ such that for all $|h|<\delta$, and all $t\in [-1,0]$ that $t+h\in[-1,0]$, and all $\phi\in L^2(B_1^+,x_n^p\ud x)$,
%\[
%\left|\int_{B_1^+}(u(x,t+h)-u(x,t))\phi(x)x_n^p\,\ud x\right|<\va.
%\]
%\end{rem}

%In the following, we use approximations to prove the existence of solutions. We need slightly higher integrability on $\pa_t a, c$ and $f$ than those in \eqref{eq:assumptioncoefficient} and \eqref{eq:assumptionf} for $p>0$, so that the approximating uniformly parabolic equations have weak solutions. But the energy estimates \eqref{eq:energyestimateu} shown in the above do not depend on these norms. For two real numbers $r_1$ and $r_2$, we denote  $r_1\vee r_2=\max(r_1,r_2)$. 

\begin{thm}\label{thm:existenceofweaksolution}
Suppose $a$ is continuous in $\overline Q_1^+$, and the conditions \eqref{eq:rangep},  \eqref{eq:ellip2},  \eqref{eq:assumptioncoefficient} and \eqref{eq:assumptionf} hold. Then there exists a unique weak solution of \eqref{eq:linear-eq} with the full boundary condition $u\equiv 0$ on $\pa_{pa} Q_1^+$. 
\end{thm}

\begin{proof} 
For two real numbers $r_1$ and $r_2$, we denote  $r_1\vee r_2=\max(r_1,r_2)$. We first consider the case with an additionally assume that $\pa_t a, c \in L^q(Q_1^+,x_n^p\vee 1\,\ud x\ud t)$ and $f\in L^\frac{2\chi}{\chi-1}(Q_1^+,x_n^p\vee 1\,\ud x\ud t)$, where $\chi$ is the constant in Theorem \ref{thm:weightedsobolev} or Theorem \ref{thm:weightedsobolev2}. An approximation argument in the end would remove this assumption. 

For all $\va\in(0,1)$, let $a^\va\in C^2(\overline Q_1^+)$ be such that $a^\va\to a$ uniformly on $Q_1^+$, and $\pa_t a^\va\to \pa_t a$ in $L^q(Q_1^+,x_n^p\vee 1\,\ud x\ud t)$. Then there exists a unique energy weak solution $u_\va\in C ([-1,0]; L^2(B_1^+)) \cap  L^2((-1,0];H_{0}^1(B_1^+))$ to the uniformly parabolic equation
\begin{align}\label{eq:degiorgilinearappappendix}
&a^\va \cdot (x_n+\va)^{p} \pa_t u_\va -D_j(a_{ij} D_i u_\va+d_j u_\va)+b_iD_i u_\va+c(x_n+\va)^p u_\va+c_0 u_\va\nonumber \\
&\quad=(x_n+\va)^pf+f_0 -D_if_i \quad \mbox{in }Q_1^+
\end{align}
with  $u_\va\equiv 0$ on $\pa_{pa} Q_1^+$. That is, 
\begin{equation}\label{eq:definitionweaksolution22}
\begin{split}
&\int_{B^+_1}a^\va(x,s) (x_n+\va)^{p}  u_\va(x,s) \varphi(x,s)\,\ud x-\int_{-1}^s\int_{B_1^+} (x_n+\va)^{p}(\varphi\partial_t a^\va+a^\va\partial_t \varphi)u_\va\,\ud x\ud t\\
&=- \int_{-1}^s\int_{B_1^+} \big(a_{ij}D_iu_\va D_j\varphi+d_ju_\va D_j\varphi+b_jD_ju_\va \varphi+c(x_n+\va)^p u_\va \varphi+c_0u_\va \varphi\big)\,\ud x\ud t\\
&\quad +\int_{-1}^s\int_{B_1^+} ((x_n+\va)^pf\varphi+f_0\varphi+f_jD_j\varphi)\,\ud x\ud t
\end{split}
\end{equation}
for every $\varphi\in \mathring W^{1,1}_2(Q^+_1)$ satisfying $\varphi (\cdot,-1)\equiv 0$ in $B_1^+$ (in the trace sense).  By the same proof of \eqref{eq:energyestimateu}, we have
\begin{align}\label{eq:energysmoothcase}
&\sup_{t\in[-1,0]}\int_{B_1^+}(x_n+\va)^{p} u_\va^2\,\ud x+\|\nabla u_\va\|^2_{L^2(Q_1^+)}\le  CF_\va^2,
\end{align}
where 
\[
F_\va=\|f\|_{L^\frac{2\chi}{2\chi-1}(Q_1^+,(x_n+\va)^p\ud x\ud t)}+\|f_0\|_{L^\frac{2\chi}{2\chi-1}(Q_1^+)}+ \sum_{j=1}^n\|f_j\|_{L^2(Q_1^+)}.
\] 
Hence, if $p\ge 0$, then
\[
\|u_\va\|^2_{V_2(Q_1^+)}\le   CF_\va^2.
\]
If $-1<p<0$, then by \eqref{eq:energysmoothcase} and the proof of Theorem \ref{thm:weightedsobolev2}, we have
\begin{align}\label{eq:weightedsobolev22}
 \left(\int_{Q^+_{R,T}} (x_n+\va)^p |u_\va|^{2\chi}\ud x \ud t \right)^{\frac{1}{\chi}}  \le  CF_\va^2.
\end{align}
%and
%\[
%\sup_{t\in[-1,0]}\int_{B_1^+}u_\va^2\,\ud x+\|\nabla u_\va\|^2_{L^2(Q_1^+)}\le  C\|f_0\|^2_{L^{\frac{2\chi}{2\chi-1}}(Q_1^+)} +   C \sum_{j=1}^n\|f_j\|_{L^2(Q_1^+)}.
%\]

Therefore, by Theorem \ref{thm:weightedsobolev} and Theorem \ref{thm:weightedsobolev2}, for all $p>-1$, there exist   $u\in L^{2\chi}(Q_1^+)\cap L^2((-1,0];H_{0}^1(B_1^+))$ and a subsequence $\{u_{\va_j}\}$,  such that $u_{\va_j} \rightharpoonup u$ weakly  in $L^{2\chi}(Q_1^+)$ and $Du_{\va_j} \rightharpoonup Du$ weakly in $L^{2}(Q_1^+)$.  Let $\varphi\in  C^\infty(Q^+_1)$ be such that $\varphi \equiv 0$ near the parabolic boundary $\pa_{pa}Q^+_1$ and let
\[
h_j(s):= \int_{B^+_1}a^{\va_j}(x,s) (x_n+\va_j)^{p}  u_{\va_j}(x,s) \varphi(x,s)\,\ud x.
\]
By \eqref{eq:definitionweaksolution22}, \eqref{eq:energysmoothcase}, \eqref{eq:weightedsobolev22}, Theorem \ref{thm:weightedsobolev}, and the absolute continuity of Lebesgue integrals (applying to the right hand side of \eqref{eq:definitionweaksolution22}), we know that $h_j$ is uniformly bounded and equicontinuous on $[-1,0]$. By the Ascoli-Arzela Theorem, there is a subsequence of $\{h_j\}$, which is still denoted by $\{h_j\}$, such that $h_j$ uniformly converges to a function $h\in C([-1,0])$. On the other hand, since $u_{\va_j} \rightharpoonup u$ weakly  in $L^{2\chi}(Q_1^+)$, we have that for every interval $I\subset[-1,0]$,
\[
\int_I h_j(s)\,\ud s \to \int_I \int_{B^+_1}a(x,s) x_n^{p}  u(x,s) \varphi(x,s)\,\ud x\ud s.
\]
Hence,
\[
h(s)=\int_{B^+_1}a(x,s) x_n^{p}  u(x,s) \varphi(x,s)\,\ud x\quad\mbox{a.e. in }[-1,0].
\]
Therefore, if one considers such a $\varphi$ independent of the time variable, then we know from \eqref{eq:energysmoothcase} that $u\in L^\infty([-1,0]; L^2(B_1^+,x_n^{p}\ud x))$, and it is straightforward to verify by sending $\va_j\to 0$ in \eqref{eq:definitionweaksolution22} that $u$ satisfies \eqref{eq:definitionweaksolution} for every $\varphi\in  C^\infty(Q^+_1)$ being such that $\varphi \equiv 0$ near the parabolic boundary $\pa_{pa}Q^+_1$. 

When $p\ge 0$, then by a standard density argument, it is straightforward to verify that $u$ satisfies \eqref{eq:definitionweaksolution}  for every $\varphi\in \mathring W^{1,1}_2(Q^+_1)$ satisfying $\varphi (\cdot,-1)\equiv 0$ in $B_1^+$ (in the trace sense). By Lemma \ref{lem:sobolevdense}, this $u$ satisfies \eqref{eq:definitionweaksolution} for every $\varphi\in \mathring V^{1,1}_2(Q^+_1)$ satisfying $\varphi (\cdot,-1)\equiv 0$ in $B_1^+$.

%Therefore, by Theorem \ref{thm:weightedsobolev} (for $p=0$), there exist $u\in L^{2}(Q_1^+)\cap L^2((-1,0];H_{0}^1(B_1^+))$ and a subsequence $\{u_{\va_j}\}$,  such that $u_{\va_j} \rightharpoonup u$ weakly  in $L^{2}(Q_1^+)$ and $Du_{\va_j} \rightharpoonup Du$ weakly in $L^{2}(Q_1^+)$. Therefore, there exists $u\in L^\infty([-1,0], L^2(B_1^+))\cap L^2([-1,0], H^1_0(B_1^+))$ and  a sequence $\{u_{\va_j}\}$,  such that $u_{\va_j} \rightharpoonup u$  in the weak $*$ topology of $L^\infty([-1,0], L^2(B_1^+))\cap L^2([-1,0], H^1_0(B_1^+))$.  By the H\"older estimates for the weak solutions of \eqref{eq:degiorgilinearappappendix} away from $\{x_n=0\}$, we know that there exists a subsequence, which is still denoted as $\{u_{\va_j}\}$, such that $u_{\va_j}$ converges to $u$ almost everywhere in $Q_1^+$. Then it follows from \eqref{eq:energysmoothcase} and Fatou's Lemma that 

When $-1<p<0$, we also use approximation arguments. Let  $\varphi\in \mathring V^{1,1}_2(Q^+_1)$ satisfy $\varphi (\cdot,-1)\equiv 0$ in $B_1^+$. Using Minkowski's integral inequality, for every $\delta>0$, there exists $\mu>0$ such that 
\[
\|\varphi\|_{V^{1,1}_2(Q^+_1\cap\{x_n<\mu\})} +\sup_{s\in (-1,0]}\|\varphi(\cdot,s)\|_{L^2(B^+_1\cap\{x_n<\mu\}, x_n^p\ud x)}+\|\varphi\|_{L^{2\chi}(Q^+_1\cap\{x_n<\mu\})}<\delta,
\] 
where $\chi$ is the one in Theorem \ref{thm:weightedsobolev2}. Let $\eta_\mu$ be a smooth cut-off  function such that $\eta\equiv 1$ on $[\mu,+\infty)$ and $\eta\equiv 0$ on $[0,\mu/2]$. Let $\varphi_1(x,t)=\eta(x_n)\varphi(x,t)$ and $\varphi_2(x,t)=(1-\eta(x_n))\varphi(x,t)$. Using the fact that $V^{1,1}_2(Q^+_{1})\subset W^{1,1}_2(Q^+_{1})$ when $-1<p<0$, we have \eqref{eq:definitionweaksolution22}. Similar to the above, by using the weak convergence of $u_\va$, it is straightforward to verify that
\begin{align*}
&\lim_{\va\to 0}\int_{B^+_1}a^\va(x,s) (x_n+\va)^{p}  u_\va(x,s) \varphi_1(x,s)\,\ud x=\int_{B^+_1}a(x,s) x_n^{p}  u(x,s) \varphi_1(x,s)\,\ud x\ \ \mbox{a.e. }s\in[-1,0]
\end{align*}
and
\begin{align*}
&\lim_{\va\to 0}\int_{-1}^s\int_{B_1^+} (x_n+\va)^{p}(\varphi_1\partial_t a^\va+a^\va\partial_t \varphi_1)u_\va\,\ud x\ud t=\int_{-1}^s\int_{B_1^+} x_n^{p}(\varphi_1\partial_t a+a\partial_t \varphi_1)u\,\ud x\ud t.
\end{align*}
%\begin{align}\label{eq:energysmoothcasewithp}
%\sup_{t\in[-1,0]}\int_{B_1^+}u^2x_n^{p}\,\ud x+\|\nabla u\|^2_{L^2(Q_1^+)}\le  C\|f_0\|^2_{L^{\frac{2\chi}{2\chi-1}}(Q_1^+)} +   C \sum_{j=1}^n\|f_j\|_{L^2(Q_1^+)}.
%\end{align}
By using Theorem \ref{thm:weightedsobolev2}, H\"older's inequality, \eqref{eq:energysmoothcase} and \eqref{eq:weightedsobolev22}, we can verify that
\begin{align*}
\left|\int_{B^+_1}a^\va(x,s) (x_n+\va)^{p}  u_\va(x,s) \varphi_2(x,s)\,\ud x-\int_{-1}^s\int_{B_1^+} (x_n+\va)^{p}(\varphi_2\partial_t a^\va+a^\va\partial_t \varphi_2)u_\va\,\ud x\ud t\right|&\le C\delta,\\
\left|\int_{B^+_1}a(x,s) x_n^{p}  u(x,s) \varphi_2(x,s)\,\ud x-\int_{-1}^s\int_{B_1^+} x_n^{p}(\varphi_2\partial_t a+a\partial_t \varphi_2)u\,\ud x\ud t\right|&\le C\delta.
\end{align*}
Then by sending $\va\to 0$ and then $\delta\to 0$ in \eqref{eq:definitionweaksolution22}, it follows that \eqref{eq:definitionweaksolution} holds for every $\varphi\in \mathring V^{1,1}_2(Q^+_1)$ satisfying $\varphi (\cdot,-1)\equiv 0$ in $B_1^+$. 

Next, we want to verify that $u\in C ([-1,0]; L^2(B_1^+,x_n^{p}\ud x)) $. Note that we have
\begin{equation}\label{eq:continuityV2-1}
\begin{split}
&\int_{B^+_1}a(x,s) x_n^{p}  u(x,s) \varphi(x,s)\,\ud x-\int_{-1}^s\int_{B_1^+} x_n^{p}a u\partial_t \varphi \,\ud x\ud t\\
&=\int_{-1}^s\int_{B_1^+} (g_0\varphi+g_jD_j\varphi)\,\ud x\ud t\quad\mbox{a.e. }s\in (-1,0],
\end{split}
\end{equation}
where
\begin{align*}
g_j&= f_j-a_{ij}D_iu-d_ju,\\
g_0&=x_n^pf+f_0-b_jD_ju-c_0u+x_n^p(\pa_t a+c) u.
\end{align*}
Hence, we know that $g_j\in L^2(Q_1^+), j=1,\cdots,n$, and $g_0\in L^{\frac{2\chi}{2\chi-1}}(Q_1^+)$. Moreover, we clearly have
\begin{equation}\label{eq:continuityV2-2}
\|u(\cdot,s)\|_{L^2(B_1^+,x_n^p\ud x)}\le \|u\|_{V_2(Q_1^+)},\quad\mbox{a.e. }s\in(-1,0].
\end{equation}
Denote
\[
I:=\{s\in[-1,0]: \eqref{eq:continuityV2-1} \mbox{ and } \eqref{eq:continuityV2-2} \mbox{ hold for }s.\}
\]
Then we know that $I$ is of measure zero. We can redefine $u(x,s)$ such that 
 \eqref{eq:continuityV2-1} and \eqref{eq:continuityV2-2} for $s\in[-1,0]\setminus I$. Indeed, because of $\eqref{eq:continuityV2-2}$, for every $s_0\in [-1,0]\setminus I$, there exists $\{s_k\}\subset I$ such that $s_k\to s_0$ and $u(\cdot,s_k)\rightharpoonup v(\cdot)$ in $L^2(B_1^+,x_n^p\ud x)$. We redefine $u(\cdot,s_0)=v(\cdot)$. Then \eqref{eq:continuityV2-1} and \eqref{eq:continuityV2-2} hold for $s_0$, and moreover, by \eqref{eq:continuityV2-1}, this $v(\cdot)$ is independent on the choice  of the sequence $\{s_k\}$. Thus, we can assume that \eqref{eq:continuityV2-1} and \eqref{eq:continuityV2-2} hold for all $s\in[-1,0]$.
 
Let $Q_{1,s,h}^+=B_1^+\times(s,s+h)$ when $s,s+h\in(-1,0)$ (here, we assume $h>0$, and the argument for the case $h<0$ can be modified correspondingly). From \eqref{eq:continuityV2-1}, we obtain
\begin{equation}\label{eq:continuityV2-3}
\begin{split}
&\int_{B^+_1}a(x,s+h) x_n^{p}  u(x,s+h) \varphi(x,s+h)\,\ud x-\int_{B^+_1}a(x,s) x_n^{p}  u(x,s) \varphi(x,s)\,\ud x\\
&=\int_{Q_{1,s,h}^+} x_n^{p}a u\partial_t \varphi \,\ud x\ud t+\int_{Q_{1,s,h}^+} (g_0\varphi+g_jD_j\varphi)\,\ud x\ud t.
\end{split}
\end{equation}
By choosing $\varphi$ as a function in $C^\infty_c(B_1^+)$, and since $a$ is continuous in $\overline Q_1^+$, we have
\begin{align*}
\left|\int_{B^+_1}[a(x,s+h)-a(x,s)] x_n^{p}  u(x,s+h) \varphi(x)\,\ud x\right|\to 0 \quad\mbox{as }h\to 0.
\end{align*}
Hence,
\begin{equation}\label{eq:weakconvergenceweight}
\lim_{h\to 0}\int_{B^+_1} x_n^{p} a(x,s) (u(x,s+h) -u(x,s)) \varphi(x)\,\ud x=0\quad\mbox{uniformly in }s.
\end{equation}
By a density argument, \eqref{eq:weakconvergenceweight} holds for all $\varphi\in L^2(B_1^+,x_n^p\ud x)$. 

Choose $\va>0$ small such that $s\pm\va, s+h\pm\va \in (-1,0)$. Let
\[
\tilde u(x,t)=
\begin{cases}
u(x,s+h),\quad (x,t)\in B_1^+\times(s+h,\infty),\\
u(x,t), \quad (x,t)\in B_1^+\times(s,s+h],\\
u(x,s), \quad (x,t)\in B_1^+\times(-1,s].
\end{cases}
\]
and
\[
\varphi_\va(x,t)=\frac{1}{2\va}\int_{t-\va}^{t+\va}\tilde u(x,\tau)\,\ud \tau.
\]
Then $\varphi_\va\in \mathring V^{1,1}_2(Q^+_1)$ and \eqref{eq:continuityV2-3} holds for $\varphi_\va$. Note that
 \begin{align*}
&\int_{s}^{s+h}\int_{B_1^+} x_n^{p}au\partial_t \varphi_\va\,\ud x\ud t\\
&=\frac{1}{2\va} \int_{s}^{s+h}\int_{B_1^+} x_n^{p}\big[a(x,t)u(x,t)\tilde u(x,t+\va)-a(x,t)u(x,t)\tilde u(x,t-\va)\big]\,\ud x\ud t\\
&=\frac{1}{2\va}\int_{B_1^+} x_n^{p}\int_{s+h-\va}^{s+h} a(x,t)u(x,t)u(x,s+h)\,\ud x\ud t\\
&\quad-\frac{1}{2\va}\int_{B_1^+} x_n^{p}\int_{s}^{s+\va} a(x,t)u(x,t)u(x,s)\,\ud x\ud t\\
&\quad + \frac{1}{2\va}\int_{B_1^+} x_n^{p}\int_{s+\va}^{s+h} u(x,t-\va)u(x,t)[a(x,t-\va)-a(x,t)]\,\ud x \ud t.
%&\to \frac{1}{2}\int_{B_1^+} x_n^{p} a(x,t)u(x,t)^2\,\ud x-\frac{1}{2}\int_{B_1^+} \int_{-1}^{s}x_n^{p}u(x,t)^2\partial_t a\,\ud x\ud t.
\end{align*}
Using the continuity of $a$ and \eqref{eq:weakconvergenceweight}, we have
 \begin{align*}
&\lim_{\va \to 0}\int_{s}^{s+h}\int_{B_1^+} x_n^{p}au\partial_t \varphi_\va\,\ud x\ud t\\
&=\frac{1}{2}\int_{B_1^+} x_n^{p}a(x,s+h)u(x,s+h)^2\,\ud x-\frac{1}{2}\int_{B_1^+} x_n^{p}a(x,s)u(x,s)^2\,\ud x.
%&\to \frac{1}{2}\int_{B_1^+} x_n^{p} a(x,t)u(x,t)^2\,\ud x-\frac{1}{2}\int_{B_1^+} \int_{-1}^{s}x_n^{p}u(x,t)^2\partial_t a\,\ud x\ud t.
\end{align*}
Setting $\varphi=\varphi_\va$ in \eqref{eq:continuityV2-3} and then letting $\va\to 0$, we have by using \eqref{eq:weakconvergenceweight} that
\begin{align*}
&\frac{1}{2}\int_{B_1^+} x_n^{p} a(x,s+h)u(x,s+h)^2\,\ud x-\frac{1}{2}\int_{B_1^+} x_n^{p} a(x,s)u(x,s)^2\,\ud x\\
&=-\frac{1}{2}\int_{Q_{1,s,h}^+} x_n^{p}u^2\partial_t a\,\ud x\ud t+\int_{Q_{1,s,h}^+} (g_0\varphi+g_jD_j\varphi)\,\ud x\ud t.
\end{align*}
Hence, 
\[
\lim_{h\to 0}\int_{B_1^+} x_n^{p} a(x,s+h)u(x,s+h)^2\,\ud x=\int_{B_1^+} x_n^{p} a(x,s)u(x,s)^2\,\ud x.
\]
Since $a\in C(\overline Q_1^+)$, we have
\begin{equation}\label{eq:normconvergenceweight00}
\lim_{h\to 0}\int_{B_1^+} x_n^{p} [a(x,s+h)-a(x,s)]u(x,s+h)^2\,\ud x=0,
\end{equation}
and thus,
\begin{equation}\label{eq:normconvergenceweight}
\lim_{h\to 0}\int_{B_1^+} x_n^{p} a(x,s)u(x,s+h)^2\,\ud x=\int_{B_1^+} x_n^{p} a(x,s)u(x,s)^2\,\ud x.
\end{equation}
It follows from \eqref{eq:weakconvergenceweight} and \eqref{eq:normconvergenceweight} that
\[
\lim_{h\to 0}\int_{B^+_1} x_n^{p} a(x,s) |u(x,s+h) -u(x,s)|^2 \,\ud x=0.
\]
Since $a\ge\lambda>0$, we obtain
\[
\lim_{h\to 0}\int_{B^+_1} x_n^{p} |u(x,s+h) -u(x,s)|^2 \,\ud x=0.
\]
Hence, $u\in C ([-1,0]; L^2(B_1^+,x_n^{p}\ud x)) $, and thus,  $u\in \mathring V_2^{1,0}(Q^+_{1})$.

Now let us use another approximation to remove the assume that $\pa_t a, c \in L^q(Q_1^+,x_n^p\vee 1\,\ud x\ud t)$  and $f\in L^\frac{2\chi}{\chi-1}(Q_1^+,x_n^p\vee 1\,\ud x\ud t)$.  Suppose \eqref{eq:assumptioncoefficient} and \eqref{eq:assumptionf} hold. Let $a^\va\in C^2(\overline Q_1^+)$ be such that $a^\va\to a$ uniformly on $Q_1^+$ and $\pa_t a^\va\to \pa_t a$ in $L^q(Q_1^+,x_n^p\,\ud x\ud t)$, $c^\va\in C^2(\overline Q_1^+)$ be such that $c^\va\to c$  in $L^q(Q_1^+,x_n^p\,\ud x\ud t)$, and $f^\va\in C^2(\overline Q_1^+)$ be such that $f^\va\to f$ in $L^\frac{2\chi}{\chi-1}(Q_1^+,x_n^p\,\ud x\ud t)$. Then as proved in the above, there exists a weak solution $u_\va\in \mathring V_2^{1,0}(Q^+_{1})$ to the  parabolic equation
\[
a^\va \cdot x_n^{p} \pa_t u_\va -D_j(a_{ij} D_i u_\va+d_j u_\va)+b_iD_i u_\va+c^\va x_n^p u_\va+c_0 u_\va=x_n^pf^\va+f_0 -D_if_i \quad \mbox{in }Q_1^+
\]
with  $u_\va\equiv 0$ on $\pa_{pa} Q_1^+$. Then by the energy estimate in Theorem \ref{thm:uniquenessofweaksolution} and the same argument as above, one can show that $u_\va$ will converge to a  weak solution of \eqref{eq:linear-eq} with the full boundary condition $u\equiv 0$ on $\pa_{pa} Q_1^+$. 

Finally, the uniqueness follows from Theorem \ref{thm:uniquenessofweaksolution}.
\end{proof}

\subsection{$W^{1,1}_2$ regularity}

Next, we want to study the $W^{1,1}_2$ regularity of  weak solutions to the equation \eqref{eq:linear-eq} with slightly stronger assumptions on the coefficients.   Consider the following equation
\begin{equation}\label{eq:linear-eq-app}
a x_n^{p} \pa_t u -D_j[a_{ij} D_i u+(x_n^{p/2}\wedge 1)d_j u]+(x_n^{p/2}\wedge 1)b_iD_i u+cx_n^{p}u+c_0 u=x_n^{p}f\quad \mbox{in }Q_1^+,
\end{equation}
where $x_n^{p/2}\wedge 1=\min(x_n^{p/2},1)$. For the coefficients, besides \eqref{eq:ellip2},  we suppose that 
\begin{equation}\label{eq:assumptioncoefficient2}
\left\|\pa_t a\right\|_{L^q(Q_1^+,x_n^p\ud x\ud t)}+ \|a_{ij}\|_{Lip(\overline Q_1^+)}+\|d_{j}\|_{Lip(\overline Q_1^+)}+\|c_0\|_{Lip(\overline Q_1^+)}+\||b_j|+|c|\|_{L^\infty(Q_1^+)}\le\Lambda,
\end{equation}
for some $q>\frac{\chi}{\chi-1}$. We also suppose that $-\mathrm{div} (A \nabla ) +c_0$ is coercive, where $A=(a_{ij})$, i.e., there exists a constant $\bar \lda>0$ such that
\be \label{eq:coer-app}
\int_{B_1^+} A\nabla \phi \nabla \phi +c_0 \phi^2 \ge \bar \lda \int_{B_1^+} \phi^2 \quad \quad \forall \phi\in H^1_0(B_1^+), a.e.\ t\in[-1,1].
\ee
Note that \eqref{eq:coer-app} implies that there exists $\tilde\lambda>0$ depending only on $\bar\lambda,\lambda,\Lambda$ and $n$ such that
\be \label{eq:coer-app2}
\int_{B_1^+} A\nabla \phi \nabla \phi +c_0 \phi^2 \ge  \tilde\lda \int_{B_1^+} |\nabla\phi|^2 \quad \quad \forall \phi\in H^1_0(B_1^+), a.e.\ t\in[-1,1].
\ee

\begin{thm}\label{thm:energyestimateut}
Suppose $a$ is  continuous in $\overline Q_1^+$, $A$ is symmetric, and the conditions \eqref{eq:rangep}, \eqref{eq:ellip2},  \eqref{eq:assumptioncoefficient2} and \eqref{eq:coer-app} hold. Suppose that $f\in L^2(Q_1^+,x^p\ud x\ud t)$. Let $u$ be the weak solution of \eqref{eq:linear-eq-app} with the full boundary condition $u\equiv 0$ on $\pa_{pa} Q_1^+$. Then
\begin{equation}\label{eq:energyestimateut}
\sup_{t\in(-1,0)}\int_{B_1^+} |\nabla u(x,t)|^2\,\ud x+\int_{Q_1^+} x_n^{p}|\pa_t u|^2\,\ud x\ud t\le C \int_{Q_1^+} x_n^pf^2\,\ud x\ud t,
\end{equation}
where $C>0$ depends only on  $\lda, \bar\lambda,\Lambda,n$ and $p$.
\end{thm}
\begin{proof}
We first assume that $f\in L^2(Q_1^+,x^p\vee 1\ud x\ud t)$ and $\pa_t a\in L^q(Q_1^+,x_n^p\vee1\,\ud x\ud t)$.

For $\va>0$, let $a^\va, a_{ij}^\va, d_j^\va, c_0^\va\in C^\infty(\R^n)$ be such that $a^\va\to a$, $a_{ij}^\va\to a_{ij}$, $d_j^\va\to d_j$, $c_0^\va\to c_0$ uniformly on  $\overline Q_1^+$,  $\pa_t a^\va\to \pa_t a$ in $L^q(Q_1^+,x_n^p\vee 1\,\ud x\ud t)$, and 
\[
 \|a_{ij}^\va\|_{Lip(\overline Q_1^+)}+\|d_{j}^\va\|_{Lip(\overline Q_1^+)}\le C\Lambda.
\]
Let $b_{i}^\va, c^\va\in C^\infty(\R^n)$ be such that $b_{i}^\va\to b_{i}$, $c^\va\to c$ in $L^q(\overline Q_1^+)$ for some $q>\frac{\chi}{\chi-1}$, and   
\[
\|b_{i}^\va\|_{L^\infty(Q_1^+)}+\|c^\va\|_{L^\infty(Q_1^+)}\le C\Lambda.
\] 
Let $f_\va\in C_c^\infty(Q_1^+)$ be such that $f_\va\to f$ in $L^2(Q_1^+,x^p\vee 1\,\ud x\ud t)$ as $\va\to 0$.

Let $u_\va\in C ([-1,0]; L^2(B_1^+)) \cap  L^2((-1,0];H_{0}^1(B_1^+))$ be the unique weak solution of
\be\label{eq:degiorgilinearappappendix2}
\begin{split}
&a^\va \cdot (x_n+\va)^{p} \pa_t u_\va -D_j[a_{ij}^\va D_i u_\va+((x_n+\va)^{p/2}\wedge (1+\va)^{p/2})d_j^\va u_\va]\\
&\quad +((x_n+\va)^{p/2}\wedge (1+\va)^{p/2})b_i^\va D_i u_\va+c_0^\va u_\va+c^\va(x_n+\va)^{p}u_\va=(x_n+\va)^{p} f_\va  \quad \mbox{in }Q_1^+
\end{split}
\ee
with  $u_\va\equiv 0$ on $\pa_{pa} Q_1^+$. By the Schauder regularity theory, we know that $D_{x} u_\va, \pa_t u_\va\in C(\overline Q_1^+)$.

For small $h>0$, denote $$u_\va^h(x,t)=\frac{u_\va(x,t+h)-u_\va(x,t)}{h}$$ for all $-1\le t\le -h$, and denote the left hand side of \eqref{eq:degiorgilinearappappendix2} as $I(x,t)$. Then we have for  all $-1<t<-h$,
\begin{equation}\label{eq:partialttestfunction}
\begin{split}
&\int_{B_1^+\times(-1,t]} (I(x,s+h)+I(x,s))u_\va^h(x,s)\,\ud x\ud s\\
&=\int_{B_1^+\times(-1,t]}(x_n+\va)^{p/2}(f_\va (x,s+h)+f_\va (x,s))u_\va^h(x,s)\,\ud x\ud s.
\end{split}
\end{equation}
Using the symmetry of $A$, we have
\begingroup
\allowdisplaybreaks
\begin{align*}
&\int_{B_1^+}\int_{-1}^t [a_{ij}^\va(x,s+h)D_iu_\va(x,s+h)+ a_{ij}^\va(x,s)D_iu_\va(x,s)]D_ju_\va^h(x,s)\,\ud s\ud x\\
%&=\frac{1}{h}\int_{B_1^+}\int_{-1}^t (A\nabla u_\va(x,s+h))\cdot \nabla u_\va(x,s+h)\,\ud s\ud x\\
%&\quad -\frac{1}{h}\int_{B_1^+}\int_{-1}^t  (A\nabla u_\va(x,s))\cdot \nabla u_\va(x,s)]\,\ud s\ud x\\
&=\frac{1}{h}\int_{B_1^+}\int_{t}^{t+h} a_{ij}^\va D_iu_\va D_ju_\va \,\ud s\ud x - \frac{1}{h}\int_{B_1^+}\int_{-1}^{-1+h} a_{ij}^\va D_iu_\va D_ju_\va\,\ud s\ud x\\
&\quad+ \int_{B_1^+}\int_{-1}^{t} \frac{a_{ij}^\va(x,s)-a_{ij}^\va(x,s+h)}{h}D_iu_\va(x,s)D_ju_\va(x,s+h) \,\ud s\ud x \\
& \to \int_{B_1^+} a_{ij}(x,t)D_i u_\va(x,t)D_j u_\va(x,t)\,\ud x -\int_{B_1^+}\int_{-1}^{t} \pa_s a_{ij}^\va D_iu_\va D_ju_\va\,\ud s\ud x\quad\mbox{as }h\to 0,
\end{align*}
\endgroup
where we used that $D_xu\in C^{0}(\overline B_1^+\times[-1,0])$ and $u(x,-1)\equiv 0$. 
Here, we used $u_\va^h(x,t)$ instead of $\pa_t u_\va$ to avoid involving $D_x\pa_t u_\va$ in the calculation.  Also,
\begingroup
\allowdisplaybreaks
\begin{align*}
&\int_{B_1^+}\int_{-1}^t ((x_n+\va)^{p/2}\wedge (1+\va)^{p/2})[d_{j}^\va(x,s+h)u_\va(x,s+h)+ d_{j}^\va(x,s)u_\va(x,s)]\pa_ju_\va^h(x,s)\,\ud s\ud x\\
%&=\frac{1}{h}\int_{B_1^+}\int_{-1}^t (A\nabla u_\va(x,s+h))\cdot \nabla u_\va(x,s+h)\,\ud s\ud x\\
%&\quad -\frac{1}{h}\int_{B_1^+}\int_{-1}^t  (A\nabla u_\va(x,s))\cdot \nabla u_\va(x,s)]\,\ud s\ud x\\
%&=\frac{1}{h}\int_{B_1^+}\int_{t}^{t+h} a_{ij}^\va\pa_iu_\va\pa_ju_\va \,\ud s\ud x - \frac{1}{h}\int_{B_1^+}\int_{-1}^{-1+h} a_{ij}^\va\pa_iu_\va\pa_ju_\va\,\ud s\ud x\\
%&\quad+ \int_{B_1^+}\int_{-1}^{t} \frac{a_{ij}^\va(x,s)-a_{ij}^\va(x,s+h)}{h}\pa_iu_\va(x,s)\pa_ju_\va(x,s+h) \,\ud s\ud x \\
& \to 2\int_{B_1^+} ((x_n+\va)^{p/2}\wedge (1+\va)^{p/2})d_{j}^\va(x,t)u_\va(x,t)\pa_j u_\va(x,t)\,\ud x \\
&\quad-2\int_{B_1^+}\int_{-1}^{t} ((x_n+\va)^{p/2}\wedge (1+\va)^{p/2})u_\va \pa_s d_{j}^\va  D_j u_\va\,\ud s\ud x\\
&\quad -2 \int_{B_1^+}\int_{-1}^{t} ((x_n+\va)^{p/2}\wedge (1+\va)^{p/2})d_{j}^\va\pa_s u_\va   D_j u_\va\,\ud s\ud x\quad\mbox{as }h\to 0.
\end{align*}
\endgroup
Using similar arguments, by sending $h\to 0$ in \eqref{eq:partialttestfunction}, and using  \eqref{eq:ellip2}, \eqref{eq:assumptioncoefficient2}, \eqref{eq:coer-app} (or \eqref{eq:coer-app2}) and H\"older's inequality, we have
\begingroup
\allowdisplaybreaks
\begin{align*}
&\int_{B_1^+\times(-1,t]} (x_n+\va)^{p} |\pa_s u_\va|^2 \,\ud x\ud s +\int_{B_1^+}  |\nabla u_\va(x,t)|^2\,\ud x \\
&\le  C \int_{B_1^+} (x_n+\va)^pu_\va(x,t)^2\,\ud x+ C \int_{B_1^+\times(-1,t]} [ |\nabla u_\va|^2 + (x_n+\va)^p(f_\va^2+u_\va^2)] \,\ud x\ud s.
\end{align*}
\endgroup
Then it follows from \eqref{eq:energysmoothcase} that
\begin{align}\label{eq:weakderivativeint}
\sup_{t\in[-1,0]}\int_{B_1^+}  |\nabla u_\va(x,t)|^2\,\ud x+ \int_{Q_1^+} (x_n+\va)^{p} |\pa_t u_\va|^2 \,\ud x\ud t \le C\int_{Q_1^+} (x_n+\va)^p f_\va^2\,\ud x\ud t. 
\end{align}
%Then, using the equation \eqref{eq:degiorgilinearappappendix2} and Hardy's inequality, we have
%\[
%\int_{B_1^+}|D_j(a_{ij}^\va(x,t) D_i u_\va(x,t))|^2\,\ud x\le C\left. \int_{B_1^+} [(x_n+\va)^{p} |\pa_t u_\va|^2+u_\va^2+|\nabla u_\va|^2+f_\va^2] \, \ud x \right\vert_t.
%&\le C\int_{Q_1^+} f_\va^2\,\ud x\ud s.
%\]
%By the $H^2$ regularity for elliptic equations, we have
%\begin{align*}
%\sum_{i,j=1}^n\int_{Q_1^+}|D_{ij} u_\va|^2\,\ud x\ud s\le & C \int_{Q_1^+} [(x_n+\va)^{p} |\pa_t u_\va|^2+u_\va^2+|\nabla u_\va|^2+f_\va^2] \, \ud x\ud s\\
%&\le C\int_{Q_1^+} f_\va^2\,\ud x\ud s.
%\end{align*}
%In summary, we obtained
%\begin{align}\label{eq:weakderivativeint}
%&\sup_{t\in[-1,0]}\int_{B_1^+}  |\nabla u_\va(x,t)|^2\,\ud x+ \sum_{i,j=1}^n\int_{Q_1^+}|D_{ij} u_\va|^2\,\ud x\ud s+ \int_{Q_1^+} (x_n+\va)^{p} |\pa_t u_\va|^2 \,\ud s\ud x \nonumber\\
%&\le C\int_{Q_1^+} f_\va^2\,\ud x\ud s. 
%\end{align}
Therefore, $\int_{(B_1^+\cap\{x_n>\delta\})\times(-1,0]} |\pa_t u_\va|^2 \le C(\delta)$ for every $\delta>0$. This implies the existence of weak derivative $\pa_t u$, and that $\pa_t u_\va$ weakly converges to $\pa_t u$ in $L^2((B_1^+\cap\{x_n>\delta\})\times(-1,0])$ for every $\delta$. Since
\begin{align*}
&\int_{Q_1^+\cap\{x_n>\delta\}} [(x_n+\va)^{p}-x_n^p] |\pa_t u_\va|^2 \,\ud x\ud t \to 0\mbox{ as }\va\to 0,\\
& \int_{Q_1^+\cap\{x_n>\delta\}} x_n^p |\pa_t u|^2 \,\ud x\ud t \le\liminf_{\va\to 0}  \int_{Q_1^+\cap\{x_n>\delta\}} x_n^p |\pa_t u_\va|^2 \,\ud x\ud t, 
\end{align*}
we have from \eqref{eq:weakderivativeint} by sending $\va\to 0$ that 
\[
\sup_{t\in[-1,0]}\int_{B_1^+}  |\nabla u(x,t)|^2\,\ud x+\int_{Q_1^+\cap\{x_n>\delta\}} x_n^p |\pa_t u|^2 \,\ud x\ud t \le  C\int_{Q_1^+} x_n^p f^2\,\ud x\ud t. 
\]
Then, \eqref{eq:energyestimateut} follows by sending $\delta \to 0$ and using the monotone convergence theorem.

Now let us use another approximation to remove the assume that $f\in L^2(Q_1^+,x^p\vee 1\ud x\ud t)$ and $\pa_t a\in L^q(Q_1^+,x_n^p\vee1\,\ud x\ud t)$. Let $a^\va\in C^2(\overline Q_1^+)$ be such that $a^\va\to a$ uniformly on $Q_1^+$ and $\pa_t a^\va\to \pa_t a$ in $L^q(Q_1^+,x_n^p\,\ud x\ud t)$, and $f_\va\in C^2_c(\overline Q_1^+)$ be such that $f_\va\to f$ in $L^2(Q_1^+,x_n^p\,\ud x\ud t)$. Then there exists a weak solution $u_\va\in \mathring V_2^{1,0}(Q^+_{1})$ to the  parabolic equation
\[
a^\va \cdot x_n^{p} \pa_t u_\va -D_j[a_{ij} D_i u_\va+(x_n^{p/2}\wedge 1)d_j u_\va]+(x_n^{p/2}\wedge 1)b_iD_i u_\va+c x_n^p u_\va+c_0 u_\va=x_n^pf_\va  \quad \mbox{in }Q_1^+
\]
with  $u_\va\equiv 0$ on $\pa_{pa} Q_1^+$. By the argument of Theorem \ref{thm:existenceofweaksolution}, $u_\va$ will converge to a  weak solution of \eqref{eq:linear-eq-app} with the full boundary condition $u\equiv 0$ on $\pa_{pa} Q_1^+$. By the same argument as above, one can show that
\[
\sup_{t\in(-1,0)}\int_{B_1^+} |\nabla u_\va(x,t)|^2\,\ud x+\int_{Q_1^+} x_n^{p}|\pa_t u_\va|^2\,\ud x\ud t\le C \int_{Q_1^+} x_n^pf_\va^2\,\ud x\ud t.
\] 
Then the conclusion follows by sending $\va\to 0$.
\end{proof}

\section{Boundedness of weak solutions}\label{sec:bound}

\subsection{A maximum principle}
Suppose
\begin{equation}\label{eq:assumptionf4}
F_1:=\|f\|_{L^{\frac{2q\chi}{q\chi+\chi-q}}(Q_1^+,x_n^p\ud x\ud t)}+\|f_0\|_{L^{\frac{2q\chi}{q\chi+\chi-q}}(Q_1^+)}+ \sum_{j=1}^n\|f_j\|_{L^{2q}(Q_1^+)}<\infty,
\end{equation}
where $q$ is the one in \eqref{eq:assumptioncoefficient}.

\begin{thm}\label{thm:aweakmaxprinciple}
Suppose $u\in C ([-1,0]; L^2(B_1^+,x_n^{p}\ud x)) \cap  L^2((-1,0];H_{0,L}^1(B_1^+))$ is a weak solution of \eqref{eq:linear-eq} with the partial boundary condition \eqref{eq:linear-eq-D}, where the coefficients of the equation satisfy \eqref{eq:rangep}, \eqref{eq:ellip2},  \eqref{eq:assumptioncoefficient} and \eqref{eq:assumptionf4}. Suppose that all $d_j=0$ and $c\ge 0$. Then 
\begin{equation*}
\|u\|_{L^\infty(Q_1^+)}\le \left\{
  \begin{aligned}
  & \sup_{\pa_{pa}Q_1^+}|u|+ CF_1 |Q_1^+|^{\frac{1}{2}(1-\frac{1}{q}-\frac{1}{\chi})},  \quad  &\mbox{for }   p\ge 0, \\
  &  \sup_{\pa_{pa}Q_1^+}|u|+ CF_1 |Q_1^+|_p^{\frac{1}{2}(1-\frac{1}{q}-\frac{1}{\chi})}, \quad   &\mbox{for }    -1<p<0,
\end{aligned}
\right.
\end{equation*}
where $|Q_1^+|$ denotes the Lebesgue measure of $Q_1^+$,  $|Q_1^+|_p$ is defined in \eqref{eq:pmeasure} in the below, and $C>0$ depends only on $\lambda,\Lambda,n,p$ and $q$.
\end{thm}

\begin{proof}
It follows from Lemma \ref{lem:Steklovapproximation} that \eqref{eq:uastestfunction} holds. Let $\varphi$ be the one there.  Using \eqref{eq:assumptioncoefficient}, Theorem \ref{thm:weightedsobolev}, Theorem \ref{thm:weightedsobolev2} and H\"older's inequality, one obtains for $p\ge 0$ that
\begin{align*}
\int_{-1}^s\int_{B_1^+} \varphi^2 ((|\partial_t a|+|c|)x_n^p+\sum_j b_j^2+|c_0|)\,\ud x\ud t&\le \va \|\varphi\|_{V_2(B_1^+\times(-1,s))}^2+C_\va \|\varphi\|_{L^2(B_1^+\times(-1,s))}^2,
\end{align*}
and for $-1<p<0$ that
\begin{align*}
&\int_{-1}^s\int_{B_1^+} \varphi^2 ((|\partial_t a|+|c|)x_n^p+\sum_j b_j^2+c)\,\ud x\ud t\\
&\le \va \|\varphi\|_{V_2(B_1^+\times(-1,s))}^2+C_\va \|\varphi\|_{L^2(B_1^+\times(-1,s),x_n^p\ud x\ud t)}^2.
\end{align*}
Similarly,  we have 
\begin{align*}
\int_{-1}^s\int_{B_1^+\cap\{u>k\}} \sum_j f_j^2\,\ud x\ud t&\le CF_1^2 |\{u>k\}\cap(B_1^+\times(-1,s))|^{1-\frac{1}{q}},\\
\int_{-1}^s\int_{B_1^+\cap\{u>k\}} f_0\varphi\,\ud x\ud t&\le  \va \|\varphi\|_{V_2(B_1^+\times(-1,s))}^2+C_\va F_1^2|\{u>k\}\cap(B_1^+\times(-1,s))|^{1-\frac{1}{q}},\\
\int_{-1}^s\int_{B_1^+\cap\{u>k\}} x_n^pf\varphi\,\ud x\ud t&\le  \va \|\varphi\|_{V_2(B_1^+\times(-1,s))}^2+C_\va F_1^2|\{u>k\}\cap(B_1^+\times(-1,s))|_p^{1-\frac{1}{q}},
\end{align*}
where
\begin{equation}\label{eq:pmeasure}
|E|_{p}=\int_{E} x_n^p\,\ud x\ud t\quad\mbox{for every measurable set }E\subset Q_1^+.
\end{equation}
By choosing $\va>0$ small, and using Theorem \ref{thm:weightedsobolev} and Theorem \ref{thm:weightedsobolev2}, we have for $p\ge 0$ that 
\[
\|\varphi\|_{L^{2\chi}(B_1^+\times(-1,s))}^2\le C \|\varphi\|_{L^2(B_1^+\times(-1,s))}^2+CF_1^2|\{u>k\}\cap(B_1^+\times(-1,s))|^{1-\frac{1}{q}},
\]
and for $-1<p<0$ that
\[
\|\varphi\|_{L^{2\chi}(B_1^+\times(-1,s),x_n^p\ud x\ud t)}^2\le C \|\varphi\|_{L^2(B_1^+\times(-1,s),x_n^p\ud x\ud t)}^2+CF_1^2|\{u>k\}\cap(B_1^+\times(-1,s))|_{p}^{1-\frac{1}{q}}.
\]

When $s+1$ is sufficiently small, we have for $p\ge 0$ that 
\[
C \|\varphi\|_{L^2(B_1^+\times(-1,s))}^2\le \frac12\|\varphi\|_{L^{2\chi}(B_1^+\times(-1,s))}^2,
\]
and for $-1<p<0$ that
\[
C \|\varphi\|_{L^2(B_1^+\times(-1,s),x_n^p\ud x\ud t)}^2\le \frac12\|\varphi\|_{L^{2\chi}(B_1^+\times(-1,s),x_n^p\ud x\ud t)}^2.
\]
Hence, we have 
\begin{align*}
\|(u-k)^+\|_{L^{2\chi}(B_1^+\times(-1,s))}^2&\le CF_1^2|\{u>k\}\cap(B_1^+\times(-1,s))|^{1-\frac{1}{q}}\quad\mbox{if }p\ge 0\\
\|(u-k)^+\|_{L^{2\chi}(B_1^+\times(-1,s),x_n^p\ud x\ud t)}^2&\le CF_1^2|\{u>k\}\cap(B_1^+\times(-1,s))|_{p}^{1-\frac{1}{q}}\quad\mbox{if }-1<p<0.
\end{align*}
For $h>k$, we have
\begin{align*}
\|(u-k)^+\|_{L^{2\chi}(B_1^+\times(-1,s))}^2\ge (h-k)^2|\{u>h\}\cap(B_1^+\times(-1,s))|^{\frac{1}{\chi}},\\
\|(u-k)^+\|_{L^{2\chi}(B_1^+\times(-1,s),x_n^p\ud x\ud t)}^2\ge (h-k)^2|\{u>h\}\cap(B_1^+\times(-1,s))|_{p}^{\frac{1}{\chi}}.
\end{align*}
Hence, if we denote 
\begin{equation*}
\psi(k)=\left\{
  \begin{aligned}
  &  |\{u>k\}\cap(B_1^+\times(-1,s))|,  \quad  &\mbox{for }   p\ge 0, \\
  &  |\{u>k\}\cap(B_1^+\times(-1,s))|_p, \quad   &\mbox{for }    -1<p<0,
\end{aligned}
\right.
\end{equation*}
then
\[
\psi(h)\le \frac{CF_1^{2\chi}}{(h-k)^{2\chi}}\psi(k)^{\beta},
\]
where $\beta=(1-\frac1q)\chi>1$ by the assumption of $q$. Define
\[
k_s=\sup_{\pa_{pa}Q_1^+}|u|+d-\frac{d}{2^s}.
\]
Then
\[
\psi(k_{s+1})\le \frac{CF_1^{2\chi}2^{2\chi}}{d^{2\chi}} (4^\chi)^s \psi(k_s)^\beta.
\]
Similar to \eqref{eq:nonlineariteration} and \eqref{eq:nonlineariteration2}, we can choose $C>0$ such that for $d\ge CF_1 |B_1^+\times(-1,s)|^{\frac{1}{2}(1-\frac{1}{q}-\frac{1}{\chi})}$, we have
\[
\psi\left(\sup_{\pa_{pa}Q_1^+}|u|+d\right)=0.
\]
That is,
\begin{equation*}
\sup_{B_1^+\times(-1,s)} u\le \left\{
  \begin{aligned}
  & \sup_{\pa_{pa}Q_1^+}|u|+ CF_1 |B_1^+\times(-1,s)|^{\frac{1}{2}(1-\frac{1}{q}-\frac{1}{\chi})},  \quad  &\mbox{for }   p\ge 0, \\
  &  \sup_{\pa_{pa}Q_1^+}|u|+ CF_1 |B_1^+\times(-1,s)|_p^{\frac{1}{2}(1-\frac{1}{q}-\frac{1}{\chi})}, \quad   &\mbox{for }    -1<p<0.
\end{aligned}
\right.
\end{equation*}
Keeping iterating for $s$ with a uniform step size, we obtain that
\begin{equation*}
\sup_{Q_1^+} u\le \left\{
  \begin{aligned}
  & \sup_{\pa_{pa}Q_1^+}|u|+ CF_1 |Q_1^+|^{\frac{1}{2}(1-\frac{1}{q}-\frac{1}{\chi})},  \quad  &\mbox{for }   p\ge 0, \\
  &  \sup_{\pa_{pa}Q_1^+}|u|+ CF_1 |Q_1^+|_p^{\frac{1}{2}(1-\frac{1}{q}-\frac{1}{\chi})}, \quad   &\mbox{for }    -1<p<0.
\end{aligned}
\right.
\end{equation*}
Applying the same result to the equation of $-u$, the conclusion follows.
\end{proof}

\subsection{A local maximum principle}
The following is the Caccioppoli inequality of weak solutions to \eqref{eq:linear-eq} and \eqref{eq:linear-eq-D}, which is the starting point of the De Giorgi iteration.
\begin{thm}\label{thm:caccipolli}
Suppose $u\in C ([-1,0]; L^2(B_1^+,x_n^{p}\ud x)) \cap  L^2((-1,0];H_{0,L}^1(B_1^+))$ is a weak solution of \eqref{eq:linear-eq} with the partial boundary condition \eqref{eq:linear-eq-D}, where the coefficients of the equation satisfy \eqref{eq:rangep}, \eqref{eq:ellip2},  \eqref{eq:assumptioncoefficient} and \eqref{eq:assumptionf4}. Let $x_0\in\pa' B_1^+$, $Q_{\rho,\tau}^+=B^+_\rho(x_0)\times(t_0,t_0+\tau]\subset Q_1^+$, $k\ge 0$, $\va\in(0,1]$, and $\xi\in  V_2^{1,1}(Q_{\rho,\tau}^+)$ such that $\xi=0$ on $\pa''B_\rho(x_0)\times (t_0,t_0+\tau]$ and  $0\le\xi\le 1$. Then 
\begin{align*}
&\max\left(\sup_{t\in(t_0,t_0+\tau)}\int_{B^+_\rho(x_0)}x_n^{p} a [\xi(u-k)^+]^2(x,t)\,\ud x, \lambda\iint_{Q_{\rho,\tau}^+} |D[\xi(u-k)^+]|^2\,\ud x\ud s\right)\\
&\le (1+\va)\int_{B^+_\rho(x_0)}x_n^{p} a [\xi(u-k)^+]^2(x,t_0)\,\ud x +C\iint_{Q_{\rho,\tau}^+} (|D\xi|^2+|\xi\pa_t\xi|x_n^p) [(u-k)^+]^2 \,\ud x\ud t \\
&\quad +\frac{C}{\va^\kappa}\left( \|[(u-k)^+] \xi\|_{L^2(Q_{\rho,\tau}^+)}^2 +  \|[(u-k)^+] \xi\|_{L^2(Q_{\rho,\tau}^+,  x_n^p\ud x\ud t)}^2\right)\\
&\quad +\frac{C}{\va^\kappa}(k^2+F_1^2)\left( |\{u>k\}\cap Q_{\rho,\tau}^+|^{1-\frac1q}+|\{u>k\}\cap Q_{\rho,\tau}^+|_p^{1-\frac1q} \right),
\end{align*}
where $|\cdot|_p$ is defined in \eqref{eq:pmeasure}, $C>0$ depends only on $\lambda,\Lambda,n,p$ and $q$, $\kappa>0$ depends only on $q$ and $\chi$, and $F_1$ is given in \eqref{eq:assumptionf4}.
\end{thm}

\begin{proof}
Similar to the proof of  Lemma \ref{lem:Steklovapproximation}, we can assume $u\in V^{1,1}_2(Q^+_1)$, since otherwise we can use its Steklov average to remove this assumption. 

Taking $\varphi=\xi^2(u-k)^+$ in \eqref{eq:definitionweaksolution}, and using \eqref{eq:ellip2} and H\"older's inequality, one obtains that

\begingroup
\allowdisplaybreaks
\begin{align}
&\frac{1}{2}\int_{B^+_\rho(x_0)}x_n^{p} a [\xi(u-k)^+]^2(x,t)\,\ud x-\frac{1}{2}\int_{B^+_\rho(x_0)}x_n^{p} a [\xi(u-k)^+]^2(x,t_0)\,\ud x\nonumber\\
&+\frac{3\lambda}{4} \int_{t_0}^t\int_{B_\rho^+(x_0)} |D(\xi(u-k)^+)|^2\,\ud x\ud s\nonumber\\
&\le C\int_{t_0}^t\int_{B_\rho^+(x_0)} (|D\xi|^2+|\xi\pa_t\xi|x_n^p) [(u-k)^+]^2 \,\ud x\ud s  \nonumber\\
&\quad +C\int_{t_0}^t\int_{B_\rho^+(x_0)} [(u-k)^+]^2 \xi^2 \left( (|\pa_t a|+|c|) x_n^p +\sum_j d_j^2+\sum_j b_j^2+ |c_0| \right)\,\ud x\ud s  \nonumber\\
&\quad + C\int_{t_0}^t\int_{B_\rho^+(x_0)}  \xi^2k^2\chi_{\{u>k\}} \left(|c|x_n^p+|c_0|+\sum_j d_j^2\right)\,\ud x\ud s\nonumber\\
&\quad+C\int_{t_0}^t\int_{B_\rho^+(x_0)\cap\{u>k\}} \left(|f|x_n^p\xi^2(u-k)^++ |f_0|\xi^2(u-k)^++\xi^2\sum_j f_j^2\right)\,\ud x\ud s. \label{eq:uastestfunction2}
\end{align}
\endgroup

Using \eqref{eq:assumptioncoefficient}, Theorem \ref{thm:weightedsobolev}, Theorem \ref{thm:weightedsobolev2} and H\"older's inequality one obtains
\begin{align*}
&\int_{t_0}^t\int_{B_\rho^+(x_0)} [(u-k)^+]^2 \xi^2  ((|\partial_t a|+|c|)x_n^p+\sum_j b_j^2+|c_0|+\sum_j d_j^2)\,\ud x\ud s\\
&\quad \le \va \|[(u-k)^+] \xi\|_{V_2(Q_{\rho,\tau}^+)}^2+\frac{C}{\va^\kappa} \|[(u-k)^+] \xi\|_{L^2(Q_{\rho,\tau}^+)}^2+\frac{C}{\va^\kappa} \|[(u-k)^+] \xi\|_{L^2(Q_{\rho,\tau}^+, x_n^p\ud x\ud t)}^2,\\
&k^2\iint_{Q_{\rho,\tau}^+} \chi_{\{u>k\}}\Big[|c|x_n^p+|c_0|+\sum_j d_j^2\Big]\,\ud x\ud t\\
&\le C k^2 |\{u>k\}\cap Q_{\rho,\tau}^+|^{1-\frac1q}+C k^2 |\{u>k\}\cap Q_{\rho,\tau}^+|_p^{1-\frac1q},
\end{align*}
and
\begin{align*}
\iint_{Q_{\rho,\tau}^+\cap\{u>k\}} \sum_j f_j^2\,\ud x\ud t&\le F_1^2 |\{u>k\}\cap Q_{\rho,\tau}^+)|^{1-\frac{1}{q}},\\
\iint_{Q_{\rho,\tau}^+\cap\{u>k\}} |f_0|\xi^2(u-k)^+\,\ud x\ud t&\le  \va \|\varphi\|_{V_2(Q_{\rho,\tau}^+)}^2+\frac{C}{\va^\kappa} F_1^2|\{u>k\}\cap Q_{\rho,\tau}^+|^{1-\frac{1}{q}},\\
\iint_{Q_{\rho,\tau}^+\cap\{u>k\}} |f|x_n^p\xi^2(u-k)^+\,\ud x\ud t&\le  \va \|\varphi\|_{V_2(Q_{\rho,\tau}^+)}^2+\frac{C}{\va^\kappa} F_1^2|\{u>k\}\cap Q_{\rho,\tau}^+|_p^{1-\frac{1}{q}}.
\end{align*}
By choosing $\va>0$ small, and using Theorem \ref{thm:weightedsobolev} and Theorem \ref{thm:weightedsobolev2}, the conclusion follows from \eqref{eq:uastestfunction2}.
\end{proof}

Now we can prove the local-in-time boundedness of weak solutions up to $\{x_n=0\}$.
\begin{thm}\label{thm:localboundedness}
Suppose $u\in C ([-1,0]; L^2(B_1^+,x_n^{p}\ud x)) \cap  L^2((-1,0];H_{0,L}^1(B_1^+))$ is a weak solution of \eqref{eq:linear-eq} with the partial boundary condition \eqref{eq:linear-eq-D}, where the coefficients of the equation satisfy \eqref{eq:rangep}, \eqref{eq:ellip2}, and 
\begin{align}
\Big\||\pa_t a|+|c|\Big\|_{L^q(Q_1^+,x_n^p\ud x\ud t)}+\Big\|\sum_{j=1}^n(b_j^2+d_j^2)+|c_0|\Big\|_{L^q(Q_1^+)}&\le\Lambda \label{eq:localbddnesscoeffcient}\\
F_1:=\|f\|_{L^{\frac{2q\chi}{q\chi+\chi-q}}(Q_1^+,x_n^p\ud x\ud t)}+\|f_0\|_{L^{\frac{2q\chi}{q\chi+\chi-q}}(Q_1^+)}+ \sum_{j=1}^n\|f_j\|_{L^{2q}(Q_1^+)}&<\infty\label{eq:localbddnessf}
\end{align}
for some $q>\max(\frac{\chi}{\chi-1},\frac{n+p+2}{2}, \frac{n+2p+2}{p+2})$.  Denote $\mathcal{Q}_R^+=B_R^+(x_0)\times(t_0-R^{p+2},t_0]\subset Q_1^+$, where $x_0\in\pa' B_1^+$. Then we have, for any $\gamma>0$,
\[
\|u\|_{L^\infty(\mathcal Q_{R/2}^+)}\le C\left(R^{-\frac{n+p+2}{\gamma}}\|u\|_{L^\gamma(\mathcal Q_R^+)}+F_1 R^{1-\frac{n+p+2}{2q}} \right),
\]
where $C>0$ depends only on $\lambda,\Lambda,n,p$, $q$ and $\gamma$.
\end{thm}
\begin{proof}
Let $\theta\in(0,1)$. We will consider $R=1$ first, and then scale it back. %Let $Q_R^+=B_R^+(x_0)\times(t_0-R^2,t_0]\subset Q_1^+$. 
We would like to first show that
\begin{align}\label{eq:boundednessbyL2}
\|u\|_{L^\infty(\mathcal{Q}_{\theta}^+)}\le C\left[(1-\theta)^{-1/\beta}\Big(\|u\|_{L^2(Q_1^+)}+ \|u\|_{L^2(Q_1^+,\ x_n^p\,\ud x\ud t)}\Big)+F_1\right],
\end{align}
where $\beta=1-\frac{1}{q}-\frac{1}{\chi}$. We only need to prove \eqref{eq:boundednessbyL2} for $\theta\in(1/2,1)$.

Let 
$$
\rho_m=\theta +2^{-m}(1-\theta), \quad k_m=k(2-2^{-m}),\quad m=0,1,2,\cdots,
$$
where $k>0$ to be fixed later. For brevity, we denote $Q_m^+=Q^+_{\rho_m}=B_{\rho_m}^+(x_0)\times(t_0-\rho_m^{p+2},t_0]$, and we take $\xi_m$ to be a cut-off function such that $\xi\in V_2^{1,1}(Q_{m}^+)$, $\xi_m\equiv 1$ on $Q_{m+1}^+$,  $\xi=0$ on $Q_{m}^+\setminus Q^+_{(\rho_m+\rho_{m+1})/2}$, and $|D\xi_m|^2+|\pa_t\xi_m|\le C(n) (\rho_m-\rho_{m+1})^{-2}.$ 

Let 
\begin{equation}\label{2}
\varphi_m=\left\{
  \begin{aligned}
  &  \|(u-k_m)^+\|_{L^2(Q_m^+)}^2  \quad  &\mbox{if }   p\ge 0, \\
  &  \|(u-k_m)^+\|_{L^2(Q_m^+,x_n^p\ud x\ud t)}^2 \quad   &\mbox{if }    -1<p<0.
\end{aligned}
\right.
 \end{equation}

Case 1: Suppose $p\ge 0$. By Theorem \ref{thm:caccipolli} and Theorem \ref{thm:weightedsobolev}, we have
\begin{align*}
& \|\xi_m(u-k_{m+1})^+\|^2_{L^{2\chi}(Q_m^+)}\\
&\le C  \|\xi_m(u-k_{m+1})^+\|^2_{V_2(Q_m^+)}\\
&\le \frac{C2^{2m}}{(1-\theta)^2}\|(u-k_{m+1})^+\|_{L^2(Q_{m}^+)}^2 +C(k^2+F_1^2)|A_m(k_{m+1})|^{1-\frac1q},
\end{align*}
where
\[
A_m(k)=\{u>k\}\cap Q_{m}^+,\ \mbox{ and }|A_m(k)| \mbox{ is the Lebesgue measure of }A_m(k).
\]
Take $k\ge F_1 $. Then,
\begin{align*}
\varphi_{m+1}&\le \|\xi_m(u-k_{m+1})^+\|^2_{L^{2}(Q_m^+)}\\
& \le \|\xi_m(u-k_{m+1})^+\|^2_{L^{2\chi}(Q_m^+)} |A_m(k_{m+1})|^{1-\frac{1}{\chi}}\\
&\le \frac{C2^{2m}}{(1-\theta)^2} \varphi_m |A_m(k_{m+1})|^{1-\frac{1}{\chi}}+ Ck^2|A_m(k_{m+1})|^{2-\frac{1}{\chi}-\frac{1}{q}}.
\end{align*}
Notice that
\[
\varphi_m=\|(u-k_m)^+\|_{L^2(Q_m^+)}^2 \ge (k_{m+1}-k_m)^2|A_m(k_{m+1})|= \frac{k^2}{2^{2m+2}} |A_m(k_{m+1})|.
\]
Hence,
\begin{equation}\label{eq:iterationstep0}
\varphi_{m+1}\le \frac{C2^{2m}}{(1-\theta)^2}  \left(\frac{2^{2m+2}}{k^2}\right)^{1-\frac{1}{\chi}}\varphi_m^{2-\frac{1}{\chi}}+ Ck^2 \left(\frac{2^{2m+2}}{k^2}\right)^{2-\frac{1}{\chi}-\frac1q}\varphi_m^{2-\frac{1}{\chi}-\frac1q}.
\end{equation}

Case 2: Suppose $-1<p<0$. By Theorem \ref{thm:caccipolli} and Theorem \ref{thm:weightedsobolev2}, we have
\begin{align*}
& \|\xi_m(u-k_{m+1})^+\|^2_{L^{2\chi}(Q_m^+,\ x_n^p\ud x\ud t)}\\
&\le C  \|\xi_m(u-k_{m+1})^+\|^2_{V_2(Q_m^+)}\\
&\le \frac{C2^{2m}}{(1-\theta)^2}\|(u-k_{m+1})^+\|_{L^2(Q_{m}^+,\ x_n^p\ud x\ud t)}^2 +C(k^2+F_1^2)|A_m(k_{m+1})|_{p}^{1-\frac1q},
\end{align*}
where
\[
A_m(k)=\{u>k\}\cap Q_{m}^+,\ \mbox{ and }|\cdot|_{p} \mbox{ is defined in }\eqref{eq:pmeasure}.
\]
Take $k\ge F_1 $. Then,
\begin{align*}
\varphi_{m+1}&\le \|\xi_m(u-k_{m+1})^+\|^2_{L^{2}(Q_m^+,\ x_n^p\ud x\ud t)}\\
& \le \|\xi_m(u-k_{m+1})^+\|^2_{L^{2\chi}(Q_m^+,\ x_n^p\ud x\ud t)} |A_m(k_{m+1})|_{p}^{1-\frac{1}{\chi}}\\
&\le \frac{C2^{2m}}{(1-\theta)^2} \varphi_m |A_m(k_{m+1})|_{p}^{1-\frac{1}{\chi}}+ Ck^2|A_m(k_{m+1})|_{p}^{2-\frac{1}{\chi}-\frac{1}{q}}.
\end{align*}
Notice that
\[
\varphi_m=\|(u-k_m)^+\|_{L^2(Q_m^+,\ x_n^p\ud x\ud t)}^2 \ge (k_{m+1}-k_m)^2|A_m(k_{m+1})|_{p}= \frac{k^2}{2^{2m+2}} |A_m(k_{m+1})|_{p}.
\]
Hence, \eqref{eq:iterationstep0} also holds.

Now let us start from \eqref{eq:iterationstep0} which holds for all $p>-1$. If we further take $k\ge \|u\|_{L^2(Q_1^+)} + \|u\|_{L^2(Q_1^+,\ x_n^p\,\ud x\ud t)} $, then
\[
y_m:=\frac{\varphi_m}{k^2 }\le 1.
\]
Thus,
\[
y_{m+1}\le \frac{C 2^{2m(2-\frac{1}{\chi})}y_m^{1+\beta}}{(1-\theta)^2}.
\]
If
\begin{equation}\label{eq:nonlineariteration}
y_0=\frac{\|(u-k)^+\|_{L^2(Q_1^+)}^2}{k^2}\le \frac{\|u\|_{L^2(Q_1^+)}^2}{k^2}\le \bar y= \frac{(1-\theta)^{2/\beta}}{C^{1/\beta}} 4^{(\frac{1}{\chi}-2)\frac{1}{\beta^2}},
\end{equation}
then one can show by induction that 
\[
y_m\le \frac{\bar y}{(4^{2-\frac{1}{\chi}})^{\frac{m}{\beta}}},
\]
and thus, 
\begin{equation}\label{eq:nonlineariteration2}
\lim_{m\to\infty}y_m=0.
\end{equation} 
That is,
\[
\sup_{Q_{1/2}^+}u \le 2k.
\]
Therefore, we  only need to choose 
\[
k=F_1+\frac{C}{(1-\theta)^{1/\beta}}(\|u\|_{L^2(Q_1^+)}+ \|u\|_{L^2(Q_1^+,\ x_n^p\,\ud x\ud t)} ).
\]
This proves \eqref{eq:boundednessbyL2}.

Now we will use a scaling argument. For any $R\in(0,1]$, define
\begin{equation}\label{eq:rescaledequationcoefficients}
\begin{split}
\tilde u(x,t)&=u(x_0+Rx,t_0+R^{p+2}t),\quad \tilde a(x,t)=a(x_0+Rx,t_0+R^{p+2}t),\\
 \tilde a_{ij}(x,t)&=a_{ij}(x_0+Rx,t_0+R^{p+2}t), \quad \tilde d_j(x,t)=d_j(x_0+Rx,t_0+R^{p+2}t), \\ 
 \tilde b_j(x,t)&=b_j(x_0+Rx,t_0+R^{p+2}t), \quad \tilde c_0(x,t)=c_0(x_0+Rx,t_0+R^{p+2}t),\\
  \tilde c(x,t)&=c(x_0+Rx,t_0+R^{p+2}t), \quad \tilde f(x,t)=f(x_0+Rx,t_0+R^{p+2}t),\\
 \tilde f_j(x,t)&=f_j(x_0+Rx,t_0+R^{p+2}t), \quad \tilde f_0(x,t)=f_0(x_0+Rx,t_0+R^{p+2}t).
 \end{split}
\end{equation}
Then
\begin{equation}\label{eq:rescaledequation}
\begin{split}
& \tilde ax_n^{p} \pa_t  \tilde u-D_j( \tilde a_{ij} D_i  \tilde u+ R\tilde d_j  \tilde u)+ R\tilde b_iD_i  \tilde u+ R^{p+2}\tilde c x_n^p \tilde u + R^2\tilde c_0 \tilde u\\
 &\quad= R^{p+2}x_n^pf+R^2\tilde f_0 -RD_i \tilde f_i \quad \mbox{in }Q_1^+.
 \end{split}
\end{equation}
Note that since $q>\max(\frac{n+p+2}{2}, \frac{n+2p+2}{p+2})$, then
\begin{equation}\label{eq:rescaledequationcoefficients1}
\begin{split}
&\Big\||\pa_t  \tilde a|+R^{p+2}|\tilde c|\Big\|_{L^q(Q_1^+,x_n^p\ud x\ud t)}+ \Big\|\sum_{j=1}^n(R^2 \tilde b_j^2+R^2 \tilde d_j^2)+R^2| \tilde c_0|\Big\|_{L^q(Q_1^+)}\\
&\le [R^{p+2-\frac{n+2p+2}{q}}+R^{2-\frac{n+p+2}{q}}] \Lambda\le\Lambda,
 \end{split}
\end{equation}
and
\begin{equation}\label{eq:rescaledequationcoefficients2}
\begin{split}
&R^{p+2}\|f\|_{L^{\frac{2q\chi}{q\chi+\chi-q}}(Q_1^+,x_n^p\ud x\ud t)}+R^2\|\tilde f_0\|_{L^{\frac{2q\chi}{q\chi+\chi-q}}(Q_1^+)}+ R\sum_{j=1}^n\|\tilde f_j\|_{L^{2q}(Q_1^+)}\\
&\le CR^{1-\frac{n+p+2}{2q}}  F_1 \le CF_1.
 \end{split}
\end{equation}
Hence, it follows from \eqref{eq:boundednessbyL2} that
\begin{align*}
&\|\tilde u\|_{L^\infty(\mathcal{Q}_{\theta}^+)}\le C\left((1-\theta)^{-1/\beta}(\|\tilde u\|_{L^2(Q_1^+)}+\|\tilde u\|_{L^2(Q_1^+,\ x_n^p\ud x\ud t)})+F_1\right),
\end{align*}
where in the second inequality we used that $q>\max(\frac{\chi}{\chi-1},\frac{n+p+2}{2})$ and $R\le 1$. %Note that $\mathcal{Q}_{\theta R}^+\subset B_{\theta R}^+(x_0)\times(t_0-\theta^2R^{p+2},t_0]$. 
Scaling the estimate of $\tilde u$ back to $u$, we then obtain  for $p\ge 0$ that
\begin{align*}
\|u\|_{L^\infty(\mathcal{Q}_{\theta R}^+)} &\le \|\tilde u\|_{L^\infty(\mathcal{Q}_\theta^+)}\\
&\le C\left(\frac{1}{R^{\frac{n+p+2}{2}}(1-\theta)^{1/\beta}}\| u\|_{L^2(\mathcal{Q}_R^+)}+ F_1\right)\\
&\le C\left(\frac{1}{R^{\frac{n+p+2}{2}}(1-\theta)^{1/\beta}}\| u\|_{L^\infty(\mathcal{Q}_R^+)}^{\frac{2-\gamma}{2}} \| u\|^\frac{\gamma}{2}_{L^\gamma(\mathcal{Q}_R^+)}+F_1\right)\\
&\le \frac{1}{2} \|u\|_{L^\infty(\mathcal{Q}_{ R}^+)} + \frac{C}{R^{\frac{n+p+2}{\gamma}}(1-\theta)^{\frac{2}{\beta\gamma}}} \| u\|_{L^\gamma(\mathcal{Q}_R^+)}+C F_1\\
&\le \frac{1}{2} \|u\|_{L^\infty(\mathcal{Q}_{ R}^+)} + \frac{C}{(R-\theta R)^{\alpha}} \| u\|_{L^\gamma(\mathcal{Q}_R^+)}+CF_1,
\end{align*}
where $\alpha= \max(\frac{n+p+2}{\gamma},\frac{2}{\beta\gamma})$. 
By an iterative lemma, Lemma 1.1 in Giaquinta-Giusti \cite{GG}, we have
\[
\|u\|_{L^\infty(\mathcal{Q}_{\theta }^+)} \le  \frac{C}{(1-\theta )^{\alpha}} \| u\|_{L^\gamma(\mathcal{Q}_1^+)}+CF_1.
\]
Applying this estimate to $\tilde u$ again, and scaling it back to $u$, we obtain the desired estimate.

Similarly, for $-1<p<0$ and $\gamma\in (0,2)$, we let $\tilde\gamma=\frac{(1+p)\gamma}{1-p}<\gamma$. Then we have
\begingroup
\allowdisplaybreaks
\begin{align*}
\|u\|_{L^\infty(\mathcal{Q}_{\theta R}^+)} &\le \|\tilde u\|_{L^\infty(\mathcal{Q}_\theta^+)}\\
&\le C\left(\frac{1}{R^{\frac{n+2p+2}{2}}(1-\theta)^{1/\beta}}\| u\|_{L^2(\mathcal{Q}_R^+,\ x_n^p\ud x\ud t)}+  F_1\right)\\
&\le C\left(\frac{1}{R^{\frac{n+2p+2}{2}}(1-\theta)^{1/\beta}}\| u\|_{L^\infty(\mathcal{Q}_R^+)}^{\frac{2-\tilde\gamma}{2}} \| u\|^\frac{\tilde\gamma}{2}_{L^{\tilde\gamma}(\mathcal{Q}_R^+,\ x_n^p\ud x\ud t)}+  F_1\right)\\
&\le \frac{1}{2} \|u\|_{L^\infty(\mathcal{Q}_{ R}^+)} + \frac{C}{R^{\frac{n+2p+2}{\tilde\gamma}}(1-\theta)^{\frac{2}{\beta\tilde\gamma}}} \| u\|_{L^{\tilde\gamma}(\mathcal{Q}_R^+,\ x_n^p\ud x\ud t)}+C  F_1\\
&\le \frac{1}{2} \|u\|_{L^\infty(\mathcal{Q}_{ R}^+)} + \frac{C}{(R-\theta R)^{\tilde \alpha}} \| u\|_{L^{\tilde\gamma}(\mathcal{Q}_R^+,\ x_n^p\ud x\ud t)}+CF_1,
\end{align*}
\endgroup
where $\tilde \alpha= \max(\frac{n+2p+2}{\tilde\gamma},\frac{2}{\beta\tilde\gamma})$. 
By an iterative lemma, Lemma 1.1 in Giaquinta-Giusti \cite{GG},  we have
\[
\|u\|_{L^\infty(\mathcal{Q}_{\theta }^+)} \le  \frac{C}{(1-\theta )^{\tilde \alpha}} \| u\|_{L^{\tilde\gamma}(\mathcal{Q}_1^+,\ x_n^p\ud x\ud t)}+CF_1.
\]
Since
\[
\| u\|_{L^{\tilde\gamma}(\mathcal{Q}_1^+,\ x_n^p\ud x\ud t)} \le C \| u\|_{L^\gamma(\mathcal{Q}_1^+)},
\]
which follows from H\"older's inequality, then
\[
\|u\|_{L^\infty(\mathcal{Q}_{\theta }^+)} \le  \frac{C}{(1-\theta )^{\tilde \alpha}} \| u\|_{L^\gamma(\mathcal{Q}_1^+)}+CF_1.
\]
Applying this estimate to $\tilde u$ again, and scaling it back to $u$, we obtain the desired estimate.
\end{proof}

If it additionally satisfies $u(\cdot,-1)=0$, then we can show the boundedness up to the initial time.
\begin{thm}\label{thm:localboundednesstobottom}
Suppose $u\in C ([-1,0]; L^2(B_1^+,x_n^{p}\ud x)) \cap  L^2((-1,0];H_{0,L}^1(B_1^+))$ is a weak solution of \eqref{eq:linear-eq} with the partial boundary condition \eqref{eq:linear-eq-D} and the initial condition $u(\cdot,-1)=0$, where the coefficients of the equation satisfy \eqref{eq:rangep}, \eqref{eq:ellip2}, \eqref{eq:localbddnesscoeffcient} and \eqref{eq:localbddnessf} for some $q>\max(\frac{\chi}{\chi-1},\frac{n+p+2}{2},\frac{n+2p+2}{p+2})$.  Denote $\widetilde{\mathcal{Q}}_R^+=B_R^+(x_0)\times(-1,-1+R^{p+2}]\subset Q_1^+$, where $x_0\in\pa' B_1^+$. Then we have, for any $\gamma>0$,
\[
\|u\|_{L^\infty(\widetilde{\mathcal{Q}}_{R/2}^+)}\le C\left(R^{-\frac{n+p+2}{\gamma}}\|u\|_{L^\gamma(\widetilde{\mathcal{Q}}_R^+)}+F_1 R^{1-\frac{n+p+2}{2q}} \right),
\]
where $C>0$ depends only on $\lambda,\Lambda,n,p$, $q$ and $\gamma$.
\end{thm}
The proof of this theorem is almost identical to that of Theorem \ref{thm:localboundedness} (which is actually simpler since we do not need to cut off in the time variables). We omit the details.

Combining Theorems \ref{thm:localboundedness} and \ref{thm:localboundednesstobottom}, we have
\begin{thm}\label{thm:localboundednessglobal}
Suppose $u\in C ([-1,0]; L^2(B_1^+,x_n^{p}\ud x)) \cap  L^2((-1,0];H_{0,L}^1(B_1^+))$ is a weak solution of \eqref{eq:linear-eq} with the partial boundary condition \eqref{eq:linear-eq-D} and the initial condition $u(\cdot,-1)=0$, where the coefficients of the equation satisfy \eqref{eq:rangep}, \eqref{eq:ellip2}, \eqref{eq:localbddnesscoeffcient} and \eqref{eq:localbddnessf} for some $q>\max(\frac{\chi}{\chi-1},\frac{n+p+2}{2},\frac{n+2p+2}{p+2})$.   Then we have, for any $\gamma>0$,
\[
\|u\|_{L^\infty(B_{1/2}^+\times(-1,0])}\le C\left(\|u\|_{L^\gamma(\mathcal{Q}_1^+)}+F_1 \right),
\]
where $C>0$ depends only on $\lambda,\Lambda,n,p$, $q$ and $\gamma$.
\end{thm}
\begin{proof}
It follows from Theorems \ref{thm:localboundedness} and \ref{thm:localboundednesstobottom}.\end{proof}

\section{H\"older regularity}\label{sec:holderregularity}

\subsection{Improvement of oscillations centered at the boundary}\label{subsec:holderregularity}
Throughout this subsection, we assume all the assumptions in Theorem \ref{thm:localboundedness} and let $u$ be as in Theorem \ref{thm:localboundedness}.  Suppose 
\[
 M:=\|u\|_{L^\infty(B_{3/4}^+\times(-3/4,0])}.
\]
Let
\[
\ud \mu_p(t)=a(x,t)x_n^p\,\ud x,\quad \ud \nu_p=a(x,t)x_n^p\,\ud x\ud t
\]
and
\[
|A|_{\mu_p(t)}= \int_{A}a(x,t)x_n^p\,\ud x\quad\mbox{for } A\subset B_1^+,\quad |\widetilde A|_{\nu_p}= \int_{\widetilde A}a(x,t)x_n^p\,\ud x\ud t\quad\mbox{for }  \widetilde A\subset Q_1^+.
\]
Recall that for $x_0\in\partial R^n_+$, 
$$\mathcal{Q}_R(x_0,t_0):=B_R(x_0)  \times (t_0-R^{p+2}, t_0], \quad \mathcal{Q}_R^+(x_0,t_0):=B_R^+(x_0)  \times (t_0-R^{p+2}, t_0].$$ We simply write it as $\mathcal{Q}_R$ and $\mathcal{Q}_R^+$ if $(x_0,t_0)=(0,0)$.  

\begin{lem}\label{lem:measurechange}
There exists $C>0$ depending only on $n$ and $p$ such that for every $\va\in(0,1)$, every $R\in(0,1]$, every $\delta>0$, every $\widetilde A\subset B_R^+\times [0,\delta R^{p+2}]$, if 
\[
\frac{|\widetilde A|_{\nu_p}}{|B_R^+\times [0,\delta R^{p+2}]|_{\nu_p}} \le\va^{p+1},
\]
then
\[
\frac{|\widetilde A|}{|B_R^+\times [0,\delta R^{p+2}]|}\le C\max(\va^{p+1}, \va).
\]
\end{lem}
\begin{proof}
If $-1<p<0$, then 
\[
\va^{p+1}\ge \frac{|\widetilde A|_{\nu_p}}{|B_R^+\times [0,\delta R^{p+2}]|_{\nu_p}} \ge  \frac{\lambda R^p |\widetilde A|}{C R^{n+2p+2}} \ge \frac{1}{C} \frac{|\widetilde A|}{|B_R^+\times [0,\delta R^{p+2}]|}.
\]
If $p\ge 0$, then we have
\begin{align*}
\frac{|\widetilde A|}{|B_R^+\times [0,\delta R^{p+2}]|}&\le \frac{C}{\delta R^{n+2+p}}\left(\int_{\widetilde A\cap\{x_n\le \va R\}}\,\ud x\ud t+\int_{\widetilde A\cap\{x_n> \va R\}}\,\ud x\ud t\right)\\
&\le \frac{C}{\delta R^{n+2+p}}\left(\int_{\widetilde A\cap\{x_n\le \va R\}}\,\ud x\ud t+\frac{1}{\va^pR^p}\int_{\widetilde A\cap\{x_n> \va R\}}x_n^p\,\ud x\ud t\right)\\
&\le \frac{C}{\delta R^{n+2+p}}\left(\va \delta R^{n+2+p}+\frac{C \va^{p+1}\delta R^{n+2+2p}}{\va^pR^p}\right)\\
&=C\va.
\end{align*}

\end{proof}

We have the following De Giorgi lemmas.

\begin{lem}\label{lem:smallonlargeset} 
Let $0<R\le 1$ and 
\[
0\le \sup_{\mathcal{Q}_R^+} u\le \mu\le  M.
\]
Then there exists  $0<\gamma_0<1$ depending only on  $\lambda,\Lambda,n,p$ and $q$ such that for $0\le k<\mu$, if
\begin{equation*}
H:=\mu-k>
\left\{
  \begin{aligned}
  &  (M+F_1)R^{1-\frac{n+p+2}{2q}},  \quad  &\mbox{for }   p\ge 0, \\
  &  (M+F_1)R^{\frac{p+2}{2}-\frac{n+2p+2}{2q}}, \quad   &\mbox{for }    -1<p<0,
\end{aligned}
\right.
\end{equation*}
and
\[
\frac{|\{(x,t)\in \mathcal{Q}_R^+: u(x,t)>k\} |_{\nu_{p}}}{|\mathcal{Q}_R^+|_{\nu_{p}}} \le \gamma_0,
\]
then
\[
u\le \mu-\frac{H}{2}\quad\mbox{in }\mathcal{Q}_{R/2}^+.
\]
\end{lem}
\begin{proof}
Let
\[
r_j=\frac {R}2+\frac{R}{2^{j+1}},\quad k_j=\mu-\frac{H}{2}- \frac{H}{2^{j+1}},\quad j=0,1,2,\cdots.
\]
Let $\eta_j $ be a smooth cut-off function satisfying 
\[
\mbox{supp}(\eta_j) \subset \mathcal{Q}_{r_j}, \quad 0\le \eta_j \le 1, \quad \eta_j=1 \mbox{ in }\mathcal{Q}_{r_{j+1}}^+, 
\]
\[
|D \eta_j(x,t)|^2+|\pa_t \eta_j(x,t)| R^{p}  \le \frac{C(n)}{(r_j-r_{j+1})^2} \quad \mbox{in }\mathcal{Q}_R^+. 
\]

Case 1: $p\ge 0$. Let us consider $n\ge 3$ first. Since $k_j>k\ge 0$, By Theorem \ref{thm:caccipolli} and Theorem \ref{thm:weightedsobolev}, we have
\begin{equation}\label{eq:auxgiorgi1}
 \Big(\int_{\mathcal{Q}_R^+}  |\eta_j v|^{2\chi}\,\ud x \ud t \Big)^{\frac{1}{\chi}} \le C\left[ \frac{2^{2j}}{R^2} \|v\|^2_{L^2(\mathcal{Q}_{r_{j}}^+)} + (M+F_1)^2 |\mathcal{Q}_{r_{j}}^+\cap \{u>k_j\}|^{1-\frac1q}\right],
\end{equation}
where $v=(u-k_j)^+$. Let $A(k_j,\rho)= \{(x,t)\in \mathcal{Q}_{\rho}^+: u> k_j\}$  for $0<\rho\le R$. Then 
\[
 \Big(\int_{\mathcal{Q}_R^+}  |\eta_j v|^{2\chi}\,\ud x \ud t \Big)^{\frac{1}{\chi}} \ge (k_{j+1}- k_j)^2 |A(k_{j+1}, r_{j+1})|^{\frac{1}{\chi}},  
\]
and
\[
\int_{\mathcal{Q}_{r_j}^+}  v^{2} \,\ud x\ud t \le H^2 |A(k_{j}, r_{j})|. 
\]
It follows that 
\begin{align*}
|A(k_{j+1}, r_{j+1})| & \le  C\left[ \frac{2^{4j}}{R^2}|A(k_{j}, r_{j})| + \frac{2^{2j}(M+F_1)^2}{H^2} |A(k_{j}, r_{j})|^{1-\frac1q}\right]^\chi\\
&\le  C\left[ \frac{2^{4j}}{R^2}|A(k_{j}, r_{j})| + \frac{2^{2j}}{R^{2-\frac{n+p+2}{q}}} |A(k_{j}, r_{j})|^{1-\frac1q}\right]^\chi\\
&\le C\left[ \frac{16^{j}}{R^{2-\frac{n+p+2}{q}}} |A(k_{j}, r_{j})|^{1-\frac1q}\right]^\chi,
\end{align*}
where we used the assumption on $H$, and $|A(k_{j}, r_{j})| \le |\mathcal{Q}_R^+|\le CR^{n+p+2}$. Hence
\begin{equation}\label{eq:auxgiorgi2}
\frac{|A(k_{j+1}, r_{j+1})| }{ |\mathcal{Q}_R^+|}\le C 16^{j\chi}\left( \frac{|A(k_{j}, r_{j})|}{|\mathcal{Q}_R^+|}\right)^{(1-\frac1q)\chi},
\end{equation}
where we used that $\chi=\frac{n+p+2}{n+p}$. Therefore, similarly to \eqref{eq:nonlineariteration} and \eqref{eq:nonlineariteration2}, there exists $\theta\in(0,1)$ such that if $\frac{|A(k_{0}, r_{0})|}{|\mathcal{Q}_R^+|}\le\theta$, then
\[
\lim_{j\to \infty}\frac{|A(k_{j+1}, r_{j+1})|}{|\mathcal{Q}_R^+|}=0.
\]
By Lemma \ref{lem:measurechange}, we only need to choose $\gamma_0=(\theta/C)^{p+1}$.

Now, let us  consider $n=1, 2$. By Theorem \ref{thm:caccipolli} and Theorem \ref{thm:weightedsobolev}, \eqref{eq:auxgiorgi1} would become
\[
 \Big(\int_{\mathcal{Q}_R^+}  |\eta_j v|^{2\chi}\,\ud x \ud t \Big)^{\frac{1}{\chi}} \le C R^{\frac{p+2-n}{p+2}}\left[ \frac{2^{2j}}{R^2} \|v\|^2_{L^2(\mathcal{Q}_{r_{j}}^+)} + (M+F_1)^2 |\mathcal{Q}_{r_{j}}^+\cap \{u>k_j\}|^{1-\frac1q}\right].
\]
By using $\chi=\frac{p+2}{p+1}$, one will still obtain \eqref{eq:auxgiorgi2}. Then the left proof is the same as above.

Case 2: $-1<p<0$. We still consider $n\ge 3$ first. By Theorem \ref{thm:caccipolli} and Theorem \ref{thm:weightedsobolev2}, we have
\begin{equation}\label{eq:auxgiorgi1-2}
 \Big(\int_{\mathcal{Q}_R^+}  |\eta_j v|^{2\chi}x_n^p\,\ud x \ud t \Big)^{\frac{1}{\chi}} \le C\left[ \frac{2^{2j}}{R^2} \|v\|^2_{L^2(\mathcal{Q}_{r_{j}}^+,x_n^p\ud x\ud t)} + (M+F_1)^2 |\mathcal{Q}_{r_{j}}^+\cap \{u>k_j\}|_{\nu_p}^{1-\frac1q}\right],
\end{equation}
where $v=(u-k_j)^+$. Since 
\[
 \Big(\int_{\mathcal{Q}_R^+}  |\eta_j v|^{2\chi}x_n^p\,\ud x \ud t \Big)^{\frac{1}{\chi}} \ge (k_{j+1}- k_j)^2 |A(k_{j+1}, r_{j+1})|_{\nu_p}^{\frac{1}{\chi}},  
\]
and
\[
\int_{\mathcal{Q}_{r_j}^+}  v^{2} \,\ud x\ud t \le H^2 |A(k_{j}, r_{j})|_{\nu_p},
\]
it follows that 
\begin{align*}
|A(k_{j+1}, r_{j+1})|_{\nu_p} & \le  C\left[ \frac{2^{4j}}{R^2}|A(k_{j}, r_{j})|_{\nu_p} + \frac{2^{2j}(M+F_1)^2}{H^2} |A(k_{j}, r_{j})|_{\nu_p}^{1-\frac1q}\right]^\chi\\
&\le  C\left[ \frac{2^{4j}}{R^2}|A(k_{j}, r_{j})|_{\nu_p} + \frac{2^{2j}}{R^{p+2-\frac{n+2p+2}{q}}} |A(k_{j}, r_{j})|_{\nu_p}^{1-\frac1q}\right]^\chi\\
&\le C\left[ \frac{16^{j}}{R^{p+2-\frac{n+2p+2}{q}}} |A(k_{j}, r_{j})|_{\nu_p}^{1-\frac1q}\right]^\chi,
\end{align*}
where we used the assumption on $H$, and $|A(k_{j}, r_{j})|_{\nu_p} \le |\mathcal{Q}_R^+|_{\nu_p}\le CR^{n+2p+2}$. Hence
\begin{equation}\label{eq:auxgiorgi2-2}
\frac{|A(k_{j+1}, r_{j+1})|_{\nu_p} }{ |\mathcal{Q}_R^+|_{\nu_p}}\le C 16^{j\chi}\left( \frac{|A(k_{j}, r_{j})|_{\nu_p}}{|\mathcal{Q}_R^+|_{\nu_p}}\right)^{(1-\frac1q)\chi},
\end{equation}
where we used that $\chi=\frac{n+2p+2}{n+p}$. Hence, there exists $\theta\in(0,1)$ such that if $\frac{|A(k_{0}, r_{0})|_{\nu_p}}{|\mathcal{Q}_R^+|_{\nu_p}}\le\theta$, then
\[
\lim_{j\to \infty}\frac{|A(k_{j+1}, r_{j+1})|_{\nu_p}}{|\mathcal{Q}_R^+|_{\nu_p}}=0.
\]

If $n=1, 2$, then by Theorem \ref{thm:caccipolli} and Theorem \ref{thm:weightedsobolev2}, \eqref{eq:auxgiorgi1-2} would become
\begin{align*}
& \Big(\int_{\mathcal{Q}_R^+}  |\eta_j v|^{2\chi}x_n^p\,\ud x \ud t \Big)^{\frac{1}{\chi}} \\
 &\le C R^{\frac{p+4-n}{3}}\left[ \frac{2^{2j}}{R^2} \|v\|^2_{L^2(\mathcal{Q}_{r_{j}}^+,x_n^p\ud x\ud t)} + (M+F_1)^2 |\mathcal{Q}_{r_{j}}^+\cap \{u>k_j\}|_{\nu_p}^{1-\frac1q}\right].
\end{align*}
By using $\chi=\frac{3}{2}$, one will still obtain \eqref{eq:auxgiorgi2-2}. Then the left proof is the same as above.
\end{proof}

\begin{lem}\label{lem:decay-1} 
Let  $0<R\le \frac12$. Suppose 
\[
0\le \sup_{B_{2R}^+ \times [t_0, t_0+ R^{p+2}] } u\le\mu\le M.
\]  
Then  there exists  $C>1$ depending only on  $\lambda,\Lambda,n,p$ and $q$ such that for every $\ell\in\mathbb{Z}^+$, there holds either
\begin{equation}\label{eq:dichonomyH}
\mu\le 
\left\{
  \begin{aligned}
  &  2^{\ell} (M+F_1)R^{1-\frac{n+p+2}{2q}},  \quad  &\mbox{for }   p\ge 0, \\
  &  2^{\ell} (M+F_1)R^{\frac{p+2}{2}-\frac{n+2p+2}{2q}}, \quad   &\mbox{for }    -1<p<0,
\end{aligned}
\right.
\end{equation}
or
\begin{equation}\label{eq:thechoiceofell}
\frac{|\{(x,t)\in B_{R}^+ \times [t_0, t_0+ R^{p+2}] : u(x,t)>\mu-\frac{\mu}{2^\ell}\}|_{\nu_{p}}}{|B_{R}^+ \times [t_0, t_0+ R^{p+2}]|_{\nu_{p}}}\le C\ell^{-\min\left(\frac{1}{4},\frac{1}{4(p+1)}\right)}.
\end{equation}

\end{lem} 

\begin{proof} 
We extend $u$ to be identically zero in $(B_1\setminus B_1^+)\times[-1,0]$, which will still be denoted as $u$. Let 
\[
k_j= \mu- \frac{\mu}{2^j},
\]
and
\[
A(k_j,R;t)= B_R \cap \{u(\cdot, t)>k_j\}, \quad A(k_j,R)= B_{R} \times [t_0, t_0+ R^{p+2}]  \cap \{u>k_j\}.
\]
Since $k_j\ge 0$, we have $A(k_j,R;t)= B_R^+ \cap \{u(\cdot, t)>k_j\}$ and
\begin{equation}\label{eq:measureauto}
\int_{B_R\setminus A(k_{j},R;t)}|x_n|^p\ud x\ge \int_{B_R\setminus B_R^+}|x_n|^p\ud x\ge C R^{n+p}.
\end{equation}
Then by Theorem \ref{thm:degiorgiisoperimetricelliptic}, we have 
\begin{align*}
&(k_{j+1}-k_j) |A(k_{j+1},R;t)|_{\mu_{p}(t)} R^{n+p}\\
%&\le C R^{n+2p+1} \int_{A(k_{j},R;t) \setminus A(k_{j+1},R;t)} |\nabla u(x,t)|\,\ud x\\
%& \le C R^{n+2p+1} \left(\int_{A(k_{j},R;t) \setminus A(k_{j+1},R;t)} |\nabla u|^2 \,\ud x\right)^{1/2} \\
%&\quad\cdot \left(\int_{A(k_{j},R;t) \setminus A(k_{j+1},R;t)}|x_n|^{p}\,\ud x\right)^{\va} \left(\int_{A(k_{j},R;t) \setminus A(k_{j+1},R;t)}|x_n|^{-\frac{2p\va}{1-2\va}}\,\ud x\right)^{\frac{1-2\va}{2}}\\
& \le C R^{n+2p+1+\frac{n(1-2\va)}{2}-\va p} \left(\int_{B_R^+} |\nabla  (u-k_j)^+|^2\,\ud x\right)^{1/2} |A(k_{j},R;t) \setminus A(k_{j+1},R;t)|_{\mu_{p}(t)}^{\va},
\end{align*}
where we choose $\va=\min\left(\frac{1}{4},\frac{1}{4(p+1)}\right)$. Integrating in the time variable, and using H\"older's inequality again, we have
\begin{align*}
&\int_{t_0}^{t_0+  R^{p+2}} |A(k_{j+1},R;t)|_{\mu_{p}(t)} \,\ud t\\&
\le \frac{C2^{j+1}}{\mu} R^{p+1+\frac{n(1-2\va)}{2}-\va p+\frac{(p+2)(1-2\va)}{2}} |A(k_{j},R) \setminus A(k_{j+1},R)|_{\nu_{p}}^{\va}\\
&\quad\cdot \left(\int_{B_R^+\times [t_0, t_0+ R^{p+2}]  } |\nabla (u-k_j)^+|^2\,\ud x \ud t\right)^{1/2} .
\end{align*}
It follows from Theorem \ref{thm:caccipolli}  (with $\xi$ independent of $t$) that
\begin{align*}
&\int_{t_0}^{t_0+  R^{p+2}}\int_{B_R^+  } |\nabla (u-k_j)^+|^2 \,\ud x \ud t \\&
\le C \Big(\int_{B_{2R}^+} |(u-k_j)^+(t_0)|^2 x_n^{p} \,\ud x+\frac{1}{R^2}\int_{t_0}^{t_0+  R^{p+2}}\int_{B_{2R}^+  } |(u-k_j)^+|^{2} \,\ud x\ud t \\
& \quad+\int_{t_0}^{t_0+  R^{p+2}}\int_{B_{2R}^+  } |(u-k_j)^+|^{2} (1\vee x_n^p)\,\ud x\ud t + (M+F_1)^2|B_{2R}^+ \times [t_0, t_0+ R^{p+2}]|^{1-\frac{1}{q}}\Big )\\
&\le C\left( \frac{\mu^2}{ 4^j} R^{n+p}+\frac{\mu^2}{ 4^j} R^{n+p} +\frac{\mu^2}{ 4^j} R^{n+2p+2} + (M+F_1)^2 R^{(n+2+p)(1-1/q)}\right). 
\end{align*}
If \eqref{eq:dichonomyH} fails for some $\ell$, then we have
\[
\int_{t_0}^{t_0+  R^{p+2}}\int_{B_R^+  } |\nabla (u-k_j)^+|^2 \,\ud x \ud t \le C \frac{\mu^2}{ 4^j} R^{n+p}
\]
for $j\le \ell$, where we used that $R^{\frac{p+2}{2}-\frac{n+2p+2}{2q}}\ge R^{1-\frac{n+p+2}{2q}}$ for $p<0$. Hence, 
\[
 |A(k_{j+1},R)|_{\nu_{p}} \le C R^{p+1+\frac{n(1-2\va)}{2}-\va p+\frac{(p+2)(1-2\va)}{2}+\frac{n+p}{2}} |A(k_{j},R) \setminus A(k_{j+1},R)|_{\nu_{p}}^{\va}
\]
or 
\[
( |A(k_{j+1},R)|_{\nu_{p}})^{\frac{1}{\va}} \le C R^{\left(\frac{1}{\va}-1\right)(n+2p+2)} |A(k_{j},R) \setminus A(k_{j+1},R)|_{\nu_{p}}. 
\]
Taking a summation, we have 
\begin{align*}
\ell (|A(k_{\ell},R)|_{\nu_{p}})^{\frac{1}{\va}} \le \sum_{j=0}^{\ell-1}  (|A(k_{j+1},R)|_{\nu_{p}})^{\frac{1}{\va}} &\le  C R^{\left(\frac{1}{\va}-1\right)(n+2p+2)} |B_{R}^+ \times [t_0, t_0+ R^{p+2}]|_{\nu_{p}} \\& \le  C  (|B_{R}^+ \times [t_0, t_0+R^{p+2}]|_{\nu_{p}})^{\frac{1}{\va}}. 
\end{align*}
The lemma follows.
\end{proof}

Now we can prove the H\"older continuity on the boundary.
\begin{thm}\label{thm:holderontheboundary}
Suppose $u\in C ([-1,0]; L^2(B_1^+,x_n^{p}\ud x)) \cap  L^2((-1,0];H_{0,L}^1(B_1^+))$ is a weak solution of \eqref{eq:linear-eq} with the partial boundary condition \eqref{eq:linear-eq-D}, where the coefficients of the equation satisfy \eqref{eq:rangep}, \eqref{eq:ellip2}, \eqref{eq:localbddnesscoeffcient} and \eqref{eq:localbddnessf} for some $q>\max(\frac{\chi}{\chi-1},\frac{n+p+2}{2}, \frac{n+2p+2}{p+2})$.   Let $\bar x\in\pa' B_{1/2}$ and $\bar t\in (-1/4,0])$. Then there exist $\alpha>0$ and $C>0$, both of which depend only on  $\lambda,\Lambda,n,p$ and $q$, such that
\[
|u(x,t)-u(\bar x, \bar t)|\le C(M+F_1) (|x-\bar x|+|t-\bar t|^{\frac{1}{p+2}})^\alpha 
\]
for every $(x,t)\in B_{1/2}^+\times(-1/4,0]$.
\end{thm}
\begin{proof}
Without loss of generality, we assume $(\bar x,\bar t)=(0,0)$.  For $R\in(0,1/2]$, denote
\[
\mu(R)=\sup_{\mathcal{Q}_R^+}u, \quad \widetilde\mu(R)=\inf_{\mathcal{Q}_R^+}u, \quad\omega(R)=\mu(R)-\widetilde\mu(R).
\]
Let $\gamma_0$ be the one in Lemma \ref{lem:smallonlargeset}. We can choose $\ell$ sufficiently large so that
\[
C\ell^{-\min\left(\frac{1}{4},\frac{1}{4(p+1)}\right)}\le\gamma_0,
\]
where $C$ is the one in \eqref{eq:thechoiceofell}.  Then it follows from Lemma \ref{lem:smallonlargeset} and Lemma \ref{lem:decay-1} that either
\begin{equation}\label{eq:dichonomyH-11}
\mu(R)\le 
\left\{
  \begin{aligned}
  &  2^{\ell} (M+F_1)R^{1-\frac{n+p+2}{2q}},  \quad  &\mbox{for }   p\ge 0, \\
  &  2^{\ell} (M+F_1)R^{\frac{p+2}{2}-\frac{n+2p+2}{2q}}, \quad   &\mbox{for }    -1<p<0,
\end{aligned}
\right.
\end{equation}
or
\begin{equation}\label{eq:oscalter2-0}
\mu(R/4)\le \mu(R)-\frac{\mu(R)}{2^{\ell+1}}.
\end{equation}
Applying these estimates to $-u$, we have
either
\begin{equation}\label{eq:dichonomyH-12}
\tilde \mu(R)\ge 
\left\{
  \begin{aligned}
  & - 2^{\ell} (M+F_1)R^{1-\frac{n+p+2}{2q}},  \quad  &\mbox{for }   p\ge 0, \\
  & - 2^{\ell} (M+F_1)R^{\frac{p+2}{2}-\frac{n+2p+2}{2q}}, \quad   &\mbox{for }    -1<p<0,
\end{aligned}
\right.
\end{equation}
or
\begin{equation}\label{eq:oscalter2-1}
\tilde\mu(R/4)\ge \tilde\mu(R)-\frac{\tilde\mu(R)}{2^{\ell+1}}.
\end{equation}
In any case, we will obtain 
\begin{equation*}
\omega(R/4)\le 
\left\{
  \begin{aligned}
  &  (1-2^{-\ell-1})\omega(R)+2^{\ell+1} (M+F_1)R^{1-\frac{n+p+2}{2q}},  \quad  &\mbox{for }   p\ge 0, \\
  &  (1-2^{-\ell-1})\omega(R)+2^{\ell+1} (M+F_1)R^{\frac{p+2}{2}-\frac{n+2p+2}{2q}}, \quad   &\mbox{for }    -1<p<0.
\end{aligned}
\right.
\end{equation*}
By an iterative lemma, e.g. Lemma 3.4 in Han-Lin \cite{HL} (or Lemma B.2 in \cite{JX19}), there exist $\alpha$ and $C$, both of which depend only on $\lambda,\Lambda,n,p$ and $q$, such that
\[
\omega(R)\le C(M+F_1)R^{\alpha}\quad\,\forall R\in(0,R_0],
\]
from which the conclusion follows.
\end{proof}

\subsection{Interior H\"older estimates}

When $x$ is away from the boundary $\partial'B_1^+$, the equation \eqref{eq:linear-eq} is uniformly parabolic. We observe that all the assumptions in Theorem \ref{thm:localboundedness} are stronger than those in the uniformly parabolic case (which corresponds to $p=0$ and $\chi=\frac{n+2}{n}$). Therefore, using the same proof for uniformly parabolic equations, with a small adaptation to the existence of the coefficient $a$ in front of $\partial_t u$, one can show the following interior H\"older estimate.

\begin{thm}\label{thm:uniformholdernearboundary}
Suppose $u\in C ([-1,0]; L^2(B_1^+,x_n^{p}\ud x)) \cap  L^2((-1,0];H_{0,L}^1(B_1^+))$ is a weak solution of \eqref{eq:linear-eq} with the partial boundary condition \eqref{eq:linear-eq-D}, where the coefficients of the equation satisfy \eqref{eq:rangep}, \eqref{eq:ellip2}, \eqref{eq:localbddnesscoeffcient} and \eqref{eq:localbddnessf} for some $q>\max(\frac{\chi}{\chi-1},\frac{n+p+2}{2}, \frac{n+2p+2}{p+2})$.  Then there exist $\alpha>0$ and $C>0$, both of which depend only on  $\lambda,\Lambda,n,p$ and $q$, such that for every $(x,t), (y,s)\in B_{1/4}(e_n/2)\times(-1/4,0]$, there holds
\[
|u(x,t)-u(y,s)|\le C(M+F_1) (|x-y|+|t-s|)^\alpha,
\]
where $e_n=(0,\cdots,0,1)$.
\end{thm}

The proof Theorem \ref{thm:uniformholdernearboundary} will be proved as follows. We only need to prove the H\"older continuity at the point $(e_n/2,0)$. 

Similar to Theorem \ref{thm:caccipolli}, we have the Caccipolli inequality around the point $(e_n/2,0)$. Let  $Q_{\rho,\tau}=B_\rho(e_n/2)\times(t_0,t_0+\tau]\subset Q_{1/2}(e_n/2,0)$, $k\in\R$, $\va\in(0,1]$, and $\xi\in  V_2^{1,1}(Q_{\rho,\tau})$ such that $\xi=0$ on $\pa B_\rho(e_n/2)\times (t_0,t_0+\tau]$ and  $0\le\xi\le 1$. Then 
\begin{align}
&\max\left(\sup_{t\in(t_0,t_0+\tau)}\int_{B_\rho(e_n/2)}x_n^{p} a [\xi(u-k)^+]^2(x,t)\,\ud x, \lambda\iint_{Q_{\rho,\tau}} |D[\xi(u-k)^+]|^2\,\ud x\ud s\right)\nonumber\\
&\le (1+\va)\int_{B_\rho(e_n/2)}x_n^{p} a [\xi(u-k)^+]^2(x,t_0)\,\ud x +C\iint_{Q_{\rho,\tau}} (|D\xi|^2+|\xi\pa_t\xi|x_n^p) [(u-k)^+]^2 \,\ud x\ud t \nonumber\\
&\quad +\frac{C}{\va^\kappa}\left( \|[(u-k)^+] \xi\|_{L^2(Q_{\rho,\tau})}^2 + (k^2+F_1^2)|\{u>k\}\cap Q_{\rho,\tau}|^{1-\frac1q} \right). \label{eq:caccipolliu}
\end{align}

\begin{lem}\label{lem:smallonlargesetinterior} 
Let $0<R\le 1/2$ and 
\[
 \sup_{Q_R(e_n/2,0)} u\le \mu\le  M.
\]
Then there exists  $0<\gamma_0<1$ depending only on  $\lambda,\Lambda,n,p$ and $q$ such that for $k<\mu$, if
\begin{equation*}
H:=\mu-k>(M+F_1)R^{1-\frac{n+2}{2q}},
\end{equation*}
and
\[
\frac{|\{(x,t)\in Q_R(e_n/2,0): u(x,t)>k\} |_{\nu_{p}}}{|Q_R(e_n/2,0)|_{\nu_{p}}} \le \gamma_0,
\]
then
\[
u\le \mu-\frac{H}{2}\quad\mbox{in }Q_{R/2}(e_n/2,0).
\]
\end{lem}
The proof of Lemma \ref{lem:smallonlargesetinterior} is almost identical to that of Lemma \ref{lem:smallonlargeset}, and thus, we omit it.

\begin{lem}\label{lem:interiordecay-1} 
Let $0<\delta\le 1$, $0<R\le \frac14$, $0<\sigma<1$ and 
\[
\sup_{B_{2R}(e_n/2) \times [t_0, t_0+\delta R^{2}] } u\le\mu\le M.
\]  
Suppose that $k<\mu$ and
\begin{equation}\label{eq:assumptionofmeasure}
|\{x\in B_R(e_n/2): u(x,t)>k\}|_{\mu_{p}(t)}\le (1-\sigma) |B_R(e_n/2)|_{\mu_{p}(t)} \quad \mbox{for any } t_0\le t\le  t_0+\delta R^{2}.
\end{equation}
Then  there exists  $C>1$ depending only on  $\lambda,\Lambda,n,p$ and $q$ such that for every $\ell\in\mathbb{Z}^+$, there holds either
\begin{equation}\label{eq:dichonomyHinterior}
H:=\mu-k\le   2^{\ell} (M+F_1)R^{1-\frac{n+2}{2q}},  
\end{equation}
or
\[
\frac{|\{(x,t)\in B_{R}(e_n/2) \times [t_0, t_0+\delta R^{2}] : u(x,t)>\mu-\frac{H}{2^\ell}\}|_{\nu_{p}}}{|B_{R}(e_n/2) \times [t_0, t_0+\delta R^{2}]|_{\nu_{p}}}\le \frac{C}{\sigma \sqrt{\delta  \ell}}.
\]
\end{lem} 
\begin{proof}
The proof is very similar to that of Lemma \ref{lem:decay-1} with the following two changes.  The first is that $k_j$ should be defined as $k_j=\mu-\frac{H}{2^j}$ instead.  The second is that the estimate \eqref{eq:measureauto} should be replaced by the assumption \eqref{eq:assumptionofmeasure}. The left proofs are identical so that we omit it.
\end{proof}

The next lemma was not needed in the proof of Theorem \ref{thm:holderontheboundary}, and its proof is slightly different from the uniformly parabolic equations with $a\equiv 1$. Thus, we provide a proof.

\begin{lem}\label{lem:decay-2} 
Let $0<\sigma<1$.  There exist $R_0\in(0,\frac14)$ and $s_0>1$ depending only on   $\lambda,\Lambda,n,p, q$ and $\sigma$  such that the following holds. Let $R\in(0,R_0]$ and 
\[
\sup_{B_{2R}(e_n/2) \times [t_0, t_0+R^{2}] } u\le\mu\le M.
\]
Suppose that $k<\mu$ and
\[
|\{x\in B_R(e_n/2): u(x,t_0)>k\}|_{\mu_{p}(t_0)}\le (1-\sigma) |B_R(e_n/2)|_{\mu_{p}(t_0)}.
\]  
Then either
\begin{equation}\label{eq:dichonomyH2}
H:=\mu-k\le  2^{s_0} (M+F_1)R^{1-\frac{n+2}{2q}}
\end{equation}
or
\[
|\{x\in B_R(e_n/2): u(x,t)>\mu-\frac{H}{2^{s_0}}\}|_{\mu_{p}(t)}\le (1-\frac{\sigma}{2}) |B_R(e_n/2)|_{\mu_{p}(t)} \quad \mbox{for all } t_0\le t\le  t_0+R^{p+2} . 
\] 
\end{lem}  

\begin{proof}  Let $\eta$ be a cut-off function supported in $B_R(e_n/2)$ and $\eta=1$ in $B_{\beta R}(e_n/2)$, where $0<\beta<1$ will be fixed later.  Let $0<\delta \le 1$  and 
\[
A^\delta (k,R)= \{B_R(e_n/2) \times[t_0, t_0+\delta R^{2}] \}\cap \{u>k\}.
\]
Let $k_1>1$.  By \eqref{eq:caccipolliu},  we have
\begin{align*}
&\sup_{t_0<t< t_0+\delta  R^{2}} \int_{B_R(e_n/2)} x_n^p a v^{2}\eta^2  \,\ud x  \\&
\quad \le (1+\va) \int_{B_R(e_n/2)}  x_n^p a v^{2}\eta^2 \,\ud x\Big|_{t_0}  + \frac{C}{\va^\kappa } \left(\frac{H^2  |A^\delta (k ,R)| }{(1-\beta)^2R^2}+ (M+F_1)^2 |A^\delta (k ,R)|^{1-\frac{1}{q}} \right),
\end{align*} 
where $v=(u-k)^+$. Note that 
\begin{align*}
\int_{B_R(e_n/2)} x_n^p a v^{2}\eta^2  \,\ud x \Big|_t &\ge (1-2^{-k_1})^2 H^2 |B_{\beta R}(e_n/2) \cap \{u(x,t)> \mu- H 2^{-k_1}\}|_{\mu_{p}(t)},\\
\int_{B_R(e_n/2)} x_n^p a  v^{2}\eta^2  \,\ud x\Big|_{t_0}  &\le H^2\{x\in B_R(e_n/2): u(x,t_0)>k\}|_{\mu_{p}(t_0)} \\&\le  (1-\sigma) H^2 |B_R(e_n/2)|_{\mu_{p}(t_0)}.
\end{align*}
It follows that if \eqref{eq:dichonomyH2} fails, then for all $t\in [t_0, t_0+\delta R^{p+2}]$, 
\begin{align*}
& |B_{\beta R}(e_n/2) \cap \{u(x,t)> \mu- H 2^{-k_1}\}|_{\mu_{p}(t)} \\&\le |B_R(e_n/2)|_{\mu_{p}(t_0)}  \frac{(1+\va) (1-\sigma)}{(1-2^{-k_1})^2 }+\frac{CR^{n}}{\va^\kappa }\left[ \frac{C }{(1-\beta)^2} \mathscr{A}^\delta(k ,R)+\left(\mathscr{A}^\delta(k ,R)\right)^{1-\frac1q}  \right],
\end{align*}
where
\begin{equation*}
\mathscr{A}^\delta(k ,R):=  \frac{|A^\delta(k ,R)| }{R^{n+2}}.
\end{equation*}
Hence,
\begin{align*}
& |B_{ R}(e_n/2) \cap \{u(x,t)> \mu- H 2^{-k_1}\}|_{\mu_{p}(t)}  \\&\le |B_R(e_n/2)|_{\mu_{p}(t_0)}  \left( C(1-\beta)+ \frac{(1-\sigma)}{(1-2^{-k_1})^2}+4\va+\frac{C }{(1-\beta)^2\va^\kappa } \left(\mathscr{A}^\delta(k ,R)\right)^{1-\frac1q} \right).
\end{align*}
By choosing $\beta$ such that
\[
(1-\beta)^3=\left(\mathscr{A}^\delta(k ,R)\right)^{1-\frac1q}, 
\] 
we have
\begin{align}\label{eq:smallinitiallater0}
& |B_{ R}(e_n/2) \cap \{u(x,t)> \mu- H 2^{-k_1}\}|_{\mu_{p}(t)}  \nonumber\\
&\le |B_R(e_n/2)|_{\mu_{p}(t_0)}  \left( \frac{(1-\sigma)}{(1-2^{-k_1})^2}+4\va+\frac{C }{\va^\kappa } \left(\mathscr{A}^\delta(k ,R)\right)^{\frac13(1-\frac1q)} \right).
\end{align}

For every $t_0\le\tau_1\le\tau_2<t_0+R^{p+2}$, we have
\begin{align*}
||B_R(e_n/2)|_{\mu_{p}(\tau_1)}-|B_R(e_n/2)|_{\mu_{p}(\tau_2)}|&\le \int_{B_R(e_n/2)}|a(x,\tau_1)-a(x,\tau_2)|x_n^p\,\ud x\\
&\le \int_{t_0}^{t_0+R^{p+2}}\int_{B_R(e_n/2)}|\pa_\tau a(x,\tau)|x_n^p\,\ud x\ud \tau\\
&\le \Lambda \left(\int_{t_0}^{t_0+R^{2}}\int_{B_R(e_n/2)}x_n^{p}\,\ud x\ud \tau\right)^\frac{q-1}{q}\\
&\le C R^{\frac{(n+2)(q-1)}{q}}\\
&= C\theta R^{n},
\end{align*}
where we used \eqref{eq:localbddnesscoeffcient} in the third inequality, and
\[
\theta= R^{2-\frac{n+2}{q}}.
\]
Then \eqref{eq:smallinitiallater0} becomes
\begin{align*}
& |B_{ R}(e_n/2) \cap \{u(x,t)> \mu- H 2^{-k_1}\}|_{\mu_{p}(t)}  \nonumber\\
&\le |B_R(e_n/2)|_{\mu_{p}(t)}  \left( \frac{(1+C\theta)(1-\sigma)}{(1-2^{-k_1})^2}+C\va+\frac{C }{\va^\kappa } \left(\mathscr{A}^\delta(k ,R)\right)^{\frac13(1-\frac1q)} \right).
\end{align*}
If we let 
\[
\va=\left(\mathscr{A}^\delta(k ,R)\right)^{\frac{1}{3(1+\kappa )}(1-\frac1q)},
\]
 then
\begin{align}\label{eq:smallinitiallater}
& |B_{ R}(e_n/2) \cap \{u(x,t)> \mu- H 2^{-k_1}\}|_{\mu_{p}(t)} \nonumber\\
&\le |B_R(e_n/2)|_{\mu_{p}(t)}  \left( \frac{(1+C\theta)(1-\sigma)}{(1-2^{-k_1})^2}+C\left(\mathscr{A}^\delta(k ,R)\right)^{\frac{1}{3(1+\kappa )}(1-\frac1q)}\right).
\end{align}
Since 
\[
\mathscr{A}^\delta(k ,R)\le C\delta,
\]
we fix an $a$ such that
\[
C\delta^{\frac{1}{3(1+\kappa )}(1-\frac1q)}  <\frac{1}{8}\min(1-\sigma,\sigma).
\]
We choose $\delta$ slightly smaller if necessary to make $\delta^{-1}$ to be an integer. Let $N=\delta^{-1}$ and denote
 \[
 t_j=t_0+j\delta R^{2}\quad j=1,2,\cdots,N.
 \]
 We will inductively prove that there exist $s_1<s_2<\cdots<s_N$ such that
 \begin{equation}\label{eq:densitypropa}
|B_{ R}(e_n/2) \cap \{u(x,t)> \mu- H 2^{-s_j}\}|_{\mu_{p}(t)}  \le \left(1-\sigma+\frac{j}{4N} \sigma\right)  |B_R(e_n/2)|_{\mu_{p}(t)}
\end{equation}
for all $t_{j-1}\le t\le t_j$, where all the $s_j$ depend only on  $\lambda,\Lambda,n,p, q$ and $\sigma$, from which the conclusion of this lemma follow.

Let us consider $j=1$ first.   

Since $2-\frac{n+2}{q}>0$, there exist $R_0$ small and $k_0$ large, depending on $\sigma$, such that  for all  $k_1\ge k_0$ and $R\le R_0$, we have
\[
\frac{(1+C\theta)(1-\sigma)}{(1-2^{-k_1})^2 }\le 1-\sigma+\frac{\sigma}{8N}.
\]
Then, 
\[
|B_{ R}(e_n/2) \cap \{u(x,t)> \mu- H 2^{-k_1}\}|_{\mu_{p}(t)}  \le \left(1-\frac{3}{4} \sigma\right)  |B_R(e_n/2)|_{\mu_{p}(t)} 
\]
 for all $t\in [t_0, t_1]$.  Applying Lemma \ref{lem:interiordecay-1}, for every $k_2>k_1$, we have
\[
|A^\delta(\mu- H 2^{-k_2},R)|_{\nu_p} \le \frac{C}{\sigma\sqrt{\delta(k_2-k_1)}} |\{B_R(e_n/2) \times[t_0, t_0+\delta R^{2}] \}|_{\nu_p}\le \frac{C\sqrt{\delta}\,R^{n+2}}{\sigma\sqrt{k_2-k_1}}.
\]
Hence,
\[
\frac{|A^\delta(\mu- H 2^{-k_2},R)|}{R^{n+2}}\le \frac{C\sqrt{\delta}}{\sigma\sqrt{k_2-k_1}}.
\]
Hence, we can choose $k_2$ large enough such hat
\[
C\left( \mathscr{A}^\delta(\mu- H 2^{-k_2} ,R) \right)^{\frac{1}{3(1+\kappa )}(1-\frac1q)}\le \frac{\sigma}{8N}.
\]
Let $k_1=k_0$ and $s_1=k_1+k_2$. By replacing $H$ by $H 2^{-k_2}$ in \eqref{eq:smallinitiallater}, it follows that
\[
|B_{ R}(e_n/2) \cap \{u(x,t)> \mu- H 2^{-s_1}\}|_{\mu_{p}(t)}  \le \left(1-\sigma+\frac{1}{4N} \sigma\right)  |B_R(e_n/2)|_{\mu_{p}(t)}\quad\mbox{for all } t_{0}\le t\le t_1.
\]
This prove \eqref{eq:densitypropa} for $j=1$. The proof for $j=2,3,\cdots,N$ is similar, and we omit it.
\end{proof}

Combining the above three lemmas, we will have the following improvement of oscillations. 
\begin{lem}\label{lem:decay-3} 
Let $0<\sigma<1$.  There exist $R_0\in(0,\frac12)$ and $s>1$ depending only on   $\lambda,\Lambda,n,p, q$ and $\sigma$  such that the following holds. Let $R\in(0,R_0]$ and 
\[
\sup_{B_{2R}(e_n/2) \times [-R^{2},0] } u\le\mu\le M.
\]
Suppose that $k<\mu$ and
\[
|\{x\in B_R(e_n/2): u(x,-R^{2})>k\}|_{\mu_{p}(-R^{2})}\le (1-\sigma) |B_R(e_n/2)|_{\mu_{p}(-R^{2})}.
\]  
Then either
\begin{equation}\label{eq:dichonomyH3}
H:=\mu-k\le 2^{s} (M+F_1)R^{1-\frac{n+2}{2q}}
\end{equation}
or
\[
\sup_{Q_{R/2}(e_n/2,0)} u\le \mu-\frac{H}{2^{s}}.  
\]

\end{lem}

\begin{proof} Let $R_0$ and $s_0$ be those from Lemma \ref{lem:decay-2}. Suppose \eqref{eq:dichonomyH3} fails for some $s>s_0$, which will be fixed in the end. 
Then it follows from  Lemma \ref{lem:decay-2} that
\[
|\{x\in B_R(e_n/2): u(x,t)>\mu-\frac{H}{2^{s_0}}\}|_{\mu_{p}(t)}\le (1-\frac{\sigma}{2}) |B_R(e_n/2)|_{\mu_{p}(t)} \quad \mbox{for every } t_0\le t\le  t_0+R^{2} . 
\] 
Then using Lemma \ref{lem:interiordecay-1}, we have
\[
\frac{|\{(x,t)\in B_{R}(e_n/2) \times [t_0, t_0+R^{2}] : u(x,t)>\mu-\frac{H}{2^{s-1}}\}|_{\nu_{p}}}{|B_{R}(e_n/2) \times [t_0, t_0+R^{2}]|_{\nu_{p}}}\le \frac{C}{\sigma \sqrt{s-s_0-1}}.
\]
Let $\gamma_0$ be the one in Lemma \ref{lem:smallonlargesetinterior}. We can choose $s$ sufficiently large so that
\[
\frac{C}{\sigma \sqrt{s-s_0-1}}\le\gamma_0.
\]
Then it follows from Lemma \ref{lem:smallonlargesetinterior} that 
\[
\sup_{Q_{R/2}(e_n/2,0)} u\le \mu-\frac{H}{2^{s}}.  
\]
\end{proof}

\begin{rem}\label{rem:decay-3}
From the above proof, for $\delta_0\le\delta\le\delta_0^{-1}$, if we consider the problem in $B_{2R}(e_n/2)\times[-\delta R^{p+2},0]$ instead of $B_{2R}(e_n/2)\times[- R^{p+2},0]$, then the conclusion in Lemma \ref{lem:decay-3} still holds, where the constant $s$ would additionally depend on $\delta_0$.
\end{rem}

%\begin{lem}\label{lem:decay-4} 
%Let $0<\sigma<1$.  There exist $R_0\in(0,\frac12)$ and $s>1$ depending only on   $\lambda,\Lambda,n,p, q$ and $\sigma$  such that the following holds. Let $R\in(0,R_0]$ and 
%\[
%\inf_{B_{2R}(e_n/2) \times [-R^{2},0] } u\ge\widetilde\mu\ge -M.
%\]
%Suppose that $k>\widetilde\mu$ and
%\[
%|\{x\in B_R(e_n/2): u(x,-R^{2})<k\}|_{\mu_{p}(-R^{2})}\le (1-\sigma) |B_R(e_n/2)|_{\mu_{p}(-R^{2})}.
%\]  
%Then either
%\begin{equation*}
%H:=k-\widetilde\mu\le  2^{s} (M+F_1)R^{1-\frac{n+2}{2q}}
%\end{equation*}
%or
%\[
%\inf_{Q_{R/2}(e_n/2,0)} u\ge \widetilde\mu+\frac{H}{2^{s}}.  
%\]
%\end{lem}

\begin{proof}[Proof of Theorem \ref{thm:uniformholdernearboundary}]
We only need to prove the H\"older continuity at the point $(e_n/2,0)$. 
 Let $R_0$ be the one in Lemma \ref{lem:decay-3} with $\sigma=1/2$. For $R\in(0,R_0]$, denote
\[
\mu(R)=\sup_{Q_R(e_n/2,0)}u, \quad \widetilde\mu(R)=\inf_{\mathcal{Q}_R(e_n/2,0)}u, \quad\omega(R)=\mu(R)-\widetilde\mu(R).
\]
Then one of the following two inequalities must hold:
\begin{align}
\left|\left\{x\in B_{\frac{R}{2}}(e_n/2): u\Big(x,-(\frac R2)^{2}\Big)>\mu(R)- \frac12\omega(R)\right\}\right|_{\mu_{p}(-(\frac R2)^{2})}\le \frac12 |B_{\frac R2}(e_n/2)|_{\mu_{p}(-(\frac R2)^{2})}, \label{eq:dichonomy1}\\
\left|\left\{x\in B_{\frac R2}(e_n/2): u\Big(x,-(\frac R2)^{2}\Big)<\widetilde\mu(R)+ \frac12\omega(R)\right\}\right|_{\mu_{p}(-(\frac R2)^{2})}\le \frac12 |B_{\frac R2}(e_n/2)|_{\mu_{p}(-(\frac R2)^{2})}.\label{eq:dichonomy2}
\end{align}
If \eqref{eq:dichonomy1} holds, then by Lemma \ref{lem:decay-3}, there exists $s>1$ such that
either
\begin{equation}\label{eq:oscalter1}
\frac{\omega(R)}{2}\le 2^{s} (M+F_1)R^{1-\frac{n+2}{2q}}
\end{equation}
or
\begin{equation}\label{eq:oscalter2}
\mu(R/4)\le \mu(R)-\frac{\omega(R)}{2^{s+2}}.
\end{equation}
If \eqref{eq:dichonomy2} holds, then by applying the above estimates to $-u$, one has either \eqref{eq:oscalter1}  or 
\begin{equation}\label{eq:oscalter3}
\widetilde\mu(R/4)\ge \widetilde\mu(R)+\frac{\omega(R)}{2^{s+2}}.
\end{equation}
In any case, we obtain 
\begin{equation*}
\omega(R/4)\le  (1-2^{-s-2})\omega(R)+2^{s+1} (M+F_1)R^{1-\frac{n+2}{2q}}.
\end{equation*}
By an iterative lemma, e.g. Lemma 3.4 in Han-Lin \cite{HL} (or Lemma B.2 in \cite{JX19}), there exist $\alpha$ and $C$, both of which depend only on $\lambda,\Lambda,n,p$ and $q$, such that
\[
\omega(R)\le C(M+F_1)R^{\alpha}\quad\,\forall R\in(0,R_0],
\]
from which the conclusion follows.
\end{proof}

\subsection{H\"older estimates near the boundary}

Together with the H\"older regularity at the boundary in Theorem \ref{thm:holderontheboundary} and the interior H\"older regularity in Theorem \ref{thm:uniformholdernearboundary}, one can obtain the H\"older regularity up to the boundary. 
\begin{thm}\label{thm:holdernearboundary}
Suppose $u\in C ([-1,0]; L^2(B_1^+,x_n^{p}\ud x)) \cap  L^2((-1,0];H_{0,L}^1(B_1^+))$ is a weak solution of \eqref{eq:linear-eq} with the partial boundary condition \eqref{eq:linear-eq-D}, where the coefficients of the equation satisfy \eqref{eq:rangep}, \eqref{eq:ellip2}, \eqref{eq:localbddnesscoeffcient} and \eqref{eq:localbddnessf} for some $q>\max(\frac{\chi}{\chi-1},\frac{n+p+2}{2}, \frac{n+2p+2}{p+2})$.  Then for every $\gamma>0$, there exist $\theta>0$ and $C>0$, both of which depend only on  $\lambda,\Lambda,n,p,\gamma$ and $q$, such that for every $(x,t), (y,s)\in B_{1/2}^+\times(-1/4,0]$, there holds
\[
|u(x,t)-u(y,s)|\le C(\|u\|_{L^\gamma({\mathcal{Q}}_1^+)}+F_1) (|x-y|+|t-s|)^\theta.
\]
\end{thm}
\begin{proof}
By normalization, we assume $\sup_{B_{3/4}\times[-3/4,0]}|u|+F_1=1$. For any $\bar x=(0,\bar x_n )\in B_{1/2}^+$, we let $R:=\bar x_n>0$, and rescale the solution and the coefficients as in \eqref{eq:rescaledequationcoefficients} with $x_0=0$.  Then \eqref{eq:rescaledequation}, \eqref{eq:rescaledequationcoefficients1} and \eqref{eq:rescaledequationcoefficients2} hold. By Theorem \ref{thm:uniformholdernearboundary}, there exist $C>1$ and $0<\beta<1$, both of which depend only on  $\lambda,\Lambda,n,p$ and $q$, such that  
\begin{equation}\label{eq:holderafterscaling}
|\tilde u(e_n,0)-\tilde u(y,s)|\le C ||y-e_n|+\sqrt{s}|^\beta \quad \mbox{for all }(y,s) \mbox{ such that }|y-e_n|+\sqrt{s}<1/2.
\end{equation}

Consider $t\in (-1/2,0]$. If $|t|\le R^{2p+4}$, then we have
\[
|u(\bar x,t)-u(\bar x, 0)|=|\tilde u(e_n,t/R^{p+2})-\tilde u(e_n,0)|\le C |t/R^{p+2}|^{\beta/2}\le C t^{\beta/4},
\]
where we used \eqref{eq:holderafterscaling} in the first inequality. If $|t|\ge R^{2p+4}$, then we have
\begin{align*}
|u(\bar x,t)-u(\bar x, 0)|& \le |u(\bar x,t)-u(0,t)|+|u(0,t)-u(0,0)|+|u(0,0)-u(\bar x, 0)|\\
& \le C(R^{\alpha}+|t|^{\frac{\alpha}{p+2}})\\
&\le C |t|^{\frac{\alpha}{2(p+2)}},
\end{align*}
where we used Theorem \ref{thm:holderontheboundary} in the second inequality. This shows that $u$ is H\"older continuous in the time variable.

Consider $\tilde x=(\tilde x',\tilde x_n)\in B_{1/2}^+$ such that $\tilde x_n\le \bar x_n$. If $\tilde x\in B_{R^2}(\bar x)$, then we have 
\[
|u(\bar x,0)-u(\tilde x, 0)|=|\tilde u(e_n,0)-\tilde u(\tilde x/R,0)|\le C ||\tilde x-\bar x|/R|^{\beta}\le C |\tilde x-\bar x|^{\beta/2},
\]
where we used \eqref{eq:holderafterscaling} in the first inequality. If $\tilde x\not\in B_{R^2}(\bar x)$, then we have 
\begin{align*}
|u(\bar x,0)-u(\tilde x, 0)|& \le |u(\bar x,0)-u(0,0,0)|+|u(0,0,0)-u(\tilde x',0,0)|+|u(\tilde x',0,0)-u(\tilde x, 0)|\\
& \le C(R^{\alpha}+|\tilde x_n|^\alpha)\\
&\le C |\bar x-\tilde x|^{\frac{\alpha}{2}},
\end{align*}
where we used Theorem \ref{thm:holderontheboundary} in the second inequality. This shows that $u$ is H\"older continuous in the spatial variables.

Together with Theorem \ref{thm:localboundedness}, we finish the proof of this theorem.
\end{proof}

\subsection{H\"older estimates up to the initial time}
We can also show H\"older estimates up to the initial time.
\begin{thm}\label{thm:holderboundarybottom}
Suppose $u\in C ([-1,0]; L^2(B_1^+,x_n^{p}\ud x)) \cap  L^2((-1,0];H_{0,L}^1(B_1^+))$ is a weak solution of \eqref{eq:linear-eq} with the partial boundary condition \eqref{eq:linear-eq-D} and the initial condition $u(\cdot,-1)=0$, where the coefficients of the equation satisfy \eqref{eq:rangep}, \eqref{eq:ellip2}, \eqref{eq:localbddnesscoeffcient} and \eqref{eq:localbddnessf} for some $q>\max(\frac{\chi}{\chi-1},\frac{n+p+2}{2}, \frac{n+2p+2}{p+2})$.   Let $\bar x\in\pa' B_{1/4}$. Then for every $\gamma>0$, there exist $\alpha>0$ and $C>0$, both of which depend only on  $\lambda,\Lambda,n,p,\gamma$ and $q$, such that
\[
|u(x,t)-u(\bar x, -1)|\le C(\|u\|_{L^\gamma({\mathcal{Q}}_1^+)}+F_1) (|x-\bar x|+|t+1|^{\frac{1}{p+2}})^\alpha 
\]
for every $(x,t)\in B_{1/4}^+\times[-1,-\frac34]$.
\end{thm}
\begin{proof}
Let $M=\|u\|_{L^\infty(B_{3/4}^+\times(-1,-1/4))}$,
\[
\mu(R)=\sup_{\mathcal{Q}_R^+(\bar x,-1)}u, \quad \widetilde\mu(R)=\inf_{\mathcal{Q}_R^+(\bar x,-1)}u, \quad\omega(R)=\mu(R)-\widetilde\mu(R),
\]
\[
r_j=\frac {R}2+\frac{R}{2^{j+1}},\quad k_j=\frac{\mu(R)}{2}- \frac{\mu(R)}{2^{j+1}},\quad j=0,1,2,\cdots.
\]
For brevity, we denote
\[
\mathcal{Q}_{j,\delta}^+=B_{r_j}^+(\bar x)\times(-1,-1+\delta r_j^{p+2}).
\]
Let $\eta_j(x) $ be a smooth cut-off function satisfying 
\[
\mbox{supp}(\eta_j) \subset B_{r_j}(\bar x), \quad 0\le \eta_j \le 1, \quad \eta_j=1 \mbox{ in }B_{r_{j+1}}(\bar x), 
\]
\[
|D \eta_j(x,t)|^2  \le \frac{C(n)}{(r_j-r_{j+1})^2} \quad \mbox{in }B_R(\bar x). 
\]

Case 1: $p\ge 0$. Let us consider $n\ge 3$ first. By Theorem \ref{thm:caccipolli} and Theorem \ref{thm:weightedsobolev}, we have
\begin{equation}\label{eq:auxgiorgiinitial}
 \Big(\int_{\mathcal{Q}_{j,\delta}^+}  |\eta_j v|^{2\chi}\,\ud x \ud t \Big)^{\frac{1}{\chi}} \le C\left[ \frac{2^{2j}}{R^2} \|v\|^2_{L^2(\mathcal{Q}_{j,\delta}^+)} + (M+F_1)^2 |\mathcal{Q}_{j,\delta}^+\cap \{u>k_j\}|^{1-\frac1q}\right],
\end{equation}
where $v=(u-k_j)^+$. Let $A(k,r_j)= \{(x,t)\in \mathcal{Q}_{j,\delta}^+: u> k\}$. Then 
\[
 \Big(\int_{\mathcal{Q}_{j,\delta}^+}  |\eta_j v|^{2\chi}\,\ud x \ud t \Big)^{\frac{1}{\chi}} \ge (k_{j+1}- k_j)^2 |A(k_{j+1}, r_{j+1})|^{\frac{1}{\chi}},  
\]
and
\[
\int_{\mathcal{Q}_{j,\delta}^+}  v^{2} \,\ud x\ud t \le \mu^2 |A(k_{j}, r_{j})|. 
\]
If
\[
\mu\ge (M+F_1)R^{1-\frac{n+p+2}{2q}},
\]
then
\begin{align*}
|A(k_{j+1}, r_{j+1})| & \le  C\left[ \frac{2^{4j}}{R^2}|A(k_{j}, r_{j})| + \frac{2^{2j}(M+F_1)^2}{\mu^2} |A(k_{j}, r_{j})|^{1-\frac1q}\right]^\chi\\
&\le  C\left[ \frac{2^{4j}}{R^2}|A(k_{j}, r_{j})| + \frac{2^{2j}}{R^{2-\frac{n+p+2}{q}}} |A(k_{j}, r_{j})|^{1-\frac1q}\right]^\chi\\
&\le C\left[ \frac{16^{j}}{R^{2-\frac{n+p+2}{q}}} |A(k_{j}, r_{j})|^{1-\frac1q}\right]^\chi,
\end{align*}
where we used $|A(k_{j}, r_{j})| \le \delta|\mathcal{Q}_{j,\delta}^+|\le C\delta R^{n+p+2}$. Hence
\begin{equation}\label{eq:auxgiorgiinitial2}
\frac{|A(k_{j+1}, r_{j+1})| }{ |\mathcal{Q}_R^+|}\le C 16^{j\chi}\left( \frac{|A(k_{j}, r_{j})|}{|\mathcal{Q}_R^+|}\right)^{(1-\frac1q)\chi},
\end{equation}
where we used that $\chi=\frac{n+p+2}{n+p}$. Therefore, similarly to \eqref{eq:nonlineariteration} and \eqref{eq:nonlineariteration2}, there exists $\delta_0\in(0,1)$ such that if $\delta\le\delta_0$, then
\begin{equation}\label{eq:auxgiorgiinitial3}
\lim_{j\to \infty}\frac{|A(k_{j+1}, r_{j+1})|}{|\mathcal{Q}_R^+|}=0.
\end{equation}

Now, let us  consider $n=1, 2$. By Theorem \ref{thm:caccipolli} and Theorem \ref{thm:weightedsobolev}, \eqref{eq:auxgiorgi1} would become
\[
 \Big(\int_{\mathcal{Q}_R^+}  |\eta_j v|^{2\chi}\,\ud x \ud t \Big)^{\frac{1}{\chi}} \le C R^{\frac{p+2-n}{p+2}}\left[ \frac{2^{2j}}{R^2} \|v\|^2_{L^2(\mathcal{Q}_{r_{j}}^+)} + (M+F_1)^2 |\mathcal{Q}_{r_{j}}^+\cap \{u>k_j\}|^{1-\frac1q}\right].
\]
By using $\chi=\frac{p+2}{p+1}$, one will still obtain \eqref{eq:auxgiorgiinitial2} and \eqref{eq:auxgiorgiinitial3}. Then the left proof is the same as above.

Case 2: $-1<p<0$. Again, we consider $n\ge 3$ first. By Theorem \ref{thm:caccipolli} and Theorem \ref{thm:weightedsobolev2}, we have
\begin{equation}\label{eq:auxgiorgiinitial-2}
\begin{split}
& \Big(\int_{\mathcal{Q}_{j,\delta}^+}  |\eta_j v|^{2\chi}x_n^p\,\ud x \ud t \Big)^{\frac{1}{\chi}} \\
 &\le C\left[ \frac{2^{2j}}{R^2} \|v\|^2_{L^2(\mathcal{Q}_{j,\delta}^+,x_n^p\ud x\ud t)} + (M+F_1)^2 |\mathcal{Q}_{j,\delta}^+\cap \{u>k_j\}|_{\nu_p}^{1-\frac1q}\right],
 \end{split}
\end{equation}
where $v=(u-k_j)^+$. Then 
\[
 \Big(\int_{\mathcal{Q}_{j,\delta}^+}  |\eta_j v|^{2\chi}x_n^p\,\ud x \ud t \Big)^{\frac{1}{\chi}} \ge (k_{j+1}- k_j)^2 |A(k_{j+1}, r_{j+1})|_{\nu_p}^{\frac{1}{\chi}},  
\]
and
\[
\int_{\mathcal{Q}_{j,\delta}^+}  v^{2} x_n^p\,\ud x\ud t \le \mu^2 |A(k_{j}, r_{j})|_{\nu_p}. 
\]
If
\[
\mu\ge (M+F_1)R^{\frac{p+2}{2}-\frac{n+2p+2}{2q}},
\]
then
\begin{align*}
|A(k_{j+1}, r_{j+1})|_{\nu_p} & \le  C\left[ \frac{2^{4j}}{R^2}|A(k_{j}, r_{j})|_{\nu_p} + \frac{2^{2j}(M+F_1)^2}{\mu^2} |A(k_{j}, r_{j})|_{\nu_p}^{1-\frac1q}\right]^\chi\\
&\le  C\left[ \frac{2^{4j}}{R^2}|A(k_{j}, r_{j})|_{\nu_p} + \frac{2^{2j}}{R^{p+2-\frac{n+2p+2}{q}}} |A(k_{j}, r_{j})|_{\nu_p}^{1-\frac1q}\right]^\chi\\
&\le C\left[ \frac{16^{j}}{R^{p+2-\frac{n+2p+2}{q}}} |A(k_{j}, r_{j})|_{\nu_p}^{1-\frac1q}\right]^\chi,
\end{align*}
where we used $|A(k_{j}, r_{j})|_{\nu_p} \le \delta|\mathcal{Q}_{j,\delta}^+|_{\nu_p}\le C\delta R^{n+2p+2}$. Hence
\begin{equation}\label{eq:auxgiorgiinitial2-2}
\frac{|A(k_{j+1}, r_{j+1})|_{\nu_p} }{ |\mathcal{Q}_R^+|_{\nu_p}}\le C 16^{j\chi}\left( \frac{|A(k_{j}, r_{j})|_{\nu_p}}{|\mathcal{Q}_R^+|_{\nu_p}}\right)^{(1-\frac1q)\chi},
\end{equation}
where we used that $\chi=\frac{n+2p+2}{n+p}$. Therefore, there exists $\delta_0\in(0,1)$ such that if $\delta\le\delta_0$, then
\begin{equation}\label{eq:auxgiorgiinitial3-2}
\lim_{j\to \infty}\frac{|A(k_{j+1}, r_{j+1})|_{\nu_p}}{|\mathcal{Q}_R^+|_{\nu_p}}=0.
\end{equation}

Now, let us  consider $n=1, 2$. By Theorem \ref{thm:caccipolli} and Theorem \ref{thm:weightedsobolev}, \eqref{eq:auxgiorgi1} would become
\begin{align*}
& \Big(\int_{\mathcal{Q}_R^+}  |\eta_j v|^{2\chi}x_n^p\,\ud x \ud t \Big)^{\frac{1}{\chi}} \\
&\le C R^{\frac{p+4-n}{3}}\left[ \frac{2^{2j}}{R^2} \|v\|^2_{L^2(\mathcal{Q}_{r_{j}}^+,x_n^p\ud x\ud t)} + (M+F_1)^2 |\mathcal{Q}_{r_{j}}^+\cap \{u>k_j\}|_{\nu_p}^{1-\frac1q}\right].
\end{align*}
By using $\chi=\frac{3}{2}$, one will still obtain \eqref{eq:auxgiorgiinitial2-2} and \eqref{eq:auxgiorgiinitial3-2}. Then the left proof is the same as above.

In each case, we have that if $0<\delta\le\delta_0$, then
\[
\sup_{B_{R/2}(\bar x)\times(-1,-1+\delta(R/2)^{p+2})} u\le \frac{\mu(R)}{2}.
\]
Applying this estimate to $-u$, one have 
\[
\inf_{B_{R/2}(\bar x)\times(-1,-1+\delta(R/2)^{p+2})} u\ge \frac{\widetilde \mu(R)}{2}.
\]

Meanwhile, it follows from Lemma \ref{lem:smallonlargeset} and Lemma \ref{lem:decay-1} that there exists $\ell>0$ such that
either
\begin{equation*}
\mu\le  
\left\{
  \begin{aligned}
  &  2^{\ell} (M+F_1)R^{1-\frac{n+p+2}{2q}},  \quad  &\mbox{for }   p\ge 0, \\
  &  2^{\ell} (M+F_1)R^{\frac{p+2}{2}-\frac{n+2p+2}{2q}}, \quad   &\mbox{for }    -1<p<0,
\end{aligned}
\right.
\end{equation*}
or
\[
\sup_{B_{R/4}(\bar x)\times(-1+\delta(R/2)^{p+2},-1+(R/2)^{p+2}]} u\le \mu(R)-\frac{\mu(R)}{2^{\ell}}.  
\]
and either
\begin{equation*}
-\widetilde \mu\le  
\left\{
  \begin{aligned}
  &  2^{\ell} (M+F_1)R^{1-\frac{n+p+2}{2q}},  \quad  &\mbox{for }   p\ge 0, \\
  &  2^{\ell} (M+F_1)R^{\frac{p+2}{2}-\frac{n+2p+2}{2q}}, \quad   &\mbox{for }    -1<p<0,
\end{aligned}
\right.
\end{equation*}
or
\[
\inf_{B_{R/4}(\bar x)\times(-1+\delta(R/2)^{p+2},-1+(R/2)^{p+2}]} u\ge \widetilde\mu(R)-\frac{\widetilde\mu(R)}{2^{\ell}}.  
\]
In any case, we obtain 
\begin{equation*}
\omega(R/4)\le
\left\{
  \begin{aligned}
  &  (1-2^{\ell+1})\omega(R)+ 2^{\ell} (M+F_1)R^{1-\frac{n+p+2}{2q}},  \quad  &\mbox{for }   p\ge 0, \\
  &  (1-2^{\ell+1})\omega(R)+ 2^{\ell} (M+F_1)R^{\frac{p+2}{2}-\frac{n+2p+2}{2q}}, \quad   &\mbox{for }    -1<p<0.
\end{aligned}
\right.
\end{equation*}
By an iterative lemma, e.g. Lemma 3.4 in Han-Lin \cite{HL} (or Lemma B.2 in \cite{JX19}), there exist $\alpha$ and $C$, both of which depend only on $\lambda,\Lambda,n,p$ and $q$, such that
\[
\omega(R)\le C(M+F_1)R^{\alpha}\quad\,\forall R\in(0,1/4].
\]
The conclusion follows from the above and Theorem \ref{thm:localboundednessglobal}.
\end{proof}

It has been pointed by the referee that Theorem \ref{thm:holderboundarybottom} also follows from applying Theorem \ref{thm:holderontheboundary} to the  solution that is extended to be zero for $t<-1$.

Similar to the justifications of Theorem \ref{thm:uniformholdernearboundary} and Theorem \ref{thm:holderboundarybottom}, we also have
\begin{thm}\label{thm:uniformholderinitial}
Suppose $u\in C ([-1,0]; L^2(B_1^+,x_n^{p}\ud x)) \cap  L^2((-1,0];H_{0,L}^1(B_1^+))$ is a weak solution of \eqref{eq:linear-eq} with the partial boundary condition \eqref{eq:linear-eq-D} and the initial condition $u(\cdot,-1)=0$, where the coefficients of the equation satisfy \eqref{eq:rangep}, \eqref{eq:ellip2}, \eqref{eq:localbddnesscoeffcient} and \eqref{eq:localbddnessf} for some $q>\max(\frac{\chi}{\chi-1},\frac{n+p+2}{2}, \frac{n+2p+2}{p+2})$.   Then for every $\gamma>0$, there exist $\alpha>0$ and $C>0$, both of which depend only on  $\lambda,\Lambda,n,p,\gamma$ and $q$, such that for every $(x,-1), (y,s)\in B_{1/4}(e_n/2)\times[-1,-\frac34]$, there holds
\[
|u(x,-1)-u(y,s)|\le C(\|u\|_{L^\gamma({\mathcal{Q}}_1^+)}+F_1) (|x-y|+|s+1|)^\alpha,
\]
where $e_n=(0,\cdots,0,1)$.
\end{thm}

Together with Theorem \ref{thm:uniformholdernearboundary} and Theorem \ref{thm:uniformholderinitial}, using similar scaling arguments to those in the proof of Theorem \ref{thm:holdernearboundary}, we have
\begin{thm}\label{thm:uniformholdernearinitial}
Suppose $u\in C ([-1,0]; L^2(B_1^+,x_n^{p}\ud x)) \cap  L^2((-1,0];H_{0,L}^1(B_1^+))$ is a weak solution of \eqref{eq:linear-eq} with the partial boundary condition \eqref{eq:linear-eq-D} and the initial condition $u(\cdot,-1)=0$, where the coefficients of the equation satisfy \eqref{eq:rangep}, \eqref{eq:ellip2}, \eqref{eq:localbddnesscoeffcient} and \eqref{eq:localbddnessf} for some $q>\max(\frac{\chi}{\chi-1},\frac{n+p+2}{2}, \frac{n+2p+2}{p+2})$.  Then for every $\gamma>0$, there exist $\alpha>0$ and $C>0$, both of which depend only on  $\lambda,\Lambda,n,p,\gamma$ and $q$, such that for every $(x,t), (y,s)\in B_{1/4}(e_n/2)\times[-1,0]$, there holds
\[
|u(x,t)-u(y,s)|\le C(\|u\|_{L^\gamma({\mathcal{Q}}_1^+)}+F_1) (|x-y|+|t-s|)^\alpha,
\]
where $e_n=(0,\cdots,0,1)$.
\end{thm}
\begin{proof}
The proof is in the same spirit as that of Theorem \ref{thm:holdernearboundary}. We omit the details, and one can also refer to the proof of Theorem \ref{thm:uniformholderglobal} in the below.
\end{proof}

Finally, we have the space-time global H\"older estimate:
\begin{thm}\label{thm:uniformholderglobal}
Suppose $u\in C ([-1,0]; L^2(B_1^+,x_n^{p}\ud x)) \cap  L^2((-1,0];H_{0,L}^1(B_1^+))$ is a weak solution of \eqref{eq:linear-eq} with the partial boundary condition \eqref{eq:linear-eq-D} and the initial condition $u(\cdot,-1)=0$, where the coefficients of the equation satisfy \eqref{eq:rangep}, \eqref{eq:ellip2}, \eqref{eq:localbddnesscoeffcient} and \eqref{eq:localbddnessf} for some $q>\max(\frac{\chi}{\chi-1},\frac{n+p+2}{2}, \frac{n+2p+2}{p+2})$.   Then for every $\gamma>0$, there exist $\alpha>0$ and $C>0$, both of which depend only on  $\lambda,\Lambda,n,p,\gamma$ and $q$, such that for every $(x,t), (y,s)\in B_{1/2}^+\times[-1,0]$, there holds
\[
|u(x,t)-u(y,s)|\le C(\|u\|_{L^\gamma({\mathcal{Q}}_1^+)}+F_1) (|x-y|+|t-s|)^\alpha.
\]
\end{thm}
\begin{proof}
By Theorem \ref{thm:localboundednessglobal} and normalization, we assume $\sup_{B_{3/4}^+\times[-1,0]}|u|+F_1=1$. For any $\bar x=(0,\bar x_n )\in B_{1/4}^+$ and $\bar t\in (-1,0]$, we let $R:=\max(\bar x_n, (\bar t+1)^{\frac{1}{p+2}})>0$, and rescale the solution and the coefficients as in \eqref{eq:rescaledequationcoefficients} with $x_0=0$.  Then we have \eqref{eq:rescaledequation} in $\mathcal{Q}^+_{1/R}$. Also, \eqref{eq:rescaledequationcoefficients1} and \eqref{eq:rescaledequationcoefficients2} hold with $Q_1^+$ replaced by $\widetilde Q^+=B_2^+\times(-R^{-p-2},-R^{-p-2}+1])$.

Case 1: $R=\bar x_n$.

Consider $s\in (-1,\bar t]$. If $|\bar t-s|\le R^{2p+4}$, then by Theorem \ref{thm:uniformholdernearinitial},  we have
\[
|u(\bar x,\bar t)-u(\bar x, s)|=|\tilde u(e_n,\bar t/R^{p+2})-\tilde u(e_n,s/R^{p+2})|\le C |(\bar t-s)/R^{p+2}|^{\alpha/2}\le C |\bar t-s|^{\alpha/4}.
\]
If $|\bar t-s|\ge R^{2p+4}$, then we have
\begin{align*}
|u(\bar x,\bar t)-u(\bar x, s)|& \le |u(\bar x,\bar t)-u(\bar x,-1)|+|u(\bar x,s)-u(\bar x, -1)|\\
& = |\tilde u(e_n,\bar t/R^{p+2})-\tilde u(e_n,-1/R^{p+2})|+|\tilde u(e_n,s/R^{p+2})-\tilde u(e_n, -1/R^{p+2})|\\
& \le C|t+1|^{\alpha}\\
&\le CR^{(p+2)\alpha}\\
&\le C |\bar t-s|^{\frac{\alpha}{2}}.
\end{align*}
This shows that $u$ is H\"older continuous in the time variable.

Consider $\tilde x=(\tilde x',\tilde x_n)\in B_{1/2}^+$ such that $\tilde x_n\le \bar x_n$. If $\tilde x\in B_{R^2}(\bar x)$, then by Theorem \ref{thm:uniformholdernearinitial},  we have
\[
|u(\bar x,\bar t)-u(\tilde x, \bar t)|=|\tilde u(e_n,\bar t/R^{p+2})-\tilde u(\tilde x/R,\bar t/R^{p+2})|\le C ||\tilde x-\bar x|/R|^{\beta}\le C |\tilde x-\bar x|^{\beta/2}.
\]
 If $\tilde x\not\in B_{R^2}(\bar x)$, then we have 
\begin{align*}
|u(\bar x,\bar t)-u(\tilde x, \bar t)|& \le |u(\bar x,\bar t)-u(0,-1)|+|u(0,-1)-u(\tilde x',0,-1)|+|u(\tilde x',0,-1)-u(\tilde x, \bar t)|\\
& \le C(R^{\alpha}+|\bar t+1|^\alpha)\\
&\le C(R^{\alpha}+|\bar t+1|^\alpha)\\
&\le C |\bar x-\tilde x|^{\frac{\alpha}{2}},
\end{align*}
where we used Theorem \ref{thm:holderboundarybottom} in the second inequality. This shows that $w$ is H\"older continuous in the spatial variables.

Case 2: $R= (\bar t+1)^{\frac{1}{p+2}}$.

Consider $s\in (-1,\bar t]$. If $|\bar t-s|\le R^{2p+4}$, then by Theorem \ref{thm:holdernearboundary},  we have
\[
|u(\bar x,\bar t)-u(\bar x, s)|=|\tilde u(\bar x/R,\bar t/R^{p+2})-\tilde u(\bar x/R,s/R^{p+2})|\le C |(\bar t-s)/R^{p+2}|^{\alpha/2}\le C |\bar t-s|^{\alpha/4}.
\]
If $|\bar t-s|\ge R^{2p+4}$, then by Theorem \ref{thm:holderboundarybottom}, we have
\begin{align*}
|u(\bar x,\bar t)-u(\bar x, s)|& \le |u(\bar x,\bar t)-u(0,-1)|+|u(0,-1)-u(\bar x, s)|\\
& \le CR^{\alpha}\\
&\le C |\bar t-s|^{\frac{\alpha}{2(p+2)}}.
\end{align*}
This shows that $u$ is H\"older continuous in the time variable.

Consider $\tilde x=(\tilde x',\tilde x_n)\in B_{1/2}^+$ such that $\tilde x_n\le \bar x_n$. If $\tilde x\in B_{R^2}(\bar x)$, then by Theorem \ref{thm:uniformholdernearboundary},  we have
\[
|u(\bar x,\bar t)-u(\tilde x, \bar t)|=|\tilde u(e_n,\bar t/R^{p+2})-\tilde u(\tilde x/R,\bar t/R^{p+2})|\le C ||\tilde x-\bar x|/R|^{\beta}\le C |\tilde x-\bar x|^{\beta/2}.
\]
 If $\tilde x\not\in B_{R^2}(\bar x)$, then we have 
\begin{align*}
|u(\bar x,\bar t)-u(\tilde x, \bar t)|& \le |u(\bar x,\bar t)-u(0,-1)|+|u(0,-1)-u(\tilde x',0,-1)|+|u(\tilde x',0,-1)-u(\tilde x, \bar t)|\\
& \le C(R^{\alpha}+|\bar t+1|^\alpha)\\
&\le C(R^{\alpha}+|\bar t+1|^\alpha)\\
&\le C |\bar x-\tilde x|^{\frac{\alpha}{2}},
\end{align*}
where we used Theorem \ref{thm:holderboundarybottom} in the second inequality. This shows that $u$ is H\"older continuous in the spatial variables.

Together with Theorem \ref{thm:localboundednessglobal}, we finish the proof of this theorem.
\end{proof}

\subsection{The Cauchy-Dirichlet problem}
In the end, let us go back to the Cauchy-Dirichlet problem in general domains mentioned at the beginning:
\begin{equation} \label{eq:finalgeneral}
\begin{split}
a\omega^{p} \pa_t u-D_j(a_{ij} D_i u+d_j u)+b_iD_i u+\omega^pcu+c_0u&=\omega^pf+f_0 -D_if_i\quad \mbox{in }\Omega \times(-1,0],\\
u&=0\quad \mbox{on }\partial_{pa}(\Omega \times(-1,0]),
\end{split}
\end{equation}
where $\Omega\subset\R^n$, $n\ge 1$, is a smooth bounded open set, and $\omega$ is a smooth function in $\overline\Omega$ comparable to the distance function $d(x):=\dist(x,\partial\Omega)$, that is, $0<\inf_{\Omega}\frac{\omega}{d}\le \sup_{\Omega}\frac{\omega}{d}<\infty$, and $p>-1$ is a constant.

Suppose there exist $0<\lambda\le\Lambda<\infty$ such that 
\be \label{eq:ellip2final}
\lda\le a(x,t)\le \Lda,  \quad \lda |\xi|^2 \le \sum_{i,j=1}^na_{ij}(x,t)\xi_i\xi_j\le \Lda |\xi|^2, \quad\forall\ (x,t)\in \Omega \times(-1,0],\ \forall\ \xi\in\R^n,
\ee
and
\begin{align}
\Big\||\pa_t a|+|c|\Big\|_{L^q(\Omega \times(-1,0],x_n^p\ud x\ud t)}+\Big\|\sum_{j=1}^n(b_j^2+d_j^2)+|c_0|\Big\|_{L^q(\Omega \times(-1,0])}&\le\Lambda, \label{eq:localbddnesscoeffcientfinal}\\
F_2:=\|f\|_{L^{\frac{2q\chi}{q\chi+\chi-q}}(\Omega \times(-1,0],x_n^p\ud x\ud t)}+\|f_0\|_{L^{\frac{2q\chi}{q\chi+\chi-q}}(\Omega \times(-1,0])}+ \sum_{j=1}^n\|f_j\|_{L^{2q}(\Omega \times(-1,0])}&<\infty\label{eq:localbddnessffinal}
\end{align}
for some $q>\max(\frac{\chi}{\chi-1},\frac{n+p+2}{2}, \frac{n+2p+2}{p+2})$, where $\chi>1$ is the constant in Theorem \ref{thm:weightedsobolev} or Theorem \ref{thm:weightedsobolev2} depending on the value of $p$.

We say that $u$ is a weak solution of \eqref{eq:finalgeneral} if $u\in C ((-1,0]; L^2(\Omega,\omega^{p}\ud x)) \cap  L^2((-1,0];H_{0}^1(\Omega)) $, $u(\cdot,-1)\equiv 0$, and satisfies
\begin{equation}\label{eq:definitionweaksolutionfinal}
\begin{split}
&\int_{\Omega}a(x,s) \omega(x)^{p}  u(x,s) \varphi(x,s)\,\ud x-\int_{-1}^s\int_{\Omega} \omega^{p}(\varphi\partial_t a+a\partial_t \varphi)u\,\ud x\ud t\\
&\quad+ \int_{-1}^s\int_{\Omega} \big(a_{ij}D_iuD_j\varphi+d_juD_j\varphi+b_jD_ju\varphi+c \omega^p u \varphi+c_0u\varphi\big)\,\ud x\ud t\\
&=\int_{-1}^s\int_{\Omega} (\omega^p f\varphi+f_0\varphi+f_jD_j\varphi)\,\ud x\ud t\quad\mbox{a.e. }s\in (-1,0]
\end{split}
\end{equation}
for every $\varphi\in \{g\in L^2(\Omega \times(-1,0]): \partial_tg\in L^2(\Omega \times(-1,0],\omega^{p}\ud x\ud t), D_i g \in L^2(\Omega \times(-1,0]), i=1,\cdots,n, g=0\mbox{ on } \partial\Omega\times (-1,0]\}.$ 

\begin{thm}\label{thm:uniformholderglobalfinal}
Suppose  $p>-1$, \eqref{eq:ellip2final}, \eqref{eq:localbddnesscoeffcientfinal} and \eqref{eq:localbddnessffinal} hold for some $q>\max(\frac{\chi}{\chi-1},\frac{n+p+2}{2}, \frac{n+2p+2}{p+2})$.   Then there exists a unique weak solution $u\in C ((-1,0]; L^2(\Omega,\omega^{p}\ud x)) \cap  L^2((-1,0];H_{0}^1(\Omega)) $ of \eqref{eq:finalgeneral}. Furthermore,  for every $\gamma>0$, there exist $\alpha>0$ and $C>0$, both of which depend only on $\lambda,\Lambda,n,\Omega, p,\gamma$ and $q$, such that for every $(x,t), (y,s)\in \Omega\times[-1,0]$, there holds
\[
|u(x,t)-u(y,s)|\le C(\|u\|_{L^\gamma(\Omega\times[-1,0])}+F_2) (|x-y|+|t-s|)^\alpha.
\]
\end{thm}
\begin{proof}
The H\"older estimate of the weak solution follows from Theorem \ref{thm:uniformholderglobal}, Theorem \ref{thm:uniformholdernearinitial}, the flattening boundary technique and a covering argument.

The uniqueness of the weak solution follows from a similar energy estimate to that in Theorem \ref{thm:uniquenessofweaksolution}.

The existence of  weak solutions follows by a similar argument to the proof of Theorem \ref{thm:existenceofweaksolution}. Here, we do not need to assume $a$ to be continuous, since the approximating solutions in the proof of Theorem \ref{thm:existenceofweaksolution} under the assumption of this theorem will be uniformly H\"older continuous up to the boundary. The argument there will go through without the assumption of the continuity of $a$. We leave the details to the readers. 
\end{proof}

\bigskip
\newpage

\noindent T. Jin

\noindent Department of Mathematics, The Hong Kong University of Science and Technology\\
Clear Water Bay, Kowloon, Hong Kong\\[1mm]
Email: \textsf{tianlingjin@ust.hk}

\medskip

\noindent J. Xiong

\noindent School of Mathematical Sciences, Laboratory of Mathematics and Complex Systems, MOE\\ Beijing Normal University, 
Beijing 100875, China\\[1mm]
Email: \textsf{jx@bnu.edu.cn}

\end{document}